\documentclass[reqno]{amsart}
\usepackage[utf8]{inputenc}

\usepackage{amsfonts,color,newclude}
\usepackage{amsthm}
\usepackage{amssymb,transparent}
\usepackage{amsmath}
\usepackage{amscd}
\usepackage{thmtools}
\usepackage{thm-restate}
\usepackage{enumitem}

\usepackage{graphicx}
\usepackage{subfiles}
\usepackage{tikz}
\usepackage{pgffor} 
\usetikzlibrary{arrows, positioning, calc}
\usepackage{tikz-cd}
\usepackage{MnSymbol}

\usepackage[style=alphabetic]{biblatex}
\usepackage[colorlinks=true, pdfstartview=FitV, linkcolor=blue, citecolor=blue, urlcolor=blue]{hyperref}
\usepackage[capitalize,nameinlink,noabbrev]{cleveref}

\renewcommand{\bar}{\overline}
\renewcommand{\hat}{\widehat}
\renewcommand{\tilde}{\widetilde}
\newcommand{\boundary}{\partial}

\DeclareMathOperator{\can}{Cancel} 

\newcommand{\inv}{^{-1}}
\renewcommand{\epsilon}{\varepsilon}

\newcommand{\Stab}{\operatorname{Stab}}

\newcommand{\oskel}{^{(1)}}
\newcommand{\EX}{\mc{E}X}
\DeclareMathOperator{\length}{length}

\newcommand{\rightQ}[2]{\left.\raisebox{.2em}{$#1$}\middle/\raisebox{-.2em}{$#2$}\right.}
\newcommand{\leftQ}[2]{\left.\raisebox{-.2em}{$#2$}\middle\backslash\raisebox{.2em}{$#1$}\right.}

\newcommand{\R}{\mathbb{R}}

\newcommand{\Z}{\mathbb{Z}}

\newcommand{\eps}{\varepsilon}

\newcommand{\mc}[1]{\mathcal{#1}}
\newcommand{\EG}{\mathcal{E}X}
\newcommand{\EGbal}{\EG_{\text{bal}}}
\newcommand{\EXbal}{\EX_{\text{bal}}}
\newcommand{\vertspace}{Z}

\declaretheorem[style = plain, name=Theorem, within=section]{theorem}
\declaretheorem[style = plain, name=Proposition,sibling=theorem]{prop}
\declaretheorem[style = plain, name=Proposition,sibling=theorem]{proposition}
\declaretheorem[style = plain, name=Lemma,sibling=theorem]{lemma}
\declaretheorem[style = plain, name=Corollary,sibling=theorem]{cor}
\declaretheorem[style = plain, name=Observation,sibling=theorem]{obs}

\theoremstyle{definition}
\newtheorem{defn}[theorem]{Definition}

\theoremstyle{remark}
\newtheorem{remark}[theorem]{Remark}
\newtheorem{example}[theorem]{Example}

\newtheorem{question}[theorem]{Question}
\newtheorem{notation}[theorem]{Notation}

\title{Random Quotients of Free Products}
\date{\today}

\author[Einstein]{Eduard Einstein\textdagger}
\address{Department of Mathematics and Statistics, Swarthmore College, 500 College Ave, Swarthmore, PA 19081, USA.}

\author[Suraj Krishna]{Suraj Krishna M S}
\address{Department of Mathematics, Ashoka University, Haryana 131029, India.}
\email{suraj.meda@ashoka.edu.in}

\author[Montee]{MurphyKate Montee}
\address{Department of Mathematics and Statistics, Carleton College, 1 College St, Northfield, MN 55057, USA \& Erwin Schr\"odinger International Insitute for Mathematics and Physics, Universit\"at Wien, 1090 Wien, Austria}
\email{mmontee@carleton.edu}

\author[Ng]{Thomas Ng}
\address{Department of Mathematics, Brandeis University, 415 South Street, Waltham, PA 02453, USA.}
\email{thomas.ng.math@gmail.com}

\author[Steenbock]{Markus Steenbock}
\address{Fakult\"at f\"ur Mathematik, Universit\"at Wien, 1090 Wien, Austria  \& Erwin Schr\"odinger International Insitute for Mathematics and Physics, Universit\"at Wien, 1090 Wien, Austria}
\email{markus.steenbock@univie.ac.at}

\begin{document}

\begin{abstract}
    We introduce a density model for random quotients of a free product of finitely generated groups.
    We prove that a random quotient in this model has the following properties with overwhelming probability: if the density is below $1/2$, the free factors embed into the random quotient and the random quotient is hyperbolic relative to the free factors. Further, there is a phase transition at $1/2$, with the random quotient being a finite group above this density.
    If the density is below $1/6$, the random quotient is cubulated relative to the free factors. Moreover, if the free factors are cubulated, then so is the random quotient. 
\end{abstract}

\maketitle

\section{Introduction}
Gromov's density model of random groups is an oft-used tool for studying `generic' properties of groups. 
 The model is defined as follows: Let $n\geq 2, \ell>0, 0<d<1$. A \emph{Gromov random group} (see \cite{gromov_asymptotic}) is given by a presentation $\langle S | \mc{R}\rangle$, where $|S| = n$ and $\mc{R}$ is a collection of $(2n-1)^{d\ell}$ cyclically reduced words on the alphabet $S$ of length $\ell$, chosen uniformly at random with replacement. The value $d$ is called the \emph{density}. Traditionally, we fix a density $d$ and a number of generators $n$, and say that a property $P$ holds \emph{with overwhelming probability} if
the probability of a random group satisfying  $P$ tends to $1$ as $\ell \to \infty$.

When $d<1/2$, with overwhelming probability a Gromov random group is torsion-free non-elementary  hyperbolic; on the other hand, when $d>1/2$, with overwhelming probability a random group is either trivial or $\mathbb{Z}/2\mathbb{Z}$ \cite{gromov_asymptotic, ollivier_wise}. Other properties of random groups also exhibit a phase transition phenomenon, that is, there is a threshold density $d_0$ above and below which some property is either satisfied or not satisfied with overwhelming probability. These properties include small cancellation conditions \cite{oll_somesmall, Tsai_density_2022,Tsai-phase-transition}, satisfying Greendlinger's lemma \cite{oll_somesmall}, a Freiheitssatz-type property \cite{tsai_freiheit_2023}, and Property (T) (folklore).  For the last property, the threshold density is currently unknown, though it is known to be at most 1/3 \cite{Zuk2003, KK} and no less than 1/4 \cite{Ashcroft_1/4}. It is unknown if there is a threshold density for cubulation (that is, acting geometrically on a CAT(0) cube complex). Ollivier--Wise showed in \cite{ollivier_wise} that for $d<1/6$, Gromov random groups are, with overwhelming probability, cubulated, and Mackay--Przytycki and Montee \cite{MP, montee_314} extended these results to cocompact, but not necessarily proper, actions on CAT(0) cube complexes for $d<3/14.$

\subsection{The free product density model}

In this paper, we initiate the study of random quotients of free products of groups 
and obtain a few analogous results. 
We do so by defining 
a density model for random quotients that uses the action of a free product on a suitable Bass--Serre tree $T$, see \cite{serre_arbres_1977} or \cite[Chapter III.$\Gamma$, Section 6.3]{bridson_haefliger_1999}. 
This is directly inspired by Gromov's model: words of length $\ell$ in a free group correspond to the finite set of elements that have translation length $\ell$ on its Cayley graph. 
Note that, in our case, $T$ is not locally finite whenever one of the free factors is infinite, so 
the set of all loxodromic elements of a fixed translation length is infinite. 
For this reason, we choose a density $d$ random set of relators from a natural finite subset of loxodromic elements of fixed translation length.

\begin{defn}\label{def: the model}
    Let $n\geq 2, m\geq 1, \ell>0, 0<d<1$. Let $\mc{G} = (G_1, \dots, G_n)$ be an $n$-tuple of nontrivial finitely generated groups, that we refer to as the \emph{factor groups}. Moreover, we suppose that if $n=2$, then $G_1$ or $G_2$ is not isomorphic to $\mathbb Z/2\mathbb Z$, so that $G_1*G_2$ is not virtually cyclic.  
    Let $S_i$ denote a finite generating set for $G_i$. 
     Let $B_i(m)\subseteq G_i$ be the set of nontrivial elements in the ball of radius $m$ 
     in the word metric with respect to $S_i$.  
      Let $\mathcal{S}_\ell$ be the set of words $w=s_1s_2\cdots s_{\ell}$ such that each \emph{syllable} $s_j \in \bigcup B_i(m)$ and such that no two consecutive syllables $s_j$ and $s_{j+1}$ are in the same  factor group. In addition, we assume that $s_1$ and $s_{\ell}$ are not in the same factor group. Let $\mathcal{R}\subset \mathcal{S}_\ell$ 
    be a subset chosen by selecting $|\mc{S}_\ell|^d$ words from $\mc{S}_\ell$ uniformly at random, with replacement. 
     Then the group $G = G_1*\dots*G_n / \llangle \mathcal{R}\rrangle$
    is a random group in the \emph{free product density model} of random groups on $G_1, \dots, G_n$. We write $G \sim \mathcal{FPD}(\mc{G}; d, m, \ell)$
    to indicate that $G$ is a random group in this model. We say that a group $G \sim \mathcal{FPD}(\mc{G}; d, m, \ell)$ satisfies a property $P$ with \emph{overwhelming probability} if $\lim_{\ell \to \infty} \mathbb{P}(G\text{ satisfies $P$}) = 1.$
\end{defn}

Note that $\mathcal S_\ell$ is a finite set of elements of $G_1*\cdots *G_n$ of free product length $=\ell$ represented in normal form.  
For readers familiar with the notion of \emph{angle} in this context, note that the set $\mc{S}_\ell$ consists of the elements of translation length $\ell$ in the Bass--Serre tree $T$ that subtend an angle at most $m$ at each turn. We note that the model depends on the choice of a finite generating set for each group $G_i$. 

As the hyperbolic geometry of a free group is reflected by its action on its Cayley graph, so too is the relative hyperbolic geometry of a free product reflected by its action on the Bass--Serre tree.
From this point of view, 
our model is
a natural generalization of the Gromov model to quotients of free products. 
The use of the Bass--Serre tree allows us to push many of the methods that have been used to study Gromov random groups to this setting. 

Our first theorem below shows that there is a threshold density ($\frac{1}{2}$) below which the factor groups of the free product embed into $G$, and $G$ itself is non-elementary relatively hyperbolic with respect to the factor groups.

\begin{theorem}\label{thm: rel hyp at d half}
Let $G \sim \mathcal{FPD}(\mc{G}; d, m, \ell)$. \begin{enumerate}
    \item If $d < \frac{1}{2}$, then, with overwhelming probability, the groups $G_1$, $G_2$, $\ldots$, $G_n$ embed into $G$ and $G$ is hyperbolic relative to $\{G_1, \dots, G_n\}$.\label{I: rel hyp d<1/2}
 \item  If $d>\frac{1}{2}$ and $n\geq 3$, then, with overwhelming probability, $G$ has order at most $2$. If $d>1/2$ and $n=2$, with overwhelming probability $G$ is equal to a finite dihedral group or the free product $G_1*\dots *G_n$. \label{I: big density} 
\end{enumerate}
\end{theorem}

\begin{remark}
\label{R:twoFactors_noRelations}
     If $n=2$, then $\mathcal S_{\ell}$ is empty for all odd $\ell$. Therefore $G = G_1 * G_2$ in this case.   
\end{remark}

When each free factor is a free group, our model gives a random quotient of a free group. However, our quotients behave differently than Gromov random quotients. 

\begin{example}\label{E: factors}
    Let us consider the free product $G_k = \mathbb{F}_k*\mathbb{Z}$ of a non-abelian free group on $k$ generators and an infinite cyclic group. We may consider $G_k \cong F_{k+1}$ and ask how random quotients of $G_k$ taken in the free product density model relate to random quotients in the Gromov density model. By \Cref{thm: rel hyp at d half}, $F_k$ embeds in the random free product quotient of $G_k$ for all values of $k$, whenever $d<1/2$. In the Gromov density model, 
   however, Tsai shows in \cite[Theorem 1]{tsai_freiheit_2023} that there exists a function $\epsilon:\mathbb{N} \to (0, \frac{1}{2}]$ which tends quickly to $0$ as $k \to \infty$ so that: if $d>\epsilon(k)$ then the factor group $\mathbb{F}_k$ does not embed into a random quotient, and if $d<\epsilon(k)$ then $\mathbb F_k$ embeds. Nevertheless, Gromov random groups at density $d<1/2$ are hyperbolic (see \cite{gromov_asymptotic, oll_somesmall}), and thus hyperbolic relative to the image of $\mathbb{F}_k$. 
\end{example}

 \subsection{Actions on CAT(0) cube complexes}
We also extend the cubulation result of \cite{ollivier_wise} for Gromov random groups to the free product density model. A group is \emph{cubulated} if it acts properly and cocompactly on a CAT(0) cube complex. In a similar vein, a group is \emph{relatively cubulated} if it admits a relatively geometric action on a CAT(0) cube complex (see \Cref{def: relatively geometric}).

\begin{theorem}\label{thm: rel cubulation at d sixth}
Let $G \sim \mathcal{FPD}(\mc{G}; d, m, \ell)$. If $d < \frac{1}{6}$, then $G$ is, with overwhelming probability, cubulated relative to $\{G_1, \dots,G_n\}$.
If moreover each $G_i$ is a cubulated group, then, with overwhelming probability, $G$ is cubulated.
\end{theorem}

We actually prove a generalization of \Cref{thm: rel cubulation at d sixth}, see \Cref{T: endgame}. This theorem applies to a free product of relatively cubulated groups. Each of these groups is cubulated relative to a set of subgroups called the peripheral subgroups. The theorem then asserts a relative cubulation for the random quotient relative to the collection of the peripheral subgroups of the factor groups. The first part of \Cref{thm: rel cubulation at d sixth} follows from this, and the second part is proven in \Cref{S: geometric cubulation}. 

In the case of $d<1/12$, \Cref{thm: cancel implies planar} implies that the random quotient $G$ satisfies, with overwhelming probability, the $C'(1/6)$-condition
over the free product. 
In this case \cref{thm: rel cubulation at d sixth}  follows from  \cite{martin_steenbock},  see also   \cite{kasia_cubulating_2022} in the case that $d<1/40$, and \cite{EinsteinNg}. The idea of the proof is to extend these papers and \cite{ollivier_wise} to the setting of the free product density model. 

\begin{remark}
    If a group $G$ is cubulated relative to peripheral subgroups $G_1,\ldots,G_n$, this can have powerful consequences even if the peripherals are not themselves cubulated.  For instance, if all the peripheral subgroups $G_i$ are residually finite, then $G$ itself is residually finite and all its full relatively quasiconvex subgroups are separable \cite[Corollary 1.7]{einstein_relatively_2022}, see also \cite[Theorem 4.7]{GMSpecializing}. 
\end{remark}

\begin{remark}
        As a consequence of \cref{thm: rel cubulation at d sixth}, we obtain a partial answer to  \cite[Problem 7.2]{futer_wise}. Let $G_1$ and $G_2$ act properly and cocompactly on CAT(0) cube complexes $X_1, X_2$, respectively. In this case, let us view the free product $G=G_1*G_2$ as the fundamental group of the space $X_*$, which is composed of $X_1,X_2$ and a segment $s$ whose endpoints are glued to $X_1$ and $X_2$, respectively.  Then $G$ is cubulated by \cref{thm: rel cubulation at d sixth}, and the density of cubulation given by \Cref{thm: rel cubulation at d sixth} does not depend on the length of $s$. 
        Note, however, that \cite[Problem 7.2]{futer_wise} does not specify a model of a random quotient of a free product. It also remains open whether the density of \Cref{thm: rel cubulation at d sixth} is optimal. In fact, to our knowledge the only model for random groups with a known sharp optimal upper bound for cubulation is the 6-gonal model for random groups \cite{odr_bent_walls}.
\end{remark}

We leave open the following questions.    

\begin{question}
Let $G \sim \mathcal{FPD}(\mc{G}; d, m, \ell)$. If $d < \frac{3}{14}$, with overwhelming probability, does $G$ act cocompactly on a CAT(0) cube complex? If $d<1/4$, does it act with unbounded orbits on a CAT(0) cube complex?
\end{question}
Having unbounded orbits on a CAT(0) cube complex and Property (T) are mutually exclusive group properties. In the case that $d < \frac{1}{6}$, \Cref{thm: rel cubulation at d sixth} implies that there is such an action on a CAT(0) cube complex, so $G \sim \mathcal{FPD}(\mc{G}; d, m, \ell)$ does not have Property (T) regardless of the choice of factor groups.
However, in the free product density model the following is open. 
\begin{question}

\label{Q:propT}
Let $G \sim \mathcal{FPD}(\mc{G}; d, m, \ell)$. If $d > \frac{1}{3}$, does the random group $G$ have Property (T) with overwhelming probability? 
\end{question} 

These questions are motivated by the best-known densities in Gromov's model for random groups (see \cite{montee_314}, \cite{Ashcroft_1/4}, \cite{Zuk2003}).

A strategy to answer \Cref{Q:propT} in the affirmative is the following.
Let $G_1=\langle X_1, R_1\rangle$ and $G_2=\langle X_2, R_2 \rangle$, and let $G=\langle X_1\cup X_2, R_1\cup R_2 \cup R\rangle$. 
Note that $G'=\langle X_1\cup X_2, R\rangle$ is a random quotient in the free product density model of two free groups.
Moreover, if $G'=\langle X_1\cup X_2, R\rangle$ has Property (T), then so does $G$. 
Thus, \Cref{Q:propT} can be reduced to whether the group $G'$ has Property (T). 

It is also unclear how the group properties of groups in the free product density model compare to those in Gromov's model for random groups in the case that each factor group is infinite cyclic. The referee suggested the following question. 
\begin{question}\label{Q:referee}
    If we consider a $d$-random quotient, with any $m\geq 1$, of the free product of $k\geq 2$ infinite cyclic groups $\mathbb{Z}*\mathbb{Z}*\cdots *\mathbb Z$, it does not coincide with the Gromov density model with $k$ generators at density $d$. Are there any group properties that differ?
\end{question}

\subsection{Outline}

In \Cref{sec: model spaces}, we provide constructions of model spaces $X_{\mc{R}}$ and $X_\mc{R}(\mc{Z})$ for random quotients of free products that we use throughout the paper. The first space, $X_{\mc{R}}$, witnesses the geometry relative to the free factors while  $X_\mc{R}(\mc{Z})$ is often quasi-isometric to the group itself. The main results of \Cref{sec: rel_isopermietry} are \cref{thm: cancel implies planar}, a local isoperimetric inequality for $X_{\mc R}$, and \cref{lem: greendlinger}, a non-planar version of Greendlinger's Lemma. 
In \Cref{sec: global isoperimetry}, we use \Cref{thm: cancel implies planar} to prove a global isoperimetric inequality, \Cref{thm: local-to-global}. We then examine how the geometry of $X_{\mc{R}}$ reflects the relative hyperbolicity of a random quotient relative to the factors when $d<\frac12$, and prove \Cref{thm: rel hyp at d half}. 
In \Cref{sec: rel_cubulation}, we establish results for density $d<\frac16$ about the geometry of codimension--$1$ subspaces of our model spaces called hyperstructures.  

These results are used in \Cref{sec: applying EN criterion} and \Cref{S: geometric cubulation} to prove \Cref{thm: rel cubulation at d sixth}. 
\Cref{T: endgame} is the main result of \Cref{sec: applying EN criterion}, where we prove that random quotients are relatively geometric cubulated with overwhelming probability at $d<\frac16$ when the factors are (possibly trivially) relatively geometrically cubulated. In addition to the results from \Cref{sec: rel_cubulation}, the proof of \Cref{T: endgame} uses a relative cubulation criterion from \cite{EinsteinNg}. 
Finally, in \Cref{S: geometric cubulation}, we prove that if the factor groups are properly and cocompactly cubulated, then a random quotient at density $d<\frac16$ will be properly and cocompactly cubulated with overwhelming probability.  For this, we use a boundary criterion of Bergeron--Wise \cite[Theorem 5.1]{BergeronWise} for properness of the action, and a result of Hruska--Wise \cite[Theorem 7.12]{hruska_wise_2014} to prove cocompactness.

\subsection{Acknowledgments} 

This research was funded in whole or in part by the Austrian Science Fund (FWF) [10.55776/P35079] and NSF grant number DMS-2317001. For open access purposes, the authors have applied a CC BY public copyright license to any author-accepted manuscript version arising from this submission.  
Montee and Steenbock acknowledge support by the Erwin Schr\"odinger International Institute for Mathematics and Physics in the framework of its Research in Teams Program in 2024. 
Suraj was supported by an annual research grant of Ashoka University, and ISF grant 1226/19 at the Technion, and gratefully acknowledges the hospitality of Carleton College, Institut Henri Poincaré, Institut des Hautes Études Scientifiques, the University of Vienna, and the International Centre for Theoretical Sciences (Code: ICTS/GIG2024/07) for visits during various stages of this work. 
Ng was partially supported by 
ISF grant 660/20 and 
at the Technion by a Zuckerman Fellowship and gratefully acknowledges the hospitality of Carleton College. Finally, the authors acknowledge support by LabEx CARMIN, ANR-10-LABX-59-01 of the Institut Henri Poincar\'e (UAR 839 CNRS-Sorbonne Universit\'e) during the program \emph{Groups Acting on Fractals, Hyperbolicity and Self-Similarity}. The authors thank Xingyi Zhang for assistance in creating figures.

Finally, we wish to thank the referee for their pertinent and useful comments that shaped this work. We also thank the referee for pointing out several mistakes in our original arguments, and for sharing their ideas on how to fix those mistakes. 

\section{Model spaces for quotients of free products}
\label{sec: model spaces}

Let $G_1, G_2, \dots, G_n$ be groups.  
Let $(\vertspace_1, p_1), (\vertspace_2, p_2), \dots, (\vertspace_n, p_n)$ be 
based geodesic metric spaces such that each $G_i$ acts on $\vertspace_i$ by isometries. 
We suppress basepoints except when explicitly needed. 
Denote the free product by $G_* = G_1 * \dots * G_n$ and let $\mc{Z} = \{Z_1, \dots, Z_n\}$.
In this section we construct a space $X_{\mc{R}}(\mc{Z})$ on which certain quotients of $G_*$ act properly and cocompactly.

The free product $G_*$ is the fundamental group of a graph of groups $\Sigma$ whose underlying graph is a star with $n$ leaves: each of the free factors $G_i$ is the vertex group for exactly one of the leaf vertices, the vertex group of the central vertex is trivial, and all edge groups are trivial. Let $T$ denote the Bass--Serre tree for the graph of groups $\Sigma$ (see \cref{fig: Bass--Serre tree}). 

\begin{figure}
    \centering
    \definecolor{fuqqzz}{rgb}{0.9568627450980393,0,0.6}
\definecolor{zzttqq}{rgb}{0.6,0.2,0}
\definecolor{qqwuqq}{rgb}{0,0.39215686274509803,0}
\definecolor{ududff}{rgb}{0.30196078431372547,0.30196078431372547,1}
\definecolor{ccqqqq}{rgb}{0.8,0,0}
\begin{tikzpicture}[line cap=round,line join=round,>=triangle 45,x=1cm,y=1cm, scale =.5]
\draw [line width=1pt ] (0,0)-- (3,0);
\draw [line width=1pt ] (0,0)-- (2.68,1.68);
\draw [line width=1pt ] (0,0)-- (1.46,-2.24);
\draw [line width=1pt ] (0,0)-- (-2,1.68);
\draw [shift={(0,0)},line width=1pt ,dashed]  plot[domain=3.450685656152383:4.337790954018381,variable=\t]({1*2.498399487672058*cos(\t r)+0*2.498399487672058*sin(\t r)},{0*2.498399487672058*cos(\t r)+1*2.498399487672058*sin(\t r)});
\draw [shift={(0,0)},line width=1pt ,dashed]  plot[domain=1.0256966734882196:1.9924838459917975,variable=\t]({1*2.1986359407596336*cos(\t r)+0*2.1986359407596336*sin(\t r)},{0*2.1986359407596336*cos(\t r)+1*2.1986359407596336*sin(\t r)});
\draw [line width=1pt ] (0,12)-- (3.195,11.9925);
\draw [line width=1pt ] (0,12)-- (-3.735,14.3325);
\draw [line width=1pt ] (3.195,11.9925)-- (2,14);
\draw [line width=1pt ] (3.195,11.9925)-- (3.195,14.1075);
\draw [line width=1pt ] (3.195,11.9925)-- (4.455,13.8375);
\draw [line width=1pt ] (3.195,11.9925)-- (4.905,12.1725);
\draw [line width=1pt ] (3.195,11.9925)-- (2.205,10.6875);
\draw [line width=1pt ] (0,12)-- (1.755,8.3475);
\draw [line width=1pt ] (1.755,8.3475)-- (3.105,9.2025);
\draw [line width=1pt ] (1.755,8.3475)-- (3.51,7.8975);
\draw [line width=1pt ] (1.755,8.3475)-- (1.485,6.9075);
\draw [line width=1pt ] (1.755,8.3475)-- (0.18,8.2575);
\draw [line width=1pt ] (-3.735,14.3325)-- (-5.94,13.2525);
\draw [line width=1pt ] (-3.735,14.3325)-- (-4.5,12.6675);
\draw [line width=1pt ] (-3.735,14.3325)-- (-5.04,16.8975);
\draw [line width=1pt ] (-3.735,14.3325)-- (-2.745,16.3125);
\draw [line width=1pt ] (0,12)-- (1.71,16.5825);
\draw [line width=1pt ] (1.71,16.5825)-- (3.69,17.0775);
\draw [line width=1pt ] (1.71,16.5825)-- (2.88,18.2925);
\draw [line width=1pt ] (1.71,16.5825)-- (0.945,18.2025);
\draw [line width=1pt ] (1.71,16.5825)-- (3.42,16.0875);
\draw [line width=1pt ,dashed] (6.3,12.3525)-- (8.55,12.4425);
\draw [line width=1pt ,dashed] (5.13,17.3925)-- (7.11,18.0675);
\draw [line width=1pt ,dashed] (-6.345,15.1425)-- (-7.875,15.6825);
\draw [line width=1pt ,dashed] (3.195,7.0425)-- (4.995,6.9525);
\draw [line width=.5pt, ->] (0,6.6375)-- (0,3.5);
\draw [line width=.5pt, ->] (0,6.6375)-- (0,4);
\begin{scriptsize}
\draw [fill=black] (0,0) circle (2pt);
\draw[color=black] (0.2,0.4) node {$1$};
\draw [fill=ududff] (3,0) circle (2.8pt);
\draw[color=ududff] (3.4,.4) node {$G_1$};
\draw [fill=qqwuqq] (2.68,1.68) circle (2.8pt);
\draw[color=qqwuqq] (3.195,2) node {$G_2$};
\draw [fill=red] (1.46,-2.24) circle (2.8pt);
\draw[color=red] (1.935,-2) node {$G_n$};
\draw [fill=fuqqzz] (-2,1.68) circle (2.8pt);
\draw[color=fuqqzz] (-1.485,2) node {$G_i$};
\draw [fill=black] (0,12) circle (2pt);
\draw [fill=ududff] (3.195,11.9925) circle (2.8pt);
\draw [fill=fuqqzz] (-3.735,14.3325) circle (2.8pt);
\draw [fill=black] (2,14) circle (2pt);
\draw [fill=black] (3.195,14.1075) circle (2pt);
\draw [fill=black] (4.455,13.8375) circle (2pt);
\draw [fill=black] (4.905,12.1725) circle (2pt);
\draw [fill=black] (2.205,10.6875) circle (2pt);
\draw [fill=red] (1.755,8.3475) circle (2.8pt);
\draw [fill=black] (3.105,9.2025) circle (2pt);
\draw [fill=black] (3.51,7.8975) circle (2pt);
\draw [fill=black] (1.485,6.9075) circle (2pt);
\draw [fill=black] (0.18,8.2575) circle (2pt);
\draw [fill=black] (-5.94,13.2525) circle (2pt);
\draw [fill=black] (-4.5,12.6675) circle (2pt);
\draw [fill=black] (-5.04,16.8975) circle (2pt);
\draw [fill=black] (-2.745,16.3125) circle (2pt);
\draw [fill=qqwuqq] (1.71,16.5825) circle (2.8pt);
\draw [fill=black] (3.69,17.0775) circle (2pt);
\draw [fill=black] (2.88,18.2925) circle (2pt);
\draw [fill=black] (0.945,18.2025) circle (2pt);
\draw [fill=black] (3.42,16.0875) circle (2pt);
\end{scriptsize}
\end{tikzpicture}
    \caption{The Bass--Serre tree $T$ and the graph of groups $\Sigma$.}
    \label{fig: Bass--Serre tree}
\end{figure}

We now construct a space $T(\mc{Z})$.  
For each vertex $gG_i$ of $T$ we take  a copy of $(\vertspace_i,p_i)$ that we denote by $(Z_i^g,p_i^g)$ and for each vertex $g\{1\}$ we take a singleton $\{\star^g\}$. The space $T(\mc{Z})$ is obtained from the disjoint union of the $(Z_i^g,p_i^g)$  and $\{\star^g\}$ by adding, for each edge $(gG_i,gg_i\{1\})$ of $T$, an edge $(g_ip_i^g,\star^{gg_i})$. 

\begin{remark}
The space $T(\mc{Z})$ is a realization of a tree of spaces with underlying tree $T$, in which the spaces for each $G_i$ are copies of $(\vertspace_i, p_i)$ and all other spaces are single points. See \cite{scott_wall}.
\end{remark}

There is an induced action of $G_*$ on $T(\mc Z)$ and a natural $G_*$-equivariant projection $\pi_T: T(\mc{Z}) \twoheadrightarrow T$. We call the pre-image of a vertex of $T$ a \emph{vertex space}.

Let $\mc R$ be a finite subset of $G_*$ consisting of loxodromic elements for the action of $G_*$ on $T$. We construct model spaces $X_\mc{R}(\mc{Z})$, and $X_\mc{R}$ for $G=G_*/\llangle \mc R\rrangle$.

Let $\mathcal W\subset G_*$ be the set of conjugates of the elements of $\mathcal R$.  For every $w\in \mathcal W $, we construct a line $L(w) \subseteq T(\mc{Z})$ which is stabilized by $w$ in such a way that $L(w)$ projects to the unique bi-infinite geodesic axis $\bar{L}(w) \subseteq T$ on which $w$ acts by translations:

For each $g \in G_i$ fix a geodesic path $\alpha_g \subseteq Z_i$ joining the basepoint $p_i$ to $g\cdot p_i$.
Then, for each length $2$ subpath $(g1,gG_i)\cup (gG_i,gg_i1)$ of $\bar{L}(w) \subset T$, we choose the lift $(\star^g,p_i^g)\cup \alpha_{g_i}^g \cup (g_ip_i^g,\star^{gg_i})$ in $T(\mc{Z})$. 
This defines a unique lift $L(w)$ in $T(\mc Z)$ of $\bar{L}(w)$ by taking the union of the pre-images of its subpaths of length at most $2$. Note that $\pi(L(w)) = \bar{L}(w)$. 

Denote by $X'(\mc{Z}) = \leftQ{T(\mc{Z})}{\llangle \mc{R}\rrangle}$ and $q: T(\mc{Z}) \twoheadrightarrow X'(\mc{Z})$. Note that $G$ acts on $X'(\mc{Z})$. 
For any $x \in L(w)$ we have $q(x) = q(w\cdot x)$. Let $w$ be a conjugate 
of an element in $\mc R$ and let $x\in L(w)$.  
Since the infinite cyclic subgroup $\langle w\rangle$ acts cocompactly on $L(w)$, the finite subsegment of $L(w)$ joining $x$ and $w\cdot x$ determines a cycle $\sigma_w: S^1 \to X'(\mc{Z})$ that maps surjectively onto $q(L(w))$.
We take $\sigma_{w}$ to be the boundary attaching map of a 2-cell $D_w$ to $X'(\mc{Z})$. The resulting complex is
\[
X_\mc{R}(Z_1, \dots, Z_n) = X_\mc{R}(\mc{Z}) := \rightQ{X'(\mc{Z}) \sqcup (\bigsqcup_{ w } D_w)}{\left(\boundary D_ w \sim \sigma_w\right)},
\]
and we refer to the 2-cells $D_w$ as \emph{relator cells}.

\begin{remark}
    It may be the case that two words $w=grg^{-1}, w'=hsh^{-1}$ have the same axes $L(w), L(w')$; for example, if $g=h$ and $r = s^{-1}$. In this case we will see two disks $D_w, D_{w'}$ in $X_\mathcal{R}(\mathcal{Z})$ identified along their boundary. We will show later that, in the probabilistic setting we are interested in, this does not occur (see \cref{cor: factors embed}).
\end{remark}

We let $X_\mc{R}=X_\mc{R}(Z_1', \dots, Z_n')$ where each space $Z_i'$ is fixed to be a singleton. Note that $X_\mc{R}$ is a $2$-complex with 1-skeleton $\leftQ{T}{\llangle \mc{R} \rrangle}$. There is a natural projection map 
$$\pi_X: X_\mc{R}(\mc{Z}) \to X_\mc{R}$$
such that the diagram in \Cref{fig: tree of spaces construction} commutes.

\begin{figure}
\centering
\begin{tikzcd}
    T(\mc{Z}) 
    \ar[r]
    \ar[d, two heads, "\pi_T"]
        &   X_{\mc{R}}(\mc{Z}) \ar[d, two heads, "\pi_X"]
    \\
    T  \ar[r]
        &   X_{\mc{R}}
\end{tikzcd}
\caption{The actions of $G_*$ on $T(\mc{Z})$ and $T$, respectively, induces actions of $G$ on $X_\mc{R}(\mc{Z})$ and $X_\mc{R}$, respectively.}
\label{fig: tree of spaces construction}
\end{figure}

The following is an immediate consequence of the construction.

\begin{lemma}\label{lem: quotient action}
The action of $G_*$ on the tree of spaces $T(\mc{Z})$ (respectively $T$) induces an action of the quotient group $G = \rightQ{G_*}{\llangle \mc{R}\rrangle}$ on the space $X_{\mc{R}}(\mc{Z})$ (respectively $X_{\mc{R}}$). Since $\pi_T$ is $G_*$-equivariant, $\pi_X$ is $G$-equivariant. \qed
\end{lemma}

In general, the projection from $T(\mc Z)$ to $X_{\mc{R}}(\mc{Z})$ does not restrict to an embedding of the vertex spaces $\vertspace_i$ into $X_{\mc{R}}(\mc{Z})$. Indeed, the group $G$ could even be a finite or trivial group. The following is an example where the vertex spaces do not embed and $G$ is not finite. 

\begin{example}\label{Ex:factor-groups-do-not-inject}
Let $G_1=\langle a\rangle$ be an infinite cyclic group, and let $G_2=\langle b, c \rangle$ be a free group. Let $Z_1$ and $Z_2$ be their respective Cayley graphs. Let $\mc R=\{a(bc)^{10},(bc)^{-10}a\}$, and note that $a(bc)^{10}$ and $(bc)^{-10}a$ are loxodromic for the action of $G_1*G_2$ on its Bass-Serre tree. Observe that $a^2=1$ and $(bc)^{20}=1$ in the quotient group $G=G_1*G_2/\llangle \mc R\rrangle$. Thus the projection map from $T(\mc Z)$ to $X_\mc{R}(\mc Z)$ does not restrict to an embedding of $Z_1$ or $Z_2$. 
\end{example}

However, we will later show that in our random model with overwhelming probability the factor groups $G_i$ embed in the quotient $G$ and the projection map from $T(\mc Z)$ to $X_\mc{R}(\mc Z)$ restricts to an embedding of $\vertspace_i$ in  $X_{\mc{R}}(\mc{Z})$ (see \cref{cor: factors embed}).

\begin{example}
    Assume that $\mc{R}$ satisfies the $C'(1/6)$ small cancellation condition over the free product. Then the space $X_\mc{R}$ is the development of the complex of groups described in \cite{martin_steenbock}. 
   If, furthermore, the groups $G_i$ act properly and cocompactly on CAT(0) cube complexes $Q_i$, then 
    $X_{\mc{R}}(Q_1, \dots, Q_n)$ coincides with the ``blown-up space'' used in \cite{martin_steenbock} to build walls. 
\end{example}

   Note that vertices of $T$ are either cosets of $\{1\}$, in which case we call them \emph{central vertices}, or they are cosets of some $G_i$, in which case they are called \emph{factor vertices}. Similarly, we refer to a vertex of $X_{\mathcal R}$ that is in the image of a central vertex of $T$ as a \textit{central vertex}, and otherwise we call the vertex a  \emph{factor vertex}. By the definition of $X_{\mathcal R}$, this yields a  well-defined partition of the vertices of $X_{\mathcal R}$ into factor vertices and central vertices.

\begin{lemma}\label{lem: finite edge stabilizers}
    The quotient $G = \rightQ{G_*}{\llangle \mc{R} \rrangle}$ acts on the complex $X_\mc{R}$ with trivial edge stabilizers. Moreover, the stabilizer of a vertex of $X_\mc{R}$ satisfies the following: if the vertex is central, then the stabilizer is trivial. Otherwise, it is conjugate to the image of one of the groups $G_i$ in $G$. 
\end{lemma}

\begin{proof}
Let $\overline{x}$ be an edge or a vertex of $X_{\mathcal R}$  and let $\overline{g}\in \rightQ{G_*}{\llangle \mc{R} \rrangle}$ be the image of an element $h\not=1$ in $G_*$ such that $\overline{h}\, \overline{x}=\overline{x}$. Let $x$ be a pre-image of $\overline{x}$. If $hx\not = x$, then $h \in \llangle \mathcal R\rrangle$ and $\overline h=1$ in $G$. This settles the case where $\overline{x}$ is an edge. Otherwise, $x$ and $\overline x$ are vertices, and $h$ is in the stabilizer of $x$. If $x\in \{g\{1\}\mid g\in G\}$, then the stabilizer of $x$ is trivial and $h=1$. Thus the stabilizer of $\overline x$ is trivial in this case. Otherwise, the stabilizer of $x$ is conjugate to one of the groups $G_i$. This yields the lemma.  
\end{proof}

We end this section by introducing the following important notation.

\begin{defn} \label{def: edge triple rotation element}
    An \emph{edge triple} in $T$ is an ordered triple $(e_1, v, e_2)$ so that $e_1, e_2$ are edges in $T$ containing a factor vertex $v$. We assign a label to each edge triple in $T$, named a \emph{rotation element}, as follows: Given an edge triple $(e_1, v, e_2)$, where $e_1 = (w_1\{1\}, w_2 G_i), e_2 = (w_2 G_i, w_3\{1 \})$, the \emph{rotation element} from $e_1$ to $e_2$  is $w_1^{-1}w_3 \in G_i.$
\end{defn}
\begin{remark}\label{rem: triples invariant} The rotation element $w_1^{-1}w_3$ is indeed in $G_i$ as  case $w_1=w_2w_1'$ and $w_3=w_2w_3'$, where $w_1',w_3'\in G_i$. 
  Rotation elements of edge triples in $T$ are $G_*$-invariant, that is,  $(e_1, v,e_2)$ and $g(e_1, v,e_2)$ have the same rotation element for any $g\in G_*$.  
\end{remark}
 By \cref{rem: triples invariant} the following is well-defined.
\begin{defn}\label{def: edge triple rotation element XR}
    An \emph{edge triple} in $X_{\mc R}$ is an ordered triple $(e_1, v, e_2)$ so that $e_1$, $e_2$ are edges in $X_{\mc R}$ containing a factor vertex $v$. Let $(e_1',v',e_2')$ be a pre-image of $(e_1,v,e_2)$ in $T$. 
     The \emph{rotation element of $(e_1, v, e_2)$} is the image of the rotation element of $(e_1', v', e_2')$ under the quotient map $q:G_*\to G$. Moreover, in this case the rotation element of $(e_1',v',e_2')$ in $G_*$ is called a \emph{geometric representative} of the rotation element of $(e_1,v,e_2)$.  
\end{defn}

\begin{remark}
    The rotation element of an edge triple in $X_\mc{R}$ is well-defined in $G$  by \cref{rem: triples invariant}, but may have multiple geometric representatives in $G_i$, for some factor group $G_i$ of $G_*$. This happens if and only if $G_i$ does not embed in $G$. For instance, let $G$ be the group of \cref{Ex:factor-groups-do-not-inject} and let $(e_1',v',e_2')$ be an edge triple in $T$ with rotation element $a^2$. Then the rotation element of the image of $(e_1',v',e_2')$  in $X_\mc{R}$ is the identity element, and both the identity element and $a^2$ are geometric representatives of this rotation element in $G_*$.  Similarly, rotation elements about a given factor vertex with rotation elements $a$ and $a^3$ are also geometric representatives of the same edge triple in $X_{\mc{R}}$. This is related to the fact that a complex of group is developable if and only if the vertex groups embed, see \cite[Chapter III.$\mathcal C$ Theorem 2.13 and Proposition 3.14]{bridson_haefliger_1999}. 
\end{remark}

\section{Diagrams in $X_\mc{R}$}

By a \emph{diagram} in a 2-complex $X$ we mean a combinatorial map $Y\to X$, where $Y$ is a finite connected 2-complex. The \emph{area of $Y$}, denoted by $Area(Y)$, is the number of $2$-cells in $Y$. We denote by $\partial Y$ the set of edges in the boundary of $Y$, and by $|\partial Y|$ its cardinality. 
By a \emph{disk diagram} in a 2-complex $X$, we mean a combinatorial map $D \to X$, where $D$ is a 2-complex that is planar, connected, and simply-connected. In particular, a disk diagram $D$ may have a spur so need not be a topological disk.

\subsection{Diagrams over \texorpdfstring{$\mathcal R$}{R}}

Let $Y\to X_\mc{R}$  be a finite diagram  such that 
\begin{enumerate}
    \item the boundary of every $2$-cell has length $=2\ell$, 
    \item there is a partition of the vertices into factor vertices and central vertices that is respected by the map $Y\to X_\mc{R}$, i.e., this map sends factor vertices to factor vertices, and central vertices to central vertices. 
    \end{enumerate}
An \emph{edge triple} in $Y$ is a triple $(e_1,v,e_2)$, where $v$ is a factor vertex and $e_1$ and $e_2$ are two edges of $Y$ at $v$. 

Note that every diagram $Y\to X_\mc{R}$ satisfying the points (1) and (2) above inherits, via the map $Y\to X_{\mc{R}}$, a decoration in the following sense. 
\begin{defn}\label{D: decoration for diagrams in X_R}  A \emph{decoration} of a diagram $Y\to X_\mc{R}$ satisfying (1) and (2) above consists of 
    \begin{itemize}
        \item a factor group $G_v$ for each factor vertex $v$,
        \item for each edge triple $(e_1,v,e_2)$, where $e_1$ and $e_2$ are edges at the factor vertex $v$, a \emph{rotation element} $s\in G_v$ which is the identity element if and only if $e_1 = e_2$,
        \item for each $2$-cell a fixed factor vertex in its boundary and an orientation of the boundary, 
    \end{itemize}
such that the following hold.
\begin{enumerate}
    \item  The rotation element of an edge triple in $Y$ is a geometric representative of the rotation element of its image in $X_{\mc R}$.

    \item Let $c$ be a $2$-cell and let $s_1, s_2, \dots, s_m$ be the rotation elements, read starting from the  fixed vertex in its boundary in the orientation fixed on the boundary. Then the word $r=s_1s_2\cdots s_m\in \mathcal R$. In this case, we say that  $c$ \emph{bears} the word $r\in \mathcal R$. 
\end{enumerate}
\end{defn}

\begin{defn}\label{D: diagram over R} 
The diagram $Y$ is a \emph{diagram over $\mathcal R$} if the rotation elements assigned to edge triples $(e_1,v,e_2)$ and $(e_2,v,e_1)$ are inverse to each other in $G_*$.
\end{defn}

As the following example shows, diagrams in $X_\mc{R}$ are not necessarily diagrams over $\mathcal R$. Indeed, condition (2) of \cref{D: decoration for diagrams in X_R} may be in conflict with \cref{D: diagram over R} if factor groups do not embed in $G$. 

\begin{example} \label{ex: diagrams R vs XR}
Let $G_1 = \langle a|\emptyset \rangle, G_2 = \langle b|\emptyset \rangle$, $\mc{R} = \{ab, a^{-1}b\}$ and  $G=G_1*G_2/\llangle \mc R\rrangle=\mathbb{Z}/2\mathbb{Z}$.
    \Cref{fig:diskdiagrams}(d) is  a spherical diagram in $X_\mc{R}$, which is not a diagram over $\mc R$. Indeed, let $v$ be the front vertex in the diagram in \cref{fig:diskdiagrams}(d), and let $e,f$ be the two edges at $v$ where $e$ is the front edge in the figure. The rotation element of the edge triple $(e,v,f)$ in the relator face $C_2$ has to be $a$, and the rotation element of the edge triple $(f,v,e)$ in the relator face $C_3$ has to be $a$ as well. This is because the two $2$-cells in this diagram need to bear a relator. These two rotation elements are not inverse to each other in $G_1=\langle a \rangle$.
\end{example} 

Finally, a \emph{reduction pair} in  a diagram $Y$ over $\mathcal R$ is a pair of $2$-cells which share at least one edge, with opposite orientations and a common vertex such that the respective boundary paths starting from that vertex in the given orientation bear the same words in $\mathcal R$. A diagram $Y$ over $\mathcal R$ is \emph{reduced} if it does not contain any reduction pair. 

\begin{remark}\label{R:reduced-disk-diagram-over-R}
  If a disk diagram $D$ over $\mathcal R$ filling a given loop in $X_\mc{R}$ is of minimal area, then it is also reduced. 
\end{remark}

We have the following. 
 \begin{lemma}\label{L:diagramsoverR}
     Assume that the quotient map from $G_*$ to $G$ is injective on the factor groups. Then every diagram $Y \to X_{\mc{R}}$ inherits a unique decoration via the map $Y \to X_{\mc{R}}$, so that the rotation elements assigned to edge triples $(e_1,v,e_2)$ and $(e_2,v,e_1)$ are inverse to each other in $G_*$. In particular, with its inherited decoration $Y$ is a diagram over $\mathcal R$. 
 \end{lemma}
 \begin{remark}\label{rem: triple cycles}
     In fact even more is true. For instance, if $v$ is a factor vertex and $Y'$ is a planar subdiagram containing $v$, 
     if edge triples $(e_1,v,e_2)$, $(e_2,v,e_3)$, $(e_1,v,e_3)$ have rotation elements $s_1, s_2, s_3$ respectively then $s_1s_2 = s_3$ in $G_v$.
 \end{remark}
 \begin{proof}[Proof of \cref{L:diagramsoverR}]
     As the factor groups embed, the quotient map from $T$ to $X^{(1)}_\mc{R}$ is injective on the set of edge triples at any fixed factor vertex of the Bass-Serre tree $T$. For this reason, the rotation element of any edge triple at a factor vertex in $X_\mc{R}$ has a unique geometric representative in $G_v$.  Also, every $2$-cell of $Y$ is sent to a $2$-cell of $X_\mc{R}$; and since the quotient map from $T$ to $X^{(1)}_\mc{R}$ is injective on the set of edge triples at any fixed factor vertex, any edge triple $(e_1,v,e_2)$ in a $2$-cell of $Y$ is sent to an edge triple $(f_1,v,f_2)$ in $X_{\mc{R}}$ where $f_1\not=f_2$. As the geometric representative of the rotation element of such an edge triple is unique, the claim follows by assigning this element as the rotation element of $(e_1,v,e_2)$.
 \end{proof}

\subsection{Disk diagrams over \texorpdfstring{$\mathcal R$}{R}}

Observe that any embedded cycle $c$ in $X_\mc{R}^{1}$ bounds a disk diagram over $\mathcal R$. Indeed, this follows because $c$ lifts to a segment $\tilde{c}$ in $T$ whose endpoints are translates of each other by an element of $\llangle \mc{R} \rrangle$. This yields the following variant of the van Kampen Lemma.

If $D$ is a disk diagram over $\mathcal R$, we say that its boundary bears a word $w\in \llangle R\rrangle$ if there is a factor vertex $v_0$ on $\partial D$ and an orientation of $\partial D$ such that $w$ equals the word obtained by concatenating the rotation elements at the edge triples in $\partial D$, starting from $v_0$ in the given orientation.  

\begin{lemma}\label{L:vanKampenlemma}
    Let $w\in \llangle \mathcal R \rrangle$ be represented by a word in normal form. Then there is a disk diagram over $\mathcal R$ whose boundary bears the word $w$.  
\end{lemma}
We describe how this lemma follows from the traditional van Kampen lemma. It may help to refer to \Cref{fig:diskdiagrams}.

\begin{figure}
    \centering
    \includegraphics[width=\linewidth]{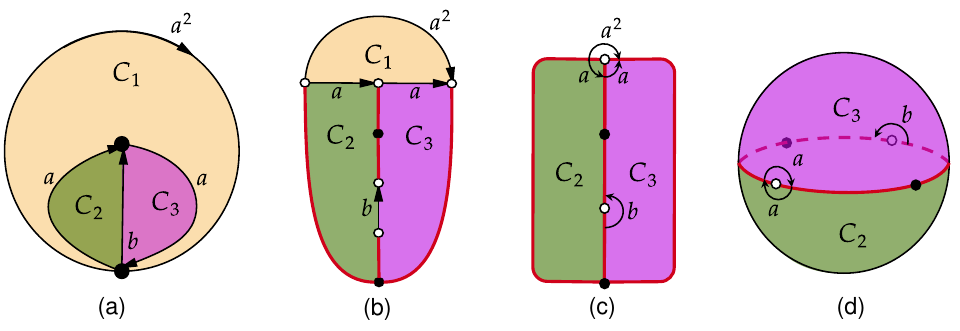}
    \caption{Let $G_1 = \langle a|\emptyset \rangle, G_2 = \langle b|\emptyset \rangle$, $\mc{R} = \{ab, a^{-1}b\}$ and  $G=G_1*G_2/\llangle \mc R\rrangle=\mathbb{Z}/2\mathbb{Z}$. Figure (a) shows a van Kampen diagram $D$ of $G$ for the relation $a^2=1$. There are two factor diagrams: firstly, the top $2$-cell enclosed by the edges labelled $a$, $a$ and $a^2$, and, secondly, the edge labelled $b$. The bottom vertex has factor degree $3$, the top vertex has factor degree $2$. Figure (b) shows the corresponding diagram $\widetilde D$ in $X_\mc{R}(\mathcal Z)$, where $Z_1$ and $Z_2$ are the Cayley complexes for $G_1$ with respect to the generating set $G_1$ and all relations in $G_1$, for $G_2$, respectively. Here we have retained the orientation and labels of the edges of $Z_1$ and $Z_2$ in $X_\mc{R}(\mathcal Z)$. Figure (c) shows a corresponding disk diagram $\overline D$ over $\mc R$. Figure (d) is a spherical diagram in $X_\mc{R}$ obtained by identifying the boundary paths in Figure (c).  Figure (a) can be obtained from Figure (b) by contracting all the unlabelled edges to a point. Figure (c) can be obtained from Figure (b) by contracting each factor diagram to a point.}
    \label{fig:diskdiagrams}
\end{figure}
\begin{proof}
For all $1\leqslant i\leqslant m$ fix a presentation $\langle G_i | \mathcal{R}_i\rangle$ of $G_i$ where $\mc{R}_i$ is the set of all relators in $G_i$ over the alphabet $G_i$. Let
$Z_i$ be the Cayley complex of $G_i$ with respect this presentation and let $\mc{Z} = \{Z_1, \dots, Z_n\}$. By \Cref{lem: quotient action}, there is a $G$-equivariant projection from $X_\mc{R}(\mathcal Z)$ onto $X_{\mc{R}}$ obtained by simultaneously contracting all the copies of the spaces $Z_i$ in $X_\mc{R}(\mathcal Z)$ to a point. 

     On the other hand, there also is a $G$-equivariant projection from $X_\mc{R}(\mathcal Z)$ onto the Cayley complex $\Gamma_G$ with respect to the presentation $\langle G_1 \cup \dots \cup G_n|\mc{R}_1\cup\dots\cup \mc{R}_n\cup \mc{R}\rangle$ obtained by simultaneously contracting each of the edges originating in the Bass-Serre tree $T$ to a point. We now mimic the inversion of this procedure for a van Kampen diagram over $G$.

       Let $w\in \llangle \mathcal R \rrangle$ be a word in normal form in the free product $G_*$. Let $D$ be a van Kampen diagram for $w$ over $G$ of minimal area (e.g. \Cref{fig:diskdiagrams}(a)). Each edge is oriented and labelled by an element of one of the factor groups. The diagram $D$ is built up of \emph{factor faces} and \emph{relator faces}. The former are the $2$-cells all of whose edges are labelled by an element of the same factor group, and whose boundary bears a relation in that same factor group. The latter are the  $2$-cells of boundary length $=\ell$ that bear the relators of $\mathcal R$. Note that no two consecutive edges in the boundary of a relator face can be labelled by elements in the same factor group. This is because the elements of $\mathcal R$ are given in normal form, so that any two consecutive syllables are in distinct factor groups. Also note that each such syllable is a generator of $G$. A \emph{factor diagram} is either a single factor face or an edge that is not contained in any factor face.  
       Let $v$ be a vertex in $D$. The \emph{multiplicity} of a factor diagram $K\to D$ at $v$ is the number of vertices of the factor diagram that are identified with $v$ via the map $K\to D$. The \emph{factor degree} of a vertex $v$ is the sum of the multiplicities of the factor diagrams at $v$. 
       
       We now build a diagram $\widetilde D$  in $X_\mc{R}(\mathcal Z)$ and a map from $\widetilde D$ to $D$ such that the following diagram commutes: 
       \[ \begin{tikzcd}
\arrow{r}{} X_\mc{R}(\mathcal Z)    & \Gamma_G \\%
 \arrow{r}{} \widetilde D \arrow[swap]{u}{} & D \arrow{u}{} 
\end{tikzcd}.
\]
       To this point replace each vertex of $D$ by a star of edges as follows. If the factor degree of the vertex $v$ is $d$, we replace $v$ by a star consisting of $d$ edges such that contracting the new edges yields the original diagram, and such that each leaf vertex of the star is in exactly one factor diagram. The central vertices of such a star are referred to as central vertices. We note that this procedure does not alter the factor diagrams, while the boundary length of a relator face increases. In fact, every vertex in the boundary of a relator face is replaced by two edges joined at a central vertex.  
       
      Since the Cayley complex of $G$ is obtained from $X_\mc{R}(\mathcal Z)$ by contracting each edge originating in the Bass-Serre tree to a point, this procedure results in a disk diagram $\widetilde{D}\to X_{\mc{R}}(\mathcal Z)$ (e.g. \Cref{fig:diskdiagrams}(b)) by sending each $2$-cell of $\widetilde D$ to the pre-image in $X_{\mc{R}}(\mathcal Z)$ of the image of the corresponding $2$-cell of $D$ in the Cayley complex of $G$. The diagram $\widetilde D$ has edges and $2$-cells originating in the factor diagrams of $D$, and edges originating in the vertices of $D$. Edges of the first type are sent to edges in the copies of the $Z_i$, edges of the second type are sent to the edges originating in the Bass-Serre tree $T$. A maximal connected subdiagram all of whose edges originate in the same factor diagram of $D$ is called a \emph{factor diagram} of $\widetilde D$. Note that the factor diagrams of $D$ and $\widetilde D$ are in 1-1 correspondence. Moreover, each $2$-cell originating from a relator face has boundary of length $=3\ell$; more precisely, $\ell$ edges are sent to images of the complexes $Z_i$ in $X_{\mc{R}}(\mathcal Z)$, and $2\ell$ edges are sent to the image of the Bass-Serre tree in $X_{\mc{R}}(\mathcal Z)$. The map from $\widetilde D$ to $D$ is given by the inverse operation of contracting each edge  of second type in $\widetilde D$ to a point. 

      We now construct a diagram $\overline D$ over $\mathcal R$ and a map from $\widetilde D$ to $\overline D$ such that the following diagram commutes:
             \[ \begin{tikzcd}
\arrow{r}{} X_\mc{R}(\mathcal Z)    & X_{\mc R} \\%
 \arrow{r}{} \widetilde D \arrow[swap]{u}{} & \overline D \arrow{u}{} 
\end{tikzcd}.
\]
      To this point, we  simultaneously contract each factor diagram of $\widetilde D$ to a point. Since $X_\mc{R}$ is obtained from $X_{\mc{R}}(\mathcal Z)$ by contracting each of the spaces $Z_i$ to a point, this procedure yields a disk diagram $\overline D \to X_{\mc R}$ (e.g. \Cref{fig:diskdiagrams}(c)).  The boundary of each $2$-cell of $\overline D$ has length $=2\ell$. Moreover, the images of the factor diagrams are the factor vertices, the remaining vertices  are the central vertices. To each factor vertex $v$ that is the image of an edge labelled by an element of $G_i$ one assigns the factor group $G_v=G_i$. Moreover, to each edge triple $(e_1,v,e_2)$ in $X_{\mc R}$ at factor vertex $v$  we assign a rotation element as follows. Suppose that $v$ is in the image of a factor diagram $K$, and let $z_1$ and $z_2 \in K$ be the vertices in the preimage of $e_1$, and $e_2$  respectively, in $\widetilde D$. Then the rotation element assigned to $(e_1,v,e_2)$ is the element of $G_v$ read on the edge path from $z_1$ to $z_2$ in $K$. By this definition the rotation element assigned to $(e_2,v,e_1)$ is inverse to the rotation element of $(e_1,v,e_2)$ in the factor group $G_v$. This procedure yields a disk diagram over $\mathcal R$ whose boundary bears the word $w$. The map from $\widetilde D$ to $\overline D$ is the map that simultaneously contracts each factor diagram of $\widetilde D$ to a point. This yields the claim. 
\end{proof}

\Cref{L:vanKampenlemma} yields the following.

\begin{lemma}\label{lem: XR simply connected}
    The complex $X_\mc{R}$ is simply connected. \qed
\end{lemma}

\section{(Relative) isoperimetric inequalities} \label{sec: rel_isopermietry}

Our interest in disk diagrams stems from the following fact. 
We say that $X$ satisfies a linear isoperimetric inequality if there is a constant $C>0$, such that $Area(D)\leqslant C |\partial D|$ for all disk diagrams $D$ of minimal area among the diagrams with same boundary curve in $X$.
\begin{prop}[{\cite[Chapter III.H, Section 2, Theorem 2.9 and 3.2 Remark (4)]{bridson_haefliger_1999}}]
\label{P:BHIPI}
    Let $L\geqslant 3$ and let $X$ be a $2$-complex such that every $2$-cell has $L$ edges. If $X$ satisfies a linear isoperimetric inequality, then $X$ is hyperbolic. 
\end{prop}
The main 
goal of this section is 
to prove \cref{thm: cancel implies planar}, which gives
a linear isoperimetric inequality for disk diagrams in $X_{\mc R}$ and an analogous bound for certain non-planar diagrams. It is used in \Cref{sec: global isoperimetry} to argue that $X_\mc{R}$ satisfies a linear isoperimetric inequality, hence, that $X_\mc{R}$ is hyperbolic. 

The following definition gives a way to bound the amount of internal gluing that exists in a (possibly non-planar) diagram.

\begin{defn}[\cite{odr_nonplanar}] Let $Y$ be a 2-complex.  The \emph{area} of $Y$, denoted $Area(Y)$, is the number of 
	 2-cells in $Y$.  
    The \emph{cancellation} of $Y$ is 
	\[
	\can(Y) = \sum_{e \in Y^{(1)},\, \deg(e)\not=0} (\deg(e) - 1),
	\]
 where $\deg(e)$ is the number of 2-cells that contain the edge $e$ with multiplicity.
\end{defn}
\begin{remark}\label{rem: cancel estimate}
  If $D$ is a disk diagram and every 2-cell in $D$ is an $\ell$-gon, then $\can(D) = \frac{1}{2}\left(\ell Area(D) - |\partial D|\right).$ This relates cancellation to isoperimetric inequalities of $Y$.   
\end{remark}

Another notion that has proved useful in the study of random groups in Gromov's density model is the notion of  ``decorated abstract diagrams" \cite{ollivier_gafa}. 
In order to use the combinatorial properties of $X_\mc{R}$ we adapt the technical notions of abstract diagrams and cancellation to our setting. We use these to prove \Cref{thm: relative non planar IPI}, which is the main technical result of this section. We start with decorated abstract diagrams.

\subsection{Decorated abstract diagrams}

From now on, let $\mc{G} = (G_1, \dots, G_n)$ be an $n$-tuple of finitely generated nontrivial groups and let $G \sim \mathcal{FPD}(\mc{G}; d, m, \ell)$. 
  We are interested in diagrams that can be immersed in $X_\mathcal{R}$. 

\begin{defn} An \emph{abstract diagram} is a finite connected 2-complex. 
A \emph{decorated abstract diagram} is a finite connected 2-complex with bipartite 1-skeleton in which every 2-cell is a $2\ell$-gon, along with the following decoration:
    \begin{enumerate}
        \item a distinguished vertex on the boundary of each 2-cell such that all distinguished vertices lie in the same part of the bipartition of the 0-cells,
        \item  an orientation in each 2-cell, and 
        \item a partition of the faces of the 2-complex.
    \end{enumerate}
The equivalence classes of the faces are called \emph{abstract relators}. In analogy with \Cref{D: diagram over R}, we refer to the elements of the part containing the distinguished vertices as \emph{factor vertices}. All other vertices are referred to as \emph{central vertices}. A triple $(e_1,v,e_2)$  in a decorated abstract diagram is an \emph{edge triple} if $e_1$ and $e_2$ are distinct edges containing a factor vertex $v$. 
The \emph{area} of a decorated abstract diagram $Y$, denoted $Area(Y)$, is the number of 2-cells in $Y$.
\end{defn}

\begin{figure}
    \centering
    \includegraphics[width=0.2\linewidth]{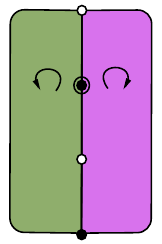}
    \caption{A decorated abstract diagram. Let $G$ and $\mathcal R$ be as in \cref{fig:diskdiagrams}. Then this decorated abstract diagram is fulfillable in $G$, see \cref{fig:diskdiagrams}(c).} 
    \label{fig:abstractdiagram}
\end{figure}

\begin{defn} \label{def: reduced}
A \emph{reduction pair} in a decorated abstract diagram is a pair of 2-cells which share at least one edge, with opposite orientations, which are in the same part of the partition of the $2$-cells,  and for which the first vertex in the overlap occurs at the same (oriented) distance from the distinguished vertex in each face.

An abstract diagram is \emph{reduced} if it contains no reduction pairs.
\end{defn}

Note that if a decorated abstract diagram contains a reduction pair, we can reduce it by identifying the 2-cells in the pair. In this way, the requirement that abstract diagrams be reduced can be interpreted as a condition that abstract diagrams do not contain extraneous 2-cells which carry the same data.

Our goal is to detect whether a given reduced decorated abstract diagram can be realized as a diagram in $X_\mc{R}$ by fulfilling its decoration to make it a diagram over $\mathcal R$, see \cref{D: diagram over R}, i.e. by assigning a rotation element to each edge triple on the boundary of each face so that the boundary of the face bears a relator in $\mathcal{R}$, and then checking if the decorated diagram so obtained can appear as a subcomplex of $X_\mc{R}$. We will compute the probability that such a decoration is possible inductively. 

\begin{defn}\label{def: fulfillable}
A decorated abstract diagram $Y$ is \emph{fulfillable} in $G$ if there is
\begin{itemize}
\item an assignment of a factor group $G_v$ to each factor vertex $v$ of degree $\geq 2$,
    \item an assignment of a relator of $\mc{R}$ for each face such that the assignment respects the partition of faces, and two faces from different equivalence classes are not assigned the same relator, 
     \item an assignment of  a nontrivial element of $\bigcup_{i=1}^nB_i(m)$ to each edge triple in a face ---  called \emph{rotation elements} in analogy to  \cref{D: diagram over R}, 
\end{itemize}
such that the following hold.
\begin{enumerate}   
  \item For any factor vertex $v$, the rotation elements assigned to each edge triple centered at $v$ are all contained in $G_v$.
       \item  Let $f$ be a face and let $w_1,\ldots, w_{\ell}$ denote the rotation elements in order from the distinguished vertex, in the direction of the given orientation. Then the word $w_1\ldots w_{\ell}$ is the relator assigned to $f$. In this case we say that the face \emph{bears} the relator assigned to it.
   \item The rotation element assigned to an edge triple $(e_1,v,e_2)$ is inverse in $G_v$ to the rotation element assigned to $(e_2,v,e_1)$.
\end{enumerate}
If a decorated abstract diagram $Y$ is fulfillable in $G$, we also say that \emph{$Y$ fulfills $\mathcal R$.}
\end{defn}
We now compare \cref{def: fulfillable} to \cref{D: diagram over R} of diagrams over $\mc R$. 

\begin{remark}\label{rem:fulfillable}
    Let $Y \to X_\mc{R}$ be a diagram over $\mathcal R$.  The map from the 2-complex $Y$ to $X_{\mc R}$ induces a decoration on the $2$-complex $Y$ which lets us view $Y$ as a decorated abstract diagram. We call $Y$ the \emph{underlying} decorated abstract diagram of $Y \to X_{\mc{R}}$. If $Y \to X_\mc{R}$ is reduced, then $Y$ is a reduced abstract diagram. Furthermore, as $Y\to X_{\mc R}$ is a diagram over $\mathcal R$, $Y$ is fulfillable.  
A concrete example of a fulfillable decorated abstract diagram is given in \cref{fig:abstractdiagram}. 
\end{remark} 

\begin{remark} Suppose that an abstract diagram admits a combinatorial map to $X_\mc{R}$. Then it also inherits a decoration via the map to $X_{\mc R}$. However, the decorated abstract diagram so obtained may not be fulfillable, see \cref{ex: diagrams R vs XR} and \Cref{fig:diskdiagrams}(d).  
\end{remark}

We first prove the linear isoperimetric inequality in $X_\mc{R}$. 
By \Cref{L:vanKampenlemma}, it is sufficient to prove the linear isoperimetric inequality for disk diagrams over $\mc R$. We actually prove that for a fixed reduced abstract diagram, the probability that it can be fulfilled and that it violates the inequality is, with overwhelming probability, 0. By \cref{rem:fulfillable}, this implies that the probability of being a reduced diagram over $\mathcal{R}$ which violates the linear isoperimetric inequality is 0. We then argue that the number of diagrams of interest in $X_\mathcal{R}$ is sufficiently small, so that this statement holds, with overwhelming probability, for all such diagrams.

 \begin{defn}\label{D: corner}
 Let $Y$ be a decorated abstract diagram. A vertex of $Y$ is called a  \emph{corner} of $Y$ if the degree of $v$ is $\geq 3$ or if  $v$ is the $\ell$-th vertex of some face $f$.
 \end{defn}

Recall that the boundary of each face $f$ has length $2\ell$, so every $|\partial f|/2$-th vertex is a corner. This will be important in our proof, see \cref{L: using corners}.

\begin{defn} \label{def: corner labeled}
A \emph{corner labeling} on a decorated abstract diagram is an assignment of a factor group to each corner and of all the rotation elements centered at such factor vertices.
\end{defn}

Note that a corner labeled decorated abstract diagram need not a priori be fulfillable.

\subsection{Counting corner labeled abstract diagrams}
We also need a notion of bounded complexity for our abstract diagrams. In particular, we need a bound on the area and on the amount of gluing required to create the diagram.

\begin{defn}\label{def: connector}
A \emph{connector} in an abstract diagram $Y$ is either a maximal segment of $Y^{(1)}$ such that all vertices on the segment other than its endpoints  have degree 2, or a cycle in $Y^{(1)}$ in which at most one vertex has degree greater than 2.
\end{defn}

\begin{figure}
    \centering
    \includegraphics[width=.3\textwidth]{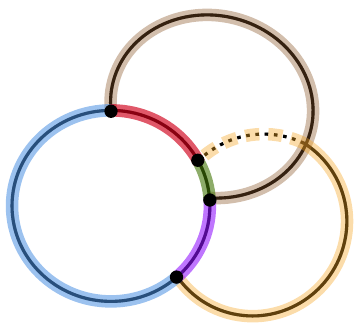}
    \caption{A non-planar diagram with three 2-cells and six maximal arcs, each indicated in a different color. Note that the intersection of two 2-cells need not be a maximal arc.}
    \label{fig:connectors}
\end{figure}

\begin{defn}
 An abstract diagram $Y$ is $(K, M)$-bounded if $Area (Y) \leq K$ and $Y^{(1)}$ is a union of at most $ M$ connectors.
\end{defn}

This definition varies from the definition of a $(K, M)$-bounded diagram in \cite[Section 2]{MP}. In their definition, $M$ bounds the number of gluings needed to construct a diagram. However, the two definitions  are roughly equivalent in the situations explored in this paper, and our definition makes some calculations easier. We note also that any $(K, M)$-bounded $2$-complex must be of complexity $\max\{K, M\}$, in the sense of \cite[Definition 2.2]{Tsai-phase-transition}. Indeed, our connectors are equivalent to Tsai's \textit{maximal arcs}, and if a complex has at most $M$ connectors then every face of the diagram must be bounded by at most $M$ maximal arcs.

\begin{lemma}\label{L: number of corners}
 Let $Y$   be a non-filamentous decorated abstract diagram. If $Y$ is $(K,M)$-bounded, then $Y$ has at most $\ell K+K$ corners. 
\end{lemma}
\begin{proof}
    There are at most $\ell K$ vertices of degree $>2$, and at most $K$ vertices that are the $\ell$-th vertex on a face $f$.
\end{proof}

We will use the following fact. An abstract diagram is \emph{non-filamentous} if all of its edges are contained in at least one $2$-cell. 

\begin{restatable}{proposition}{polynomialdiagrams}\label{lemma:polynomial_diagrams}
The number of $(K, M)$-bounded reduced corner-labeled non-filamentous decorated abstract diagrams is polynomial in $\ell$.
\end{restatable}
    \begin{proof}
        It follows from \cite[Lemma 2.3]{Tsai-phase-transition} that the number of $(K,M)$-bounded non-filamentous decorated abstract diagrams is polynomial in $\ell$. Also, the number of corners in a $(K,M)$ bounded diagram is polynomial in $\ell$ by \cref{L: number of corners}. Moreover, in a $(K, M)$-bounded diagram any vertex has degree at most $2M$ because  every  vertex of degree greater than $2$ is an endpoint of a connector. For every rotation element at a corner, there are $B=\sum_{i=1}^n |B_i(m)|$ possibilities; hence, there are at most $B^{2M}$ possibilities to assign rotation elements per corner. This yields the claim. 
    \end{proof}

We now define a relative version of cancellation for abstract diagrams of $G$, which relies on counting edge triples rather than edges. Roughly speaking, the relative cancellation of a diagram measures constraints on internal edge triples.

\subsection{A relative version of Cancellation}
We say a 2-cell \emph{fully contains} an edge triple if it contains both the edges of the triple, while a 2-cell \emph{partially contains} an edge triple if it contains exactly one of the edges of the triple.

\begin{defn}\label{def: rel degree}
 The \emph{relative degree} of an edge triple in a decorated abstract diagram is defined as follows (see \cref{fig:edge_triple_degree}): 
    \begin{equation*}
        \deg_*(e_1, v, e_2) := 
            \# \left \{\parbox{1.15 in}{\centering 2-cells that fully\\ contain  $(e_1, v, e_2)$} \right\} 
            + \begin{cases} 
                1 &\parbox{1.8in}{if there is a $2$-cell partially containing $(e_1, v, e_2 )$}\\
                0 &\mbox{otherwise}
            \end{cases}.
    \end{equation*}
\end{defn}

\begin{figure}
    \centering
    \includegraphics[width=.8\textwidth]{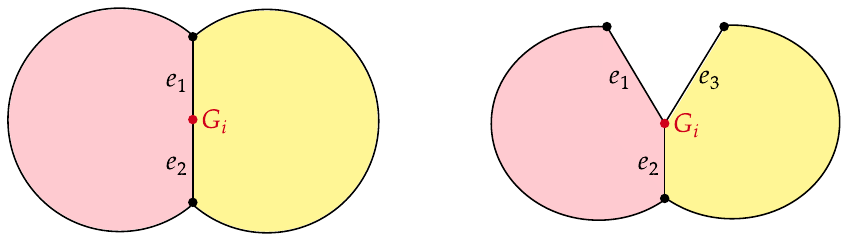}
    \caption{Two possible arrangements of adjacent faces. On the left, the edge triple $(e_1, v, e_2)$ is fully contained in both adjacent faces, and $\deg_*(e_1, v, e_2) = 2$. On the right $(e_1, v, e_2)$ is partially contained in one face and fully contained in the other, and $\deg_*(e_1, v, e_2) = 2$.}
    \label{fig:edge_triple_degree}
\end{figure}

\begin{defn}\label{def: rel cancel} Let $\mathcal E$ be the set of all edge triples that have non-zero relative degree. 
The \emph{relative cancellation} of a decorated abstract diagram $Y$ is 
    \[
    \can_*(Y) = \frac{1}{2}\sum_{(e_1, v, e_2) \in \mathcal E} (\deg_* (e_1, v, e_2) -1).
    \]
\end{defn}

\begin{remark}
    Note that $\deg_*(e_1,v,e_2)=\deg_*(e_2,v,e_1)$, and that edge triples are ordered, i.e. $(e_1, v, e_2) \neq (e_2, v, e_1)$. The factor of $1/2$ in the formula for $\can_*(Y)$ is to correct for the double counting of $\deg_*(e_1,v,e_2)=\deg_*(e_2,v,e_1)$ in $\sum_{(e_1, v, e_2) \in \mathcal E} (\deg_* (e_1, v, e_2) -1)$.
\end{remark}

Note that if an edge triple is not fully contained in a $2$-cell of $Y$, then it does not contribute to $\can_*(Y)$. 

We are now ready to state the main result of this section. 

\begin{theorem}\label{thm: relative non planar IPI}
Let $K, M, \epsilon>0$.  With overwhelming probability the following property holds. Every $(K, M)$-bounded reduced non-filamentous diagram $Y$ over $\mathcal{R}$ satisfies 
    \begin{align}
      \can_*(Y) \leq (d+\varepsilon)Area(Y)\ell.  \label{E: cancel IPI}
    \end{align}
\end{theorem}

\subsection{The proof of \cref{thm: relative non planar IPI}}
\label{S:cancelIPI_proof}

We adapt the arguments of \cite{ollivier_gafa, odr_nonplanar} to our setting. 
Let $Y$ be a corner labeled decorated abstract diagram.
Let $N$ be the number of equivalence classes of the partition of the $2$-cells of $Y$. We think of this number as of the number of distinct relators in $Y$.  
For $1\leq i\leq N$ let $m_i$ be the cardinality of the $i$-th equivalence class, that is, the number of times a relator $r_i$ appears in $Y$. 
Up to reordering, we may assume that $m_1\geq m_2 \geq \dots \geq m_N$.

Consider a factor vertex $v$ in $Y$.
{If $v$ is a corner, then every rotation element centered on $v$ is already determined because $Y$ is corner labeled. 
On the other hand, suppose that $v$ is not  a corner. Then $v$ is adjacent to at most two edges,  so it is contained in at most two edge triples. }
Suppose these edge triples are fully contained in faces $f_1, \dots, f_s$ bearing abstract relators $i_1, \dots, i_s$, and suppose that $v$ is the $k_j$-th vertex of face $f_j$.  

\begin{lemma}\label{L: using corners}
    If $Y$ is reduced, then, for every $1 \leq j < j' \leq s$ we have that $(i_j, k_j) \neq (i_{j'}, k_{j'})$.
\end{lemma}
\begin{proof}
    Suppose by contradiction that $1 \leq j < j' \leq s$ and that $(i_j, k_j) = (i_{j'}, k_{j'})$. Then $k_j=k_{j'}$ and $i_j=i_{j'}$. By definition $v$ is not a corner. Thus its degree is equal to $2$, and let $e_1$ and $e_2$ be the two edges adjacent to $v$. Also note that as the degree of $v$ is equal to $2$, the edge triple $(e_1,v,e_2)$ is fully contained in $f_j$ and $f_{j'}$. In particular, faces $f_j$ and $f_{j'}$ share edges $e_1$ and $e_2$.  As $Y$ is reduced, $f_j$ and $f_{j'}$ do not have opposite orientations. Hence $k_j = k_{j'}$ implies that $k_j = |\partial f_j|/2 = \ell$, so  $v$ is a corner. This is a contradiction.
\end{proof}

So we can define a strict lexicographic order on pairs $(i_j, k_j)$, and this gives a unique minimal index $j_{min}$ for each such edge triple. Note that this minimal index is the same for edge triples $(e, v, e')$ and $(e', v, e)$, so for a vertex $v$ of degree $2$ we can unambiguously associate $j_{min}$ with $v$. 
For the face $f_{j_{min}}$ there are no constraints on the rotation element at vertex $v$, however for every other $f_j$ containing $v$ the rotation element at $v$ is determined by the choice of rotation element in $f_{j_{min}}$ because of condition (2) in \Cref{def: fulfillable}.

\begin{defn}\label{def:constrains}
    Let $Y$ be a corner labeled decorated abstract diagram. Let $v$ be a factor vertex in $Y$ and let $f$ be a face of $Y$ containing $v$. Then $v$ \emph{constrains} $f$ if either $v$ is a corner or if $f \neq f_{j_{min}}$. Similarly, in this case, the edge triple at $v$ in $\partial f$ \emph{constrains} $f$.

Let $\delta(f)$ be the number of edge triples constraining a face $f$, let $F(Y)$ be the set of faces of $Y$, and let $F_i$ be the set of faces bearing abstract relator $r_i$.  
For $1 \leq i \leq N$, let $\kappa_i = \max_{F_i} \delta(f)$. 
\end{defn}

\begin{lemma}\label{eqn: local non planar bound}
    Let $Y$ be a reduced corner labeled non-filamentous decorated abstract diagram. Then
    \[
    \can_*(Y) \leq \sum_{f \in F(Y)} \delta(f)  \leq \sum_{i=1}^N m_i\kappa_i.
    \]
\end{lemma}

\begin{proof}
If an edge triple $(e, v, e')$ is partially contained in some face, then $\deg(v) \geq 3$, hence $v$ constraints any adjacent face. In this case, $(e, v, e')$ contributes $1/2$ to  $\can_*(Y)$ for all faces that fully contain $(e, v, e')$.   Otherwise, $(e,v,e')$ is fully contained in every face that contains either of its edges, and therefore $(e, v,e')$ contributes $1/2$ to $\can_*(Y)$ for all but one adjacent face. By definition there is at most one face that is not constrained by $v$. Summing over all edge triples $(e,v,e')$ yields the first inequality. The second inequality follows from the definitions.  
\end{proof}

We try to fulfill the abstract diagram $Y$ inductively, by picking a relator for each equivalence class of 2-cells in $Y$ in order. Note that at each step, when we fill in a face $f_j$ some of the edge triples may already  be determined either by the corner labeling of $Y$ or by a labeling of an edge triple in a previous step. These edge triples are precisely the edge triples that constrain face $f_j$.  We compute the probability that at each step, a randomly chosen relator agrees with all of the constraints on a given face. We first formalize this agreement condition.

\begin{defn}
Let $Y$ be a decorated abstract diagram, along with an ordering of the equivalence classes of the faces of $Y$, and let $\{w_1, \dots, w_q\}$ be an ordered set of words in $\mathcal S_{\ell}$. We assign the  word $w_j$ to each of the faces in the $j$-th equivalence class of $2$-cells in $Y$, and say that $\{w_1, \dots, w_q\}$  \emph{partially fulfills $Y$} if there is an assignment of a nontrivial rotation element in $\bigcup_{i=1}^n B_i(m)$ to each edge triple in the faces in the first to $q$-th equivalence class such that the following hold. 
\begin{enumerate}
       \item The rotation elements assigned to edge triples centered at a fixed factor vertex all are in the same factor group. 
    \item Each face in the $j$-th equivalence class bears the word $w_j$.
   \item The rotation element assigned to an edge triple $(e_1,v,e_2)$ is inverse to the rotation element assigned to $(e_2,v,e_1)$.
\end{enumerate}
\end{defn}

Let $Y$ be a corner labeled decorated abstract diagram. Recall that $\kappa_q$ was defined to be the maximal number of constraints for a face in the $q$-th equivalence class, see \Cref{def:constrains}; in other words, $\kappa_q$ is sum of the length of the segments that have already been labeled in the face with label $q$ for which this is maximized.

Let $f$ be the face labeled $q$ with a maximal number of constraints. Let us remove the constrained edge triples from the boundary of $f$. This set decomposes into finitely many maximal segments; let $\eta_q$ be its number. Let $\eta:=\max{\eta_q}$. 
\begin{remark}\label{rem: eta and boundedness}
    If the abstract diagram $Y$ is $(K,M)$-bounded, then $\eta\leqslant M$.
\end{remark}

\begin{lemma}\label{lem: local ipi inductive step} There is a constant $C>1$ such that the following holds. 
Let $Y$ be a reduced corner-labeled non-filamentous decorated abstract diagram.  For $1 \leq q \leq N$, let $p_q$ be the probability that $q$ words $w_1, w_2, \dots, w_q$ of free product length $\ell$ chosen at random partially fulfill $Y$ and let $p_0 = 1$. Then 
    \[
        \frac{p_q}{p_{q-1}} \leq C^{\eta+1}\lambda_{\mathcal{S}}^{-\kappa_q}.
    \]
\end{lemma}

To prove this lemma we first need estimates on the cardinality of the set $\mathcal S_{\ell}$. The following lemma is based on a standard counting argument that was explained to us by Eric Egge.

\begin{lemma}\label{lem: counting reduced words}
    Let $\mc{W}_\ell$ be the set of words of free-product length $\ell$ whose syllables are in $\bigcup_{i}B_i(m)$. There are constants  $\lambda_{\mathcal{S}}>1$ and $C>0$  such that for all $\ell>0$, 
    \[
        \frac 1C \lambda_\mathcal{S}^\ell \leq |\mathcal{W}_\ell| \leq C \lambda_\mathcal{S}^\ell.
    \]
\end{lemma}

    \begin{proof}
        Define a graph $\Upsilon$ with vertex set $B_1(m)\sqcup \cdots \sqcup B_n(m)$ and edges $(v, w)$ if and only if $v, w$ do not lie in a common ball $B_i(m)$. Let $A$ denote the adjacency matrix of $\Upsilon$. Since $A$ is symmetric, it is diagonalizable and has real eigenvalues. Denote the distinct eigenvalues by $\lambda_1, \dots, \lambda_k$, numbered so that $|\lambda_1| > |\lambda_2| > \cdots > |\lambda_k|$. 
        
       There is a one-to-one correspondence between paths of length $\ell-1$ in $\Upsilon$ and reduced words of free-product length $\ell$ whose syllables all lie in the balls $B_i(m)$, given by reading the vertices of the path.   Thus 
       $$|\mathcal W_{\ell}|=J^T A^{\ell-1}J,$$ where $J$ denotes the vector whose entries are all $1$. Since $A$ is diagonalizable, the entries of $A^{\ell-1}$ are linear combinations of $\lambda_1^{\ell-1}, \dots, \lambda_k^{\ell-1}$. Hence $|\mathcal{W}_\ell| = c_1 \lambda_1^{\ell-1} + \cdots + c_k \lambda_k^{\ell-1}$ for some real numbers $c_1, \dots, c_k$. These numbers do not depend on $\ell$. Note that at least one of the $c_i$ is non-zero, since $|\mathcal W_{\ell}|\geqslant 1$ for all $\ell>0$. For ease of notation we will assume that all the $c_i$ are non-zero; if they are not then we can simply discard the corresponding eigenvalues and renumber.  Similarly we may assume that $|\lambda_1|>0$. Since $|\lambda_1| > |\lambda_i|$ for all $i>1$, this implies that $c_1, \lambda_1, C_0$ are all positive. 
       
        Now take 
        $$c= \max_i\{|c_i|\}, \lambda_\mathcal{S} = \lambda_1, C_0=\max\left\{ \frac{c_1}{2\lambda_1}, \frac{kc}{\lambda_1}\right\}.$$         
        We then have 
        \[
        |\mathcal{W}_\ell| 
        \leq |c_1\lambda_1^{\ell-1}| + \cdots + |c_k \lambda_k^{\ell-1}|
        \leq kc'\lambda_1^{\ell-1}  
        \leqslant C_0 \lambda_\mathcal{S}^\ell.
        \]
    On the other hand, we also have 
        \begin{align*}
            |\mathcal{W}_\ell| 
            &= | c_1\lambda_1^{\ell-1} + \cdots c_k \lambda_k^{\ell-1}|  \\
            &\geq| |c_1\lambda_1^{\ell-1}| - |c_2\lambda_2^{\ell-1} + \cdots c_k \lambda_k^{\ell-1}|| .
        \end{align*}
Since $|\lambda_1|> |\lambda_i|$ for all $i>1$, we have that 
        \[
        \lim_{n \to \infty} \frac{|c_2\lambda_2^n + \cdots c_k \lambda_k^n|}{|c_1\lambda_1^n|} \leq \lim_{n\to \infty}\frac{|c_2\lambda_2^n| + \cdots |c_k \lambda_k^n|}{|c_1\lambda_1^n|}  = 0.
        \]
    In particular, there is $N>0$ so that for $n>N$, we have $|c_2\lambda_2^n + \cdots c_k \lambda_k^n| < \frac{1}{2}|c_1\lambda_1^n|.$
    And since $c_1, \lambda_1>0$, for $\ell-1>N$ this gives us
        \begin{align*}
            |\mathcal{W}_\ell| 
                &\geq | |c_1\lambda_1^{\ell-1}| - |c_2\lambda_2^{\ell-1} + \cdots c_k \lambda_k^{\ell-1}|| \\
                &\geq c_1\lambda_1^{\ell-1} - \frac{1}{2}c_1\lambda_1^{\ell-1} = \frac{c_1}{2}\lambda_1^{\ell-1}\geq \frac 1{C_0} \lambda_\mathcal{S}^\ell.
        \end{align*} 

Finally fix $C = \max_{\ell \leq N}\{C_0, \frac{\lambda_{\mathcal{S}}^\ell}{|\mathcal{W}_\ell|}\}$.

Since the number of reduced words of free-product length $2\ell$ whose syllables are in $B_1(m) \cup B_2(m)$ is $2(b_1b_2)^{\ell}>4^\ell$,  and since this is a lower bound for $|\mathcal{W}_{2\ell}|$, we must have $\lambda_\mathcal{S}>1$. 
    \end{proof}

Define $\mathcal{W}^{i, \cdot}_\ell$ to be the set of reduced words in $\mathcal{W}_\ell$ whose first syllable is not in $B_i(m)$. Define $\mathcal{W}^{\cdot, j}_\ell$ analogously, and let $\mathcal{W}^{i, j}$ be the set of reduced words in $\mathcal{W}_\ell$ whose first syllable is not in $B_i(m)$ and whose final syllable is not in $B_j(m)$. Note that $|\mathcal{W}^{i, j}_\ell| = |\mathcal{W}^{j, i}_\ell|$ and $|\mathcal{W}^{i, \cdot}_\ell| = |\mathcal{W}^{\cdot, i}_\ell|.$

\begin{cor} \label{cor: counting more words}
Let $\lambda_\mc{S}$ be as defined in \Cref{lem: counting reduced words}. If $n>3$, then for any $1\leq i, j\leq n$ we have $|\mathcal{W}^{i, j}_\ell| \asymp |\mathcal{W}^{i, \cdot}_\ell| \asymp |\mathcal{S}_\ell| \asymp \lambda_{\mc{S}}^\ell$. In particular, there is a constant $C>1$ so that for any set $\mathcal{A} \in \{\mathcal{W}_\ell, \mathcal{S}_\ell,\mathcal{W}^{i, \cdot}_\ell, \mathcal{W}^{i, j}_\ell \}$ we have 
    \[
        C^{-1}\lambda_\mathcal{S}^\ell\leq |\mathcal{A}|\leq C\lambda_\mathcal{S}^\ell.
    \]

If $n = 2$, there is a constant $C>1$ so that we have $C^{-1}\lambda_\mathcal{S}^{\ell} \leq \mathcal{S}_{\ell} \leq C \lambda_\mathcal{S}^{\ell}$ when $\ell$ is even, and for all $\ell\in \mathbb{N}$ we have  $|\mathcal{A}|\leq C\lambda_{\mathcal{S}}^\ell$ for all $\mathcal{A} \in \{\mathcal{W}_\ell, \mathcal{S}_\ell,\mathcal{W}^{i, \cdot}_\ell, \mathcal{W}^{i, j}_\ell \}$. 
\end{cor}

    \begin{proof}
        First suppose that $n\geq 3.$ For every word $w \in \mathcal{W}_{\ell-1}$ there is some (non-unique) element $g \in \bigcup_{i\neq j} B_i(m)$ so that $gw \in \mathcal{W}^{j, \cdot}_\ell$. Therefore we have $|\mathcal{W}^{j, \cdot}_\ell| \leq |\mathcal{W}_\ell| \leq C_2 \lambda_{\mathcal{S}}^\ell$, and on the other hand
            \[
            |\mathcal{W}^{j, \cdot}_\ell| \geq |\mathcal{W}_{\ell-1}| \geq C_1\lambda_{\mathcal{S}}^{\ell-1} = \frac{C_1}{\lambda_\mathcal{S}} \lambda_\mathcal{S}^\ell.
            \]
        This proves that $|\mathcal{W}^{j, \cdot}_\ell|\asymp \lambda_{S}^\ell$.

        The estimates on $|\mathcal{W}_\ell^{i, j}|$ and $ |\mathcal{S}_\ell|$ follow from analogous arguments. Indeed, for every word $w \in \mathcal{W}_{\ell-2}$ there exist $g, g' \in \bigcup_i B_i(m)$ so that $gwg' \in \mathcal{W}_{\ell}^{i, j}$, and there exist $h, h' \in \bigcup_i B_i(m)$ so that $hwh' \in \mathcal{S}_\ell$.

        Now suppose that $n = 2$. Every word of even length is cyclically reduced, so when $\ell$ is even    $|\mathcal{S}_{\ell}| = |\mathcal{W}_{\ell}|\asymp \lambda_{\mathcal{S}}^\ell$. Furthermore we can compute $\lambda_\mc{S}$ directly. Indeed, if $b_1 = |B_1(m)|, b_2 = |B_2(m)|$ then for even $\ell$, $|\mc{S}_\ell| = 2(b_1b_2)^{\ell/2}$, so $\lambda_\mc{S} = \sqrt{b_1b_2}.$ There is a bijection between $\mathcal{W}^{1, \cdot}$ and $\mathcal{W}^{2, \cdot}$ given by reversing any word. So $|\mathcal{W}^{i, \cdot}_\ell| = \frac{1}{2}|\mathcal{W}_\ell|$. Since every word in $\mathcal{W}^{i, \cdot}_{\ell}$ is in $\mathcal{W}^{i, j}_\ell$ for $j \neq i$, we also have $|\mathcal{W}^{1, 2}_{\ell}|  = \frac{1}{2}|\mathcal{W}_{\ell}|$, while $|\mathcal{W}^{1, 1}_{\ell}| = |\mathcal{W}^{2, 2}_{\ell}|=  0.$ 
        
        On the other hand, if $\ell$ is odd then every word in $\mathcal{W}^{i, \cdot}_\ell$ is in $\mathcal{W}_\ell^{i, i}$, so $|\mathcal{W}^{i, i}_\ell| = |\mathcal{W}^{i,\cdot}_\ell|$ and $|\mathcal{W}^{1, 2}_\ell| = 0$. We can compute directly that $|\mathcal{W}_\ell| = (b_1b_2)^{(\ell-1)/2}(b_1+b_2)$, and $|\mathcal{W}_\ell^{1, \cdot}| = b_1(b_1b_2)^{(\ell-1)/2}, |\mathcal{W}_\ell^{2, \cdot}| = b_2 (b_1b_2)^{(\ell-1)/2}$. This completes the proof.
    \end{proof}

\begin{proof}[Proof of \Cref{lem: local ipi inductive step}]
    Suppose that $Y$ is partially fulfilled by words $w_1,w_2,\ldots, w_{q-1}$ in $\mathcal S_\ell$. Let us estimate the number of words $w_q\in \mathcal S_\ell$  such that $w_1,w_2,\ldots w_{q-1},w_q$ partially fulfills $Y$. Let $f$ be the face labeled $q$ with a maximal number of constraints. Let $\eta_q$ be the number of maximal edge subpaths in $\partial f$ which do not contain any  edge triples $(e_1,v,e_2)$ constraining $f$; that is they do not contain $v$ and the edges $e_1$ and $e_2$. 
    Let $l_1,\ldots, l_{\eta_q}$ be the number of edge triples contained in these maximal subpaths. Note that $l_1+l_2+\ldots +l_{\eta_q}=\ell - \kappa_q$. By \cref{cor: counting more words}, there are at most $C \lambda_S^{l_i}$ ways to assign  the rotation elements to the edge triples in the $i$-th such subpath. 
    Overall the number of words $w_q$ as above is at most $C^{\eta} \lambda_S^{l_1+l_2+\ldots +l_{\eta_q}}=C^{\eta}\lambda_S^{\ell-\kappa_q}$.

    On the other hand, there are at least $C^{-1}\lambda_S^{\ell}$ words in $\mathcal S_\ell$. Thus
    $$\frac{p_q}{p_{q-1}} =P(\text{$w_q$ partially fulfills $Y$}\mid \text{$w_1, w_2, \dots, w_{q-1}$ partially fulfill $Y$})\leqslant C^{\eta+1}\lambda_{\mathcal{S}}^{-\kappa_q}.\qedhere $$
    \end{proof}

\begin{prop}\label{prop: local rel hyp} 
Let $G \sim \mathcal{FPD}(\mc{G}; d, m, \ell)$. Then there is a constant $C>1$  such that for every corner-labeled non-filamentous reduced decorated abstract  diagram $Y$ the following holds. Let $\varepsilon >0$. Then either $\can_*(Y) < (d+\varepsilon)Area(Y)\ell,$ or the probability that $Y$  fulfills $\mathcal R$ is less than $C^{\eta+N+1}\lambda_{S}^{-\varepsilon\ell}.$
\end{prop}

\cref{prop: local rel hyp} brings \cite[Proposition 1.8]{odr_nonplanar} to our setting. Up to notation, the proof that we give below is the same as the proof in \cite{odr_nonplanar} for non-planar diagrams, or of \cite[Lemma 59]{ollivier_invitation_2005} for disk diagrams, see also \cite[Section 2.2]{ollivier_gafa}.

\begin{proof}
For $1\leq q\leq N$, let $P_q$ be the probability that there exists a $q$-tuple of words partially fulfilling $Y$ in $G$. Then 
    \[
        P_q \leq |\mc{R}|^q p_q  = \mc{S}_\ell^{dq }p_q  \leq C^{dq}\lambda_{\mathcal{S}}^{q d\ell}p_q \leq C^N \lambda_{\mathcal{S}}^{q d \ell}p_q,
    \]
where $\mathcal{S}_\ell$ denotes the number of cyclically reduced words of free product length $\ell$ with syllables in $\bigcup_{i=1}^n B_i(m)$ and $\lambda_{\mathcal{S}}, C $ are as defined in \Cref{cor: counting more words}. Hence 
\[
\log_{\lambda_{\mathcal{S}}}(P_q) \leq qd\ell + \log_{\lambda_{\mathcal{S}}}(p_q) + \log_{\lambda_{\mathcal{S}}}(C^{N}).
\]

\Cref{lem: local ipi inductive step} implies that $\kappa_q \leq \log_{\lambda_{\mathcal{S}}}(p_{q-1}) - \log_{\lambda_{\mathcal{S}}}(p_q) + \log_{\lambda_{\mathcal{S}}} (C^{\eta+1})$. Since $\sum_{q=1}^N m_q=Area (Y)$,  \Cref{eqn: local non planar bound} gives us
    \begin{align*}
        \can_*(Y) &\leq \sum_{q =1}^N m_q (\log_{ \lambda_{\mathcal{S}}}p_{q -1} - \log_{\lambda_{\mathcal{S}}}p_q +\log_{\lambda_{\mathcal{S}}}(C^{\eta+1}))\\
        &= \sum_{q =1}^{N-1}\left((m_{q +1} - m_q ) \log_{\lambda_{\mathcal{S}}}p_q \right) - m_N\log_{\lambda_{\mathcal{S}}}p_N + m_1 \log_{\lambda_{\mathcal{S}}}p_0\\
         & \qquad \qquad \qquad \qquad + Area(Y)\left(\log_{\lambda_{\mathcal{S}}}(C^{\eta+1})+\log_{\lambda_S}(C^{N})\right).
    \end{align*}
    
Since $p_0 = 1$ and $m_{q +1} - m_q \leq 0$, we have 
    \begin{align*}
       \can_*(Y)  & \leq \sum_{q =1}^{N-1}(m_{q +1}-m_q )(\log_{\lambda_{\mathcal{S}}}P_q  - q d\ell) - m_N(\log_{\lambda_{\mathcal{S}}} P_N - Nd\ell)\\
     & \qquad \qquad \qquad \qquad \qquad \qquad \qquad  + Area(Y)\log_{\lambda_{\mathcal{S}}}(C^{\eta+N+1}).  
    \end{align*}
    
Since $\sum_{q =1}^{N-1}(m_{q }-m_{q +1})q d\ell + m_N Nd\ell = d\ell\sum_{q =1}^N m_q  = d\ell Area(Y),$ we have 
  \begin{align*}
      \can_*(Y)& \leq d\ell Area(Y) + \sum_{q =1}^{N-1}(m_{q +1}-m_q )\log_{\lambda_{\mathcal{S}}}P_q  - m_N\log_{\lambda_{\mathcal{S}}}P_N\\
         & \qquad \qquad \qquad \qquad \qquad \qquad \qquad   + Area(Y)\log_{\lambda_{\mathcal{S}}}(C^{\eta+N+1}).
  \end{align*}
     Note that $P_N = \min{P_q }$. Since $(m_{q +1}- m_q ) \leq 0$ and $m_1 \leq Area(Y)$ we get     
    \begin{align*}
    \can_*(Y) &\leq (d\ell+\log_{\lambda_{\mathcal{S}}}(C^{\eta+N+1})) Area(Y) + \log_{\lambda_{\mathcal{S}}}P_N\left(\sum_{q =1}^{N-1}(m_{q +1}-m_q ) - m_N\right)\\
    &= (d\ell+\log_{\lambda_{\mathcal{S}}}(C^{\eta+N+1}))Area(Y) - m_1\log_{\lambda_{\mathcal{S}}}P_N \\
    &\leq Area(Y)(d\ell - \log_{\lambda_{\mathcal{S}}}P_N+\log_{\lambda_{\mathcal{S}}}(C^{\eta+N+1})).
    \end{align*}
    We conclude that 
    \[
    \mathbb{P}(Y \mbox{is fulfillable}) \leq P_N \leq C^{\eta+N+1}\lambda_{\mathcal{S}}^{\frac{Area(Y)d\ell - \can_*(Y)}{Area(Y)}}.
    \]

Therefore, if $\can_*(Y) \geq (d+\varepsilon)Area(Y)\ell$ we have
    \[
        \mathbb{P}(Y \mbox{is fulfillable}) \leq C^{\eta+N+1}\lambda_{\mathcal{S}}^{-\varepsilon\ell}. \qedhere
    \]
\end{proof}

At this point we are ready to prove \cref{thm: relative non planar IPI}.

\begin{proof}[Proof of \cref{thm: relative non planar IPI}]
Let $Y \to X_\mc{R}$ be a reduced diagram over $\mc{R}$.  By \cref{rem:fulfillable} there is an underlying reduced decorated abstract diagram $Y$ which is fulfillable. Furthermore, we fix the corner labeling on $Y$ induced by the decoration of the diagram $Y \to X_\mc{R}$.   
Then we call $Y$ the \emph{underlying corner-labeled decorated abstract diagram}. 

Let $\mathcal Y$ be the set of corner-labeled decorated abstract diagrams underlying the reduced non-filamentous $(K,M)$-bounded diagrams over $\mc R$. Note that each diagram in $\mathcal Y$ is a fulfillable reduced corner-labeled decorated abstract non-filamentous $(K,M)$-bounded diagram.   

 Let $P_0$ be the probability that a given reduced corner-labeled decorated abstract non-filamentous $(K,M)$-bounded diagram that violates inequality \eqref{E: cancel IPI} fulfills $\mathcal R$.  By \cref{prop: local rel hyp} and \cref{rem: eta and boundedness}, $P_0\leq C^{M+K+1}{\lambda_{\mathcal{S}}}^{-\varepsilon\ell}$.
 Let $P_1$ be the probability that a given reduced corner-labeled decorated abstract non-filamentous  $(K,M)$-bounded  diagram that violates inequality \eqref{E: cancel IPI} is in $\mathcal Y$. As every decorated abstract diagram in $\mathcal Y$ fulfills $\mathcal R$,  $P_1\leq P_0$. Thus $P_1\leq C^{M+K+1}{\lambda_{\mathcal{S}}}^{-\varepsilon\ell}$.

 Let $D(K, M, \ell)$ be the number of reduced decorated abstract corner-labeled non-filamentous diagrams that are $(K, M)$-bounded. 
 Thus, the probability that $\mathcal Y$ contains a diagram that violates inequality \eqref{E: cancel IPI} is at most $D(K,M,\ell)P_1$. As $D(K, M, \ell)$ is polynomial in $\ell$ by \cref{lemma:polynomial_diagrams}, this probability tends to $0$ when $\ell \to \infty$.   
\end{proof}

\subsection{A Return to Planar Isoperimetry}

The main goal of this subsection is to relate relative cancellation to standard cancellation and isoperimetry of disk diagrams. 

\begin{lemma}\label{lem: rel and stand cancellation}
Let $Y$ be a decorated abstract diagram. Then 
    \[
        \can(Y) \leq 2\can_*(Y).
    \]
\end{lemma}

\begin{proof} Recall that $\mathcal E$ the set of edge triples of non-zero relative degree from \Cref{def: rel degree}, and let $E$ be the set of edges of $Y$ of non-zero degree. These sets are each used to count 2-cells when computing $\can_*(Y)$ and $\can(Y)$ respectively. 
Partition these sets as follows. 
Let 
$$\mathcal E_0:=\{(e,v,e')\in \mathcal E \mid \text{ there is no $2$-cell that partially contains }(e,v,e')\},$$
and let $\mathcal E_1=\mathcal E\setminus \mathcal E_0$. Note that if $(e,v,e')\in \mathcal E_0$, there still may be a $2$-cell that fully contains $(e,v,e')$. 

Similarly, let 
$$E_0:=\{e\in E \mid  \text{$\exists$  $v\in V$ and $e'\in E$ at $v$ such that $(e,v,e')\in \mathcal E_0$}\},$$  and let $E_1=E\setminus E_0$. 

Whether an edge triple belongs to $\mc E_0$ or $\mc E_1$ puts restrictions on the edges of the triple.  
Here we write $e \sim v$ when and edge $e$ is incident to a vertex $v$. 

By definition, if $(e,v,e')\in \mathcal E_0$, then both $e, e'\in E_0$. 
On the other hand, if $(e,v,e')\in \mathcal E_1$ there are more possibilities. It can be that $e\in E_0$ and $e'\in E_1$. However, in this case, there is no $2$-cell that fully contains $(e,v,e')$.  Indeed, suppose to the contrary that $c$ is a 2-cell fully containing $(e,v,e')$. Since $e\in E_0$, there is an edge $e''$ such that $(e,w,e'') \in \mc E_0$ where $e \sim w$. 
If $e'\not =e''$, then $c$ partially contains $(e, w, e'')$ contradicting the definition of $\mc E_0$. Thus $e'=e''$. But then $(e,v,e')\in \mathcal E_0$ by the above argument, also yielding a contradiction.  
This shows that if $(e,v,e') \in \mc E_1$ is fully contained in a 2-cell then both $e,e' \in E_1$. 
In particular, if an edge $e \in E_1$ is adjacent to a 2-cell $c$ and $e \sim v$ then there is another edge $e' \sim v$ such that $(e,v,e')$ is fully contained in $c$.

We will first use these partitions to relate degree and relative degree with respect to a fixed factor vertex $v$. 

If $(e,v,e')\in \mathcal E_0$ then 
$\deg_*(e,v,e') =\deg (e).$ Indeed, every $2$-cell adjacent to $e$ fully contains the triple, hence is also adjacent to $e'$. Thus,
\[
    \sum_{\substack{e,e' \in v\\(e,v,e')\in \mathcal E_0}}  (\deg_*(e,v,e')-1) = \sum_{\substack{e\in v\\e\in E_0}}(\deg(e)-1).
\]

On the other hand, if $(e,v,e')\in \mathcal E_1$, then 
$$\deg_*(e,v,e') - 1 =\# \{2\text{-cells fully containing }(e,v,e')\}.$$
Summing over all triples including $v$ double counts $2$-cells. We claim that 
 $$\sum_{\substack{e,e'\sim v\\(e,v,e')\in \mathcal E_1}}(\deg_*(e,v,e')-1) = \sum_{\substack{e \sim v\\e \in E_1}} \deg(e)
 \geqslant \sum_{\substack{e \sim v\\e \in E_1}} (\deg(e) - 1).$$
Indeed, we have seen that 2-cells fully containing triples in $\mc E_1$ and through $v$ may only pass through adjacent edges from $E_1$. On the other hand, 2-cells adjacent to edges in $E_1$ and incident to $v$ always fully contain a triple of the form $(e,v,e')$.

To conclude let $V_f(Y)$ be the collection of all factor vertices of $Y$. We combine both computations and sum over $V_f(Y)$ to obtain that
\begin{align*}
    2\can_*(Y) &= \sum_{v \in V_f(Y)}\sum_{\substack{e,e'\sim v\\(e,v,e')\in \mathcal E}}(\deg_*(e,v,e')-1) \\
    &\geqslant \sum_{v \in V_f(Y)}\sum_{\substack{e \sim v\\e \in E}} (\deg(e) - 1) = 2\can(Y).
\end{align*}
Note that the last equality in the equation above holds because the 0-skeleton of $Y$ is bipartite with $V_f(Y)$ being one of the bipartitions.
\end{proof}

Note any non-filamentous disk diagram of bounded size must also have a bounded number of connectors (see \Cref{def: connector}).

\begin{lemma}\label{ex:disks are KM bounded}
If $Y$ is a non-filamentous disk diagram with at most $K$ 2-cells, then $Y$ is $(K, \frac{1}{2}K(K-1)^2 + K^2)$ bounded.
\end{lemma}
\begin{proof}
Suppose that two 2-cells $C_1$ and $C_2$ in $Y$ have non-empty intersection. As $Y$ is planar,  $C_1\cap C_2$ splits as a union of disjoint connectors. Suppose that $C_1 \cap C_2$ is composed of $m$ connectors. Then there are at least $m-1$ disjoint cycles in $\partial C_1 \cup \partial C_2$. Since $Y$ is simply connected, each of these cycles is filled with a sub-disk diagram containing at least one 2-cell. Thus $m\leq K-1$. Therefore the total number of connectors in the interior of $Y$ is at most $\binom{K}{2} (K-1)$. Since $Y$ is simply connected, every 2-cell of $Y$ has at most $K$ connectors on $\partial Y$, so the total number of connectors in $Y$ is at most $\binom{K}{2} (K-1) + K^2.$
\end{proof}

By applying \Cref{thm: relative non planar IPI} we get bounds on (standard) cancellation and isoperimetric inequalities with high probability. Note that although our words have (syllable) length $\ell$, the 2-cells in $X_\mc{R}$ have boundary length $2\ell$. 

\begin{theorem} \label{thm: cancel implies planar}
Let $G \sim \mathcal{FPD}(\mc{G}; d, m, \ell)$. The following statements hold with overwhelming probability. For every reduced $(K, M)$-bounded diagram $Y$ over $\mc R$, 
    \begin{equation}
        \can(Y)\leq d Area(Y)2\ell.\label{Eq: cancelbound}
    \end{equation}
Furthermore, for a reduced disk diagram $D$ over $\mc R$ of area at most $K$ we have 
    \[
        |\partial D| \geq (1-2d)2\ell Area(D).
    \]
\end{theorem}

\begin{proof} Let $Y\to X_{\mc R}$ be a reduced $(K, M)$-bounded diagram $Y$ over $\mc R$.  By \Cref{thm: relative non planar IPI} applied to $Y$  and by \Cref{lem: rel and stand cancellation}, with overwhelming probability, the first claim \eqref{Eq: cancelbound} holds. Indeed, we may remove filaments from $Y$ as these do not contribute to $\can(Y)$. This splits $Y$ into smaller diagrams, each of which satisfies the required inequality. This yields the claim.  

For the second part of the theorem, let $D$ be a reduced disk diagram over $\mathcal R$ of area at most $K$. Note first that $2\can(D) + |\partial D| = 2\ell Area(D)$. 
Further, by \cref{ex:disks are KM bounded}, there exists $M$ such that $D$ is $(K,M)$-bounded, and therefore by the first part of this theorem $\can(D) \leq 2\ell d Area(D)$.
Putting these together, we get $|\partial D| \geq (1-2d)2\ell Area(D)$, with overwhelming probability.
\end{proof}

\subsection{A Non-Planar Variant of Greendlinger's Lemma}\label{subsec: greendlinger}

We prove a non-planar version of Greendlinger's Lemma which we use in \Cref{subsec:embedded-hypergraphs}. 
The proof follows the same argument as the planar case, as proved in \cite{oll_somesmall}, together with the following remark.

Let $Y$ be a $2$ complex. 
An edge $e$ of $Y$ is \emph{external} if $\deg(e) \leq 1$. All other edges are \emph{internal}. For any 2-cell $h$ in $Y$, let $I_h$ denote the set of internal edges in $h$.

\begin{remark}\label{rmk: canbound}
For every $2$-complex $Y$, we have
    \[
        \can(Y) = \sum_{h \in Y^{(2)}}\sum_{e \in h^{(1)}} \frac{\deg(e) - 1}{\deg(e)} \geq \frac{1}{2} \left(\sum_{h \in Y^{(2)}} |I_h|\right),
    \]
with equality exactly when every internal edge has degree 2. In particular, in every disk diagram equality holds.
\end{remark}

Greendlinger's Lemma has traditionally been used to show that in a disk diagram $D$ there must exist some $2$-cells which contribute at least half of their boundary edges to $\partial D$. We intend to show something slightly stronger. 

\begin{defn}\label{def: supershell}
    Let $\epsilon>0$ and suppose that $Y$ is a 2-complex in which every $2$-cell has boundary length $L$. An \emph{$\epsilon$-supershell} of $Y$ is a $2$-cell which as at least $L(\frac{1}{2}+\epsilon)$ external edges.
\end{defn}

\begin{lemma}\label{lem: greendlinger} Let $L>0$. 
Let $Y$ be a 2-complex with at least two 2-cells, such that each of its subcomplexes $Y'$ satisfies $\can(Y')<dL Area(Y')$ and all $2$-cells of $Y$ have boundary length $L$. Then there are at least two 2-cells of $Y$ each with at least $L(1-5d/2)$ external edges. Hence, for $0<\epsilon<\frac{1}{2}-\frac{5d}{2}$, $Y$ contains at least two $\epsilon$-supershells.
\end{lemma}

\begin{proof} 
Let $f$ be a 2-cell of $Y$ with the maximal number of external edges, and let $\alpha L$ be the number of external edges in $f$. Then $\alpha' L =L - \alpha L$ is the number of internal edges in $f$. Similarly, let $h$ be a 2-cell of $T$ with the second largest number of external edges and let $\beta L$ be the number of external edges in $h$. Notes that we may have $\alpha = \beta$. For any other $2$-cell $g \neq f, g \neq h$ in $Y$ the number of internal edges in $g$ is at least $\beta' L = 1 - \beta L$.

We want to prove that $\beta \geq 1-5d/2$, so assume for contradiction that $\beta < 1- 5d/2,$ or, equivalently, that $\beta' > 5d/2$.

By \cref{rmk: canbound}, 
    \begin{align}\label{E:greendlinger1}
    \begin{split}
        \can(Y) &\geq \frac{1}{2}\left( \sum_{h \in Y^{(2)}} |I_h|\right) \\
            &\geq\frac{1}{2} \left(\alpha' L + \sum_{g \in Y^{(2)}, g \neq f} \beta' L \right)\\
            &\geq \frac{1}{2} \left( \alpha' L + \beta' L(Area(Y)-1) \right).
    \end{split}
    \end{align}

Consider the 2-complex $Y'$ obtained by removing the face $f$ from $Y$, so $Area(Y) = Area(Y')+1$. Then $\can(Y) = \can(Y') + |Y' \cap f| = \can(Y') + \alpha' L$, so  

    \begin{align}\label{E:greendlinger2}
    \begin{split}
        \can(Y') &= \can(Y) - \alpha' L \\
            &\geq \frac{1}{2}\left(-\alpha' L + \beta' L(Area(Y)-1)\right)\\
            &= \frac{1}{2}\left(-\alpha' L + \beta' L Area(Y')\right).
    \end{split}
    \end{align}

There are two cases. First suppose that $\alpha'\geq d$. Since $d L Area(Y) \geq \can(Y),$ \eqref{E:greendlinger1} gives us 
$$2d(Area(Y')+1)=2d Area(Y) \geq \alpha' + \beta' Area(Y').$$ Since $\beta'>5d/2$, we know that $\beta'-2d> d/2>0$. So we get 
    \[
        Area(Y') \leq \frac{2d-\alpha'}{\beta'-2d}\leq \frac{2(2d-\alpha')}{d}.
    \]
Therefore $Area(Y')<2$, so $Area(Y)\leq 2$.

Now suppose that $\alpha'<  d.$ Since $d L Area(Y') > \can(Y')$, \eqref{E:greendlinger2} gives us $2d Area(Y') \geq -\alpha' + \beta'Area(Y').$ Since $\beta'-2d>d/2$ we get
    \[
        Area(Y') \leq \frac{\alpha'}{\beta'-2d} \leq \frac{2\alpha'}{d}.
    \]
Therefore $Area(Y) \leq 2.$ 

We have assumed that $Area(Y)\geq 2$. Therefore in either case, we can conclude that $Area(Y) = 2.$ Let $f, g$ the 2-cells of $Y$. Then $\can(Y) = |f\cap g| \leq 2d L$, so both $f$ and $g$ have at least $L - 2d L> L - 5d L/2$ external edges, as desired.
\end{proof}

\section{Global Isoperimetry and (Relative) Hyperbolicity}\label{sec: global isoperimetry} 

The main goal of this section is to prove that random quotients are relatively hyperbolic. 

\subsection{Global Isoperimetry} 

In this subsection we show that the local linear isoperimetric inequality of \Cref{thm: cancel implies planar} implies a global isoperimetric inequality. In particular, we prove the following (compare to \cite[Theorem 8]{oll_somesmall}).

\begin{theorem}
\label{thm: local-to-global}
Let $X$ be a 2-complex that is simply connected such that every 2-cell 
in $X$ has boundary length equal to $L$. 
Let $C > 0$ and $\eps > 0$.  
Let $P$ be a property of disk diagrams that is preserved by taking subdiagrams. Suppose that for some $K \geq 10^{50} \eps^{-2}C^{-3}$ any disk diagram $D$ with property P of area at most $K$ satisfies
\[
|\boundary D| \geq C L Area(D).
\]
Then any disk diagram $D$ with property P in $X$ satisfies
\[
|\boundary D| \geq (C - \eps) L Area(D).
\]
\end{theorem}

An important example of a property of disk diagrams that is preserved by taking subdiagrams is that of being a reduced van Kampen diagram over a group presentation, or being a reduced  diagram over $\mathcal R$ in the complex $X_{\mc R}$ of this paper. 

Together with the local isoperimetric inequality \Cref{thm: cancel implies planar}, \cref{thm: local-to-global} implies the following.

\begin{cor}\label{thm: global rel IPI}
Let $d < \frac{1}{2}$ and $\epsilon > 0$. Let $L = 2\ell.$  With overwhelming probability, any reduced disk diagram over $\mc{R}$ of $G \sim \mathcal{FPD}(\mc{G}; d, m, \ell)$ satisfies 

\begin{equation}
|\boundary D| \geq (1-2d-\epsilon) L Area(D). \label{Eq: global IPI}
\end{equation}
\end{cor}

\begin{proof}
By \cref{thm: cancel implies planar}, for $d < \frac{1}{2}$ and $K \geq 10^{50}\epsilon^{-2}(1-2d)^{-3}$, with overwhelming probability, the complex $X_{\mc{R}}$ satisfies a linear isoperimetric inequality for disk diagrams of area at most $K$. The result now follows from \cref{thm: local-to-global}. 
\end{proof}

To prove \Cref{thm: local-to-global} we closely follow the strategy of \cite{oll_somesmall}, and we encourage the careful reader to refer to that paper. In fact, the statement of \cite[Theorem 8]{oll_somesmall} is identical to the statement \Cref{thm: local-to-global} if we replace the space $X$ with the Cayley complex of a finitely presented group in which every relator has length equal to $L$. We argue that although the results of Ollivier are stated for the Cayley complex of finitely presented groups, they also apply to any simply connected 2-dimensional polygonal complex $X$ all of whose $2$-cells have the same boundary length.  

\begin{remark}
The results of Ollivier are written with different notation than we use in this paper. For clarity within this paper, we translate Ollivier's notation into our own. We record here the relevant dictionary for those who wish to check the Ollivier results: $A_c(D) = L^2 Area(D), L_c(D) = |\partial D|, \mathcal{A}(D) = L Area(D)$. 
\end{remark}
  
The following is an adaptation of a result of Papasoglu's \cite{papasoglu} for simplicial complexes.

\begin{proposition}[{\cite[Proposition 42]{ollivier_gafa}}]
\label{prop:papasoglu }
Let $X$ be a simply connected 2-complex such that each face of $X$ has exactly $L$ edges.  Let P be a property of disk diagrams that is preserved by taking subdiagrams.

Suppose that for some integer $K \geq 10^{10}$, any disk diagram $D$ in $X$ with property P and area between $\frac{K^2}{4}$ and $480K^2$ satisfies
    \[ |\boundary D|^2 \geq 2\cdot 10^{14}L^2 Area(D). \]
Then any disk diagram $D$ with property P in $X$ with area at least $K^2$ satisfies
    \[|\boundary D| \geq \frac{Area(D)L}{10^4K}.\]
\end{proposition}

We now apply a double induction argument, as in \cite{oll_somesmall}. Note that the statements in \cite{oll_somesmall} are written in terms of van Kampen diagrams for finitely presented groups; however, they apply equally well to abstract diagrams of 2-dimensional simplicial complexes in which there are at most finitely many possible boundary lengths of 2-cells. In particular, they apply in the setting of $X_\mc{R}$.  We sketch the argument here, with references to \cite{oll_somesmall} for details. 

For the base case, note that \Cref{prop:papasoglu } implies that if $|\partial D| \geq C_1 Area(D) L$ for any diagram $D$ satisfying P with $Area(D)\geq 10^{23} C_1^{-2}$, then any diagram satisfying P must satisfy $|\partial D|\geq C_2 Area(D) L$, where $C_2 = C/10^{15}$ (see \cite[Proposition 10]{oll_somesmall} for details). 

The following immediately implies \Cref{thm: local-to-global}.

\begin{proposition}[{\cite[Proposition 13]{oll_somesmall}}]
\label{prop: inductive step local to global}
Suppose that $X$ is a simply connected 2-complex so that every 2-cell has boundary length equal to $L$. Let P be a property of disk diagrams that is preserved by taking subdiagrams. Fix $\varepsilon>0$. Let $C, C'>0$ Suppose that all disk diagrams $D$ in $X$ with property P satisfy 
	\[
		|\partial D|\geq C' Area(D)L, 
	\]
and for some $K\geq 50/\varepsilon^2C'^3$, every diagram $D$ with property P satisfying $Area(D)\leq K$ also satisfies
	\[
		|\partial D| \geq C Area(D) L.
	\]
Then any diagram $D$ with property P satisfies
	\[
		|\partial D| \geq (C-14\varepsilon)Area(D)L.
	\]
\end{proposition}
\begin{proof}[Proof Sketch]
Let $C, C', \varepsilon>0$. Suppose that any disk diagram $D$ with P satisfies $|\partial D| \geq C' Area(D)L$, and that for some $A \geq 50/(\varepsilon C')^2$ if $|\partial D|\leq A$ then $|\partial D|\geq C Area(D) L$. By \cite[Lemma 11]{oll_somesmall} there exists a decomposition $D = D_1 \cup D_2$ so that $D_1, D_2$ have no 2-cells in common, $D_1, D_2$ both contain at least a quarter of $\partial D$, and $D_1 \cap D_2$ has length at most $L + \frac{2L}{C'}\log(Area(D))$. We can use this to show that for any diagram with $|\partial D|\leq \frac{7}{6}A L$, we have $|\partial D|\geq (C-\varepsilon)Area(D)L$ (see \cite[Proposition 12]{oll_somesmall}).   \Cref{prop: inductive step local to global} follows inductively (see \cite[Proposition 13]{oll_somesmall}. 
\end{proof}

\subsection{Relative Hyperbolicity}
We use the linear isoperimetric inequality in \Cref{thm: global rel IPI} and the fact that $\mc{R}$ is finite to prove that $X_\mc{R}$ is hyperbolic and $G$ is relatively hyperbolic.

\begin{cor}\label{cor: relators embed}
    Let $d<1/4$. With overwhelming probability there is no embedded closed path in $X_\mc{R}^{(1)}$ of length $<2\ell$. In particular, the boundary path of each 2-cell is embedded. 
\end{cor}
    \begin{proof}
        Suppose that there is an embedded closed path of length $<2\ell$. It bounds a disk diagram $D$, and $D$ must have at least two 2-cells. Then by \cref{thm: global rel IPI}, for all $\epsilon>0$ we have
            \[
               2\ell > |\partial D|\geq (1-2d-\epsilon)2\ell Area(D)\geq 2(1-2d - \epsilon)2\ell .
            \]
        Since $d<1/4$ this is a contradiction.
    \end{proof}

\begin{cor}\label{cor: factors embed}
When $d < \frac{1}{2}$, with overwhelming probability, the complex $X_{\mc{R}}$ is aspherical and the quotient map $G_* \to G$ restricts to an embedding on each factor group $G_1, \dots, G_n$.
\end{cor}

\begin{proof} We first prove that the quotient map $G_* \to G$ restricts to an embedding on each factor group $G_1, \dots, G_n$. By contradiction, 
 \Cref{L:vanKampenlemma} implies that there is a disk diagram over $\mc{R}$ of boundary length $=2$. Applying this to \Cref{thm: global rel IPI} we have 
    \[
        2\geq (1-2d-\epsilon)2\ell
    \]
for all $\epsilon>0$. Pick $\epsilon = \frac{1}{2}(1-2d)$. Since $d<1/2$, for sufficiently large $\ell$ this is a contradiction.

Now, by \Cref{L:diagramsoverR} every diagram in $X_{\mc{R}}$ admits a decoration that makes it a diagram over $\mc R$. As $X_{\mc R}$ is simply connected and $2$-dimensional, to prove asphericity it suffices to show that there are no spherical diagrams in $X_{\mc R}$. Cutting open a spherical diagram along an edge yields a disk diagram over $\mc R$ of boundary length $2$. By the above argument this is impossible.
\end{proof}

Combining \cref{cor: factors embed} with \cref{L:diagramsoverR} immediately yields the following.
\begin{cor}\label{cor: diagrams over R with probability}
When $d < \frac{1}{2}$, with overwhelming probability the following holds. Let $Y$ be a diagram in $X_{\mc R}$, equipped with its inherited decoration. Then $Y$ is a diagram over $\mc R$. 
\end{cor}
From now on we equip any diagram in $X_{\mc R}$ with its inherited decoration. Via \cref{cor: diagrams over R with probability}, then every diagram in $X_{\mc R}$ is a diagram over $\mc R$. Thus it makes sense to talk about reduced diagrams in $X_{\mc R}$. Moreover, \cref{thm: cancel implies planar} and \cref{thm: global rel IPI} apply to reduced diagrams in $X_{\mc R}$. We use these facts without further notice. 

Fineness is an important tool in Bowditch's characterization of relative hyperbolicity. There are several equivalent formulations, see \cite[Proposition 2.1]{BowditchRH}. We use the following definition.

\begin{defn}\label{D: fine}
    Let $\Gamma$ be a graph. A \emph{circuit} in $\Gamma$ is the image of a continuous injective mapping $S_1\to \Gamma$. 
    The graph $\Gamma$ is \emph{fine} if for any edge $e$ of $\Gamma$ and any $n\in\Z^{>0}$, there are finitely many circuits of length $n$ that contain $e$. 
\end{defn}

\begin{proposition}\label{prop: XR is fine}
Let $d<1/2$. With overwhelming probability $X_{\mc{R}}^{(1)}$ is a fine hyperbolic graph. 
\end{proposition}

\begin{proof}
    By \Cref{thm: global rel IPI} $X_\mc{R}$ satisfies a linear isoperimetric inequality, hence $X_\mc{R}$ is hyperbolic by \Cref{P:BHIPI}. 

    Fix a positive integer $N$. Let $e$ be any edge in $X_{\mc R}^{(1)}$ and consider a lift $\tilde e$ in the tree $T$.  By \Cref{lem: XR simply connected} $X_\mc{R}$ is simply connected, so any embedded loop $\gamma$ containing $e$ of length $|\gamma|\leq N$ is the boundary of a non-filamentous reduced disk diagram $D$ without cut vertices. By \cref{cor: factors embed} and \cref{L:diagramsoverR}, this diagram $D$ is a disk diagram over $\mc R$ with its inherited decoration. In particular, every $2$-cell bears an element of $\mc{R}$. 

    By \Cref{thm: global rel IPI} there is a uniform bound on the area of such a disk diagram, and since $\mc{R}$ is finite there are only finitely many such disk diagrams with boundary length at most $N$ and whose image contains the edge $e$.
\end{proof}

Finally, we obtain the main result of this section.

\begin{cor}[\cref{thm: rel hyp at d half}\eqref{I: rel hyp d<1/2}]\label{cor: global rel hyp}
When $d < \frac{1}{2}$, with overwhelming probability $G \sim \mathcal{FPD}(\mc{G}; d, m, \ell)$ is hyperbolic relative to $\{G_i\}_{i=1}^n$. 
Moreover, the stabilizers of vertices in $X_\mc{R}$ are exactly the conjugates of the $G_i$ factors.
\end{cor}

\begin{proof}
    This follows from \Cref{prop: XR is fine}, \Cref{cor: factors embed},  \Cref{lem: finite edge stabilizers}, and Bowditch's characterization of relative hyperbolicity \cite[Definition~3.4 (RH-4)]{Hruska2010}.
    The moreover statement is immediate from \Cref{lem: finite edge stabilizers}.
   \end{proof}

\subsection{Density more than \texorpdfstring{$1/2$}{half}}
When $d>1/2$ we show that $G$ is, with overwhelming probability, finite. 
A previous version of this article only included a rough sketch of the  proof of \cref{thm: rel hyp at d half}\eqref{I: big density}. The following proof was outlined to us by the anonymous referee. 

Recall that the Probablistic Pigeonhole Principle states that when sorting $f(N)<N$ pigeons into $N$ holes uniformly independently at random with replacement, if $\sqrt{N} \leq f(N)$ then as $N \to \infty$ the probability that one hole receives at least 2 pigeons goes to $1$.  We will use the probabilistic pigeonhole principle to produce pairs of relations between such sets with common prefixes to deduce finiteness of the quotients. As pointed out in \cref{R:twoFactors_noRelations} if there are only two factor groups, that is, $n=2$, there are no relations for $\ell$ odd. In this case we assume that $\ell$ is even.

\begin{lemma}\label{lem: pigeons} Let $d > \frac12$. 
Let $u_1, u_2$ be distinct elements in the free product $G_*$  having free product length $1$ or $3$ so that the first and final syllables of the normal forms representing $u_1$ and $ u_2$ are all in $B_i(m)$ for some $i$.
Then,  with overwhelming probability, the images of $u_1$ and $u_2$ are equal in $G$. 
\end{lemma}
\begin{proof} We prove that, for $d > \frac12$, there exists some $v$ so that $vu_1$ and $vu_2$ are simultaneously in $\mathcal{R}$ with overwhelming probability. This then implies that $u_1=u_2$ in $G$ with overwhelming probability. 

    Denote by $\mc S_\ell(u)$ the subset of $\mc S_\ell$ consisting of words $w$ with a fixed suffix $u$ such that the free product length $||wu^{-1}|| = \ell - ||u||$. Recall from \Cref{S:cancelIPI_proof} that $\mc W^{i,j}_k$ denotes words of syllable length $k$ whose first syllable is not contained in the $i$th factor and last syllable is not contained in the $j$th factor. In contrast with $\mc S_k$, words in $\mc W^{i,j}_k$ are allowed to have first and last syllables in the same factor. 
    For suffixes with free product length $||u|| \leq 3$ having first and final syllable in $B_i(m)$, we will use these sets to count prefixes. Indeed, the set of prefixes $wu^{-1}$ of the words in $\mc S_\ell(u)$ is equal to  $\mc W^{i,i}_{\ell-||u||}$. By \cref{cor: counting more words}, there are constants $C>0$ and $\lambda > 0$ such that $|\mc W^{i, i}_{\ell-{||u||}}| \leq C \lambda^{\ell-||u||} \leq C \lambda^{\ell-1}$ for sufficiently large $\ell$.

Also note that ${|\mc S_\ell(u)|} \geq  |\mc W^{i, i}_{\ell-{||u||}}| $. 
Thus, by \cref{cor: counting more words}, there exists a constant $C' \geq 0$ such that for sufficiently large $\ell$ we have
${|\mc S_\ell(u)|} \geq C' {|\mc S_\ell|} \geq C' \lambda^\ell$. 
Since $\mc R$ is chosen uniformly from $\mc S_\ell$, with overwhelming probability
\[|\mc R \cap \mc S_\ell(u) | \geq C' |\mc R \cap \mc S_\ell| = C' |\mc R| > C'\lambda^{d\ell}.\]
So when $d > \frac12$, we have that 
\[|\mc R \cap \mc S_\ell(u)| > \lambda^{\ell/2} \]
with overwhelming probability when $\ell\to \infty$.

 We now apply the probabilistic pigeonhole principle: The pigeons are pairs in $(\mc{R}\cap\mc{S}_\ell(u_1)) \times (\mc{R}\cap \mc{S}_\ell(u_2))$. The holes are subsets of $W^{i,i}_{\ell-||u_1||}$ of size at most 2. The map from pigeons to holes is $(v_1u_1, v_2u_2) \mapsto \{v_1, v_2\}.$ That is, we assign pairs to the set of their prefixes. 

Two distinct pairs $(x_1, x_2)$ and $(y_1, y_2)$ are assigned to the same set $A \subset W^{i,i}_{\ell-||u_1||}$ if and only if 
\[
x_1 = v_1u_1, \quad  x_2 = v_2u_2, \quad y_1 = v_2u_1 , \quad y_2 = v_1u_2.
\]
In this case, the desired relators are $x_1$ and $y_2$. Thus, when $\ell\to \infty$, we deduce that $\mc{R}$ contains relators $vu_1$ and $vu_2$. 
\end{proof}

We are ready to show finiteness of the quotient. 

\begin{proof}[Proof of \cref{thm: rel hyp at d half}\eqref{I: big density}]  
 
First, the quotient $G$ is generated by at most one order-2 element per factor group. Indeed, 
for any generator $x \in B_i(m)$ that does not have order two, $x^{-1}$ also appears in $B_i(m)$. Thus by applying \Cref{lem: pigeons} with $u_1 = x, u_2 = x^{-1}$ we see that with overwhelming probability $x = x^{-1}$ in $G$. Thus every generator is mapped to an element of order two in the quotient. 

Similarly, if $x, y$ are generators of $G_i$ then \Cref{lem: pigeons} implies that $x = y$ in $G$. 

When the number of factor groups $n \geq 3$, consider any pair of generators $x,y$ in distinct factor groups. Let $z$ be a generator in a factor that does not contain either $x$ or $y$. 
Then \Cref{lem: pigeons} with $u_1 = zxz$ and $u_2 = zyz$ implies that with overwhelming probability $zxz = zyz$ in $G$. Hence, $x=y$ in the quotient.
Therefore, in this case, $G$ has order at most 2.

On the other hand, when there are exactly two factors, that is, $n=2$, 
with overwhelming probability $G$ is still generated by two elements $a$ and $b$ of order 2, one for each factor. 
Moreover, with overwhelming probability $\mc R$ will contain distinct relators that map to the alternating words of length at most $\ell$ starting with $a$ and starting with $b$. 
In this case, $G$ is a quotient of a finite dihedral group, all of which are finite dihedral groups.
\end{proof}

\section{Preparing to Cubulate at \texorpdfstring{$d < \frac{1}{6}$}{density below a sixth}} \label{sec: rel_cubulation}

The following type of action of a relatively hyperbolic group on a CAT(0) cube complex was introduced by Einstein and Groves \cite[Definition~2.1]{EG:RelGeom}.

\begin{defn}\label{def: relatively geometric}
    Let $(K, \mc{D})$ be a relatively hyperbolic pair. An isometric action of $K$ on a CAT(0) cube complex $\tilde{X}$ is \emph{relatively geometric (with respect to $\mc{D}$)} if:
		\begin{enumerate}
			\item the action of $K$ on $\tilde{X}$ is cocompact, 
			\item every peripheral subgroup $D \in \mc{D}$ acts \emph{elliptically} (i.e. stabilizes a cell), and
			\item cell stabilizers are either finite or conjugate to a finite index subgroup of some $D \in \mc{D}$. 
		\end{enumerate}
		Groups that admit a relatively geometric action with respect to some collection of peripheral subgroups are called \emph{relatively cubulated}.
\end{defn}

It is possible to give a wallspace construction adapted from \cite{ollivier_wise} to $X_\mc{R}$ to show that our groups $G = (G_1 * \cdots * G_n)/\llangle \mc{R}\rrangle$ act on a CAT(0) cube complex. 
When $d<\frac{1}{6}$, to show that the action is, with overwhelming probability, relatively geometric with respect to $\{G_1, \dots, G_n\}$, we use a criterion from \cite{EinsteinNg}. This approach allows us to simultaneously show that, in the case that each $G_i$ is relatively cubulated, $G$ is also cubulated relative to a finer peripheral structure whose elements are the peripherals of the $G_i$. To do this we use the space $X_\mc{R}(Q_1, \dots, Q_n)$ introduced in \Cref{sec: model spaces}.

\subsection{Hyperstructures in Mixed Polygonal-Cubical Complexes}\label{sec: hyperstructures pc complex}

Let $\mc{Q} = (Q_1, \dots, Q_n)$ be an $n$--tuple of CAT(0) cube complexes such that $G_i$ acts cellularly on $Q_i$ without inversions. Recall that the space $X_\mc{R}(Q_1, \dots, Q_n) = X_\mc{R}(\mc{Q}),$ defined in \Cref{sec: model spaces}, is a complex composed of polygons and cube complexes. We generalize such a space in the following.

\begin{defn}
A complex $\Omega$ is a \emph{mixed polygonal-cubical complex} if $X$ is a cell complex whose cells are either $n$--cubes (not necessarily all of the same dimension) or polygons. 

We say $\Omega$ is a \emph{mixed even polygonal-cubical complex} if every polygon has an even number of sides. 
\end{defn}

 For the following, let $\Omega$ be a mixed even polygonal-cubical complex. 
 We now define $\Omega$--hyperstructures, which generalize hyperplanes in the case that $\Omega$ is a cube complex. Similar notions have been used to construct walls before in \cite{WiseSmallCancellation,martin_steenbock,EinsteinNg}, for example. 
 Let $C$ be a cell of $\Omega$, and let $e_1,e_2$ be edges of $C$. We say that $e_1$ is \emph{opposed} to $e_2$ if either of the following hold.
  
  \begin{enumerate}
 \item  If $C$ is an $n$--cube, we say that $e_1\sim_{opp} e_2$ if and only if $e_1,e_2$ are dual to the same midcube.
 \item If $C$ is a polygon, we say that $e_1\sim_{opp} e_2$ if and only if they are the same edge or are diametrically opposed in $C$ (recall that $C$ must have an even number of edges).
 \end{enumerate}
 Then $\sim_{opp}$ extends to an equivalence relation on the edges of $\Omega$ by taking the transitive closure. 
 
 \begin{example}
 If $\Omega$ is a cube complex, then each $\sim_{opp}$ equivalence class is the collection of edges dual to a hyperplane of $\Omega$. 
 \end{example}

\begin{defn}\label{D: hyperstructure}
Let $\Omega$ be a mixed even polygonal-cubical complex. The $\Omega$-- \emph{hyperstructure associated to an edge $e$ of $\Omega$} is the subspace $W_e^{\Omega}$ of $\Omega$ constructed as follows:
\begin{enumerate}
\item for each polygon $C$ and each pair $e_1,e_2$ of diametrically opposed edges with $e_1,e_2\in [e]_{\sim_{opp}}$, add a geodesic segment 
from the midpoint of $e_1$ to the midpoint of $e_2$ via the center of $C$, and \label{I: hyperstructure 1}
\item for each cube $C$, include any midcube that is dual to an edge of $[e]_{\sim_{opp}}$.  \label{I: hyperstructure 2}
\end{enumerate}
The \emph{carrier} of $W_e^{\Omega}$ is the union (in $\Omega$) of all cells whose interior intersects $W_e^{\Omega}$ nontrivially.

If $X$ is a polygonal complex, the $X$--hyperstructures will be called $\emph{edge hypergraphs}$. 
\end{defn}

\begin{example}
If $\Omega$ is a cube complex, then an $\Omega$--hyperstructure is a hyperplane.
\end{example}

 Given a hyperstructure $W$, the \emph{abstract hyperstructure corresponding to $W$} is a complex $\tilde{W}$ whose vertices are in one-to-one correspondence with the edges in $[e]_{\sim_{opp}}$. For each geodesic segment joining midpoints of edges $e_1,e_2$ as in \Cref{D: hyperstructure}\eqref{I: hyperstructure 1} the abstract hyperstructure has an edge between the vertices corresponding to $e_1,e_2$. Similarly, for each midcube $C_{mid}$ dual to edges $\{e_1, e_2, \dots\}$ as in \Cref{D: hyperstructure}\eqref{I: hyperstructure 2}, the abstract hyperstructure $\tilde{W}$ has a copy of $C_{mid}$ whose vertices are $\{e_1, e_2, \dots\}$. While the abstract hyperstructure need not be embedded (see for example Figure~4 of \cite{EinsteinNg}), $W$ is naturally the image of an immersion of $\tilde{W}$ into $\Omega$. 
 
Every hyperstructure $W$ carries a combinatorial metric on the 0--skeleton of its corresponding abstract hyperstructure. If $W$ is an edge hypergraph, this is the path metric on the abstract graph representing $W$. We call this the \emph{hyperstructure metric of $W$}, and denote it $d_W$. Note that if $x, x' \in W$ then $d_{X^{(1)}}(x, x')$ need not be equal to (the combinatorial distance) $d_W(x, x')$. This can be seen immediately in a $2n$-gon for $n \geq 3$ in which $d_{X^{(1)}}(x,x') = n$, whereas $d_W(x,x') = 1$.

We will also be interested in the following.

\begin{defn}\label{def: projection EX to X}
Let $X$ be a polygonal complex where every polygon has an even number of sides, and let $\EX$ be a mixed even polygonal-cubical complex so that there is a surjective combinatorial map:
\[p: \EX\to X\]
that takes cubes to points and nontrivial polygons to nontrivial polygons, in such a way that each open edge of $X$ has a unique open edge as preimage. We call $p$ the \emph{projection map} from $\EX$ to $X$. 
\end{defn}

\begin{defn}\label{def: projected hypergraph}
Let $X$ be an even polygonal complex, $\EX$ a mixed even polygonal-cubical complex and $p: \EX \to X$ the projection map.
    A \emph{($p$--)projected hypergraph} is the image of an $\EX$--hyperstructure. 
\end{defn}
The \emph{abstract hypergraph corresponding to $p(W)$} can be constructed by taking the abstract hypergraph corresponding to $W$ and collapsing the cells whose images in $\EX$ are collapsed by $p$. 
A projected hypergraph $p(W)$ may not be an edge hypergraph of $X$, since midcubes in $W$ are projected to vertices of $X$, and thus antipodality is also not necessarily preserved.

\begin{example}\label{ex: projected hypergraph}
    Let $G_1=\langle a\rangle$ be an infinite cyclic group, and let $G_2=\langle b,c\rangle$ be a free group. Let $Q_1, Q_2$ be the standard Cayley graphs for $G_1, G_2$, respectively. These are trees so they are CAT(0) cube complexes. Let the radius $m = 1001$. Let $\mathcal R=\{a^{1000}bc\}$. Let $C$ be a relator cell in $X_\mc{R}$, and let $\widetilde c$ be its pre-image in $X_{\mc{R}}(\mc Q)$. The hypergraph in $\widetilde c$ at the first edge of the subpath of $\partial \widetilde c$ in $Z_1$ returns to (the same copy) of $Z_1$. Thus the projected hypergraph in $C$ consists of the corresponding factor vertex and a segment to the center to the polygon. See \Cref{fig:projected hypergraph}.
\end{example}

\begin{figure}
    \centering
    \includegraphics[width=0.5\linewidth]{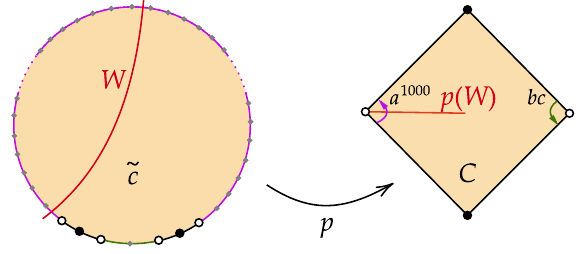}
    \caption{An antipodal hypergraph $W$ in $\EX$ need not project to a hypergraph in $X$. See \Cref{ex: projected hypergraph}.}
    \label{fig:projected hypergraph}
\end{figure}

In the following subsection, we show that there is a subdivision of $X_\mc{R}$ and $X_\mc{R}(\mc{Q})$ so that projected hypergraphs are close to being edge hypergraphs in the following sense.

\begin{defn}
    Let $X, \EX, \, p: \EX \to X$ be as in \Cref{def: projection EX to X}. Let $\epsilon\geqslant 0$, and suppose that every polygon of $X$ is an $L$-gon. We say that a projected hypergraph $Z$ is \emph{$\epsilon$-antipodal} if 
        $d_{X^{(1)}}(x, x')\geq L(\frac{1}{2}-\epsilon)$ 
    for all adjacent vertices $x, x'$ of $Z.$ 
    \end{defn}
    All edge hypergraphs are $\epsilon$-antipodal (for all possible $\epsilon$).

We will henceforth use \emph{hypergraphs} to refer to both edge hypergraphs and 
projected hypergraphs.

\subsection{Subdividing} \label{sec: subdividing}

From now on, we fix $\EX$ to be the mixed polygonal-cubical complex $X_{\mc{R}}(\mc{Q})$ as defined in \Cref{sec: model spaces}, where the spaces $Q_i \in \mc{Q}$ are CAT(0) cube complexes on which the factor groups $G_i$ act (relatively) geometrically. We also fix $X$ to be the polygonal complex $X_{\mc{R}}$. 

The next result, \Cref{P: blowup simp conn}, follows from \cite[Theorem 2.4]{Martin:NPC_boundary}. 
In \cite[Theorem 2.4]{Martin:NPC_boundary}, properness is a hypothesis, but properness is not required for statements about simple connectedness and cocompactness. 

\begin{proposition}\label{P: blowup simp conn}
    The complex $\EX$ is simply connected. 
    Moreover, the cocompactness of the action of each $Q\in\mc{Q}$ implies that $\EX$ is $G$--cocompact. 
\end{proposition}

As in \cite{martin_steenbock}, to ensure that projected hypergraphs are $\epsilon$-antipodal, we subdivide the edges of the model space $\EX = X_\mc{R}(\mc{Q})$ whose image lies in $X$. 

Recall the 
map $p: \EX \to X$ from \cref{def: projection EX to X}. By definition, each open edge in $X$ has a unique preimage in $\EX$. We call every such closed edge of $\EX$ a \emph{polygonal edge}, and every other edge of $\EX$ a \emph{cubical edge}.
We make note of this information with the following. 

\begin{obs}
\label{obs:edge_partition-tree}
    Polygonal edges and cubical edges decompose the edges of $\EX$  into a $G$-invariant partition.
\end{obs}

Let $\EX[k]$ be obtained from $\EX$ by cubically subdividing each $Q_i \in \mc{Q}$ such that every cubical edge becomes a segment of length two,
 subdividing every polygonal edge into a segment of $2k$ edges, and 
 updating the attaching maps on polygons accordingly. 

The projection map $p: \EX \to X$
induces a natural quotient map 
\[
p_k: \EX[k] \to X[k],
\]
where the space $X[k]$ is obtained by subdividing each edge of $X$ into a segment of $2k$ edges. When there is no confusion, we will denote the map $p_k$ by $p$.

\begin{remark}\label{R: bounded length in fiber}
    Let $\tau$ be the maximal translation length for all elements of all $B_i(m)$ from the finite set $\bigcup B_i(m)$.
    For each polygon $D$ in $\EX$ (respectively $\EX[k]$) a \emph{cubical segment}, i.e., a subsegment of $\partial D$ that is contained in a single cube complex (a copy of a subdivided $Q_i$),
has length at least zero, and at most the uniform constant $2\tau \geq 0$. On the other hand, a maximal \emph{polygonal segment}, i.e., a path consisting of (subdivided) polygonal edges, has length exactly $2$ (respectively $4k$).
\end{remark}   

Let $0 < d < \frac{1}{5}$ and $\ell > 0$.  
Fix the following constants to be used here and in \Cref{lem: local hypergraphs embed}: $$\epsilon_{d,\ell}  := \frac{1}{2}\min\{1/5 - d, 1/\ell\} \qquad \text{and} \qquad k_{d,\ell} := \left\lceil{\frac{\tau}{4\epsilon_{d,\ell}}}\right\rceil + 1.$$ 

\begin{remark}
    Note that the results on random quotients below are with overwhelming probability when $\ell \to \infty$, so that we can always assume that $\frac{1}{\ell}< \frac 15-d$. We keep the dependence of $\varepsilon_{d,\ell}$ on $d$ as this allows to show the results in the more general setting of groups acting on complexes with polygons of edge length $=\ell$, with trivial edge stabilizers, and linear isoperimetric inequality with isoperimetric constant $(1-2d)\ell$.
\end{remark}

\begin{defn}\label{def: Xbal and EXbal}
 The \emph{balanced polygonal-cubical complex} $\EX_{bal}$ is the subdivision $\EX[k]$ of $\EX$ for  $k = k_{d,\ell}$, while the \emph{balanced polygonal complex} $X_{bal}$ is the subdivision $X[k]$ of $X$ for $k = k_{d,\ell}$.
\end{defn}

\begin{proposition}\label{P: epsilon antipodal projections}
The subdivisions $\EX_{bal}$ and $X_{bal}$ have the following properties: 

\begin{enumerate}
    \item \label{item:mixed-even}
        $\EX_{bal}$ is a mixed even polygonal-cubical complex,
    \item \label{item:edge_stab}
        $p$ can be viewed as a combinatorial map $p:\EX_{bal}\to X_{bal}$,
    \item \label{item:same_side_length} 
        every polygon of $X_{bal}$ has the same number of sides, which we denote by $L$, and
    \item \label{item:e antipodal}
        if $W$ is an $\EXbal$--hyperstructure, its image $p(W)$ is an $\epsilon_{d,\ell}$--antipodal hypergraph in $X_{bal}$.
\end{enumerate}
\end{proposition}

\begin{proof}
Item (\ref{item:mixed-even}) is immediate from the construction of the subdivision. 
Item (\ref{item:edge_stab}) follows by applying \Cref{lem: finite edge stabilizers} since subdivisions are $G$--equivariant. Item (\ref{item:same_side_length}) follows from the construction of $X[k]$.

To show Item~\eqref{item:e antipodal}, let $c$ be a polygon of $X_{bal}$ and let $\tilde{c}$ be its pre-image in $\mathcal EX_{bal}$. 
Let $x,y$ be two antipodal points in $\partial \tilde c$. Recall that $k = k_{d, \ell}$ as defined above, that $|\partial c|=4k\ell$, and that there are $\ell$-many factor vertices on $\partial c$. Thus 
\begin{align}\label{E:antipodel_upstairs} 
    2 k\ell \leqslant d_{\mathcal EX_{bal}^{(1)}}(x,y)\leqslant (4k\ell  + 2\ell\tau)/2=2k\ell +\ell\tau.
\end{align}
The lower bound follows as 
   $d_{\mathcal EX_{bal}^{(1)}}(x,y)=\frac 12 |\,\partial \tilde c\;|\geq \frac{1}{2}|\,\partial c\,|=2k\ell.$ To see the upper bound recall that the pre-image of each factor vertex of $\partial c$ contributes at most $2\tau$-many edges to $\partial \tilde c$. As there are $\ell$-many factor vertices in $\partial c$, the upper bound follows. 

Let $\gamma$ be a shortest simple path in $\partial {c}$ from $p(x)$ to $p(y)$ and let $\widetilde \gamma$ be a simple path from $x$ to $y$ in $\partial \widetilde c$ that projects onto $\gamma$. Then $\length(\widetilde \gamma) \geq d_{\mathcal{E}X_{bal}^{(1)}}(x,y)$. Also note that 
$$\length (\gamma)\geq \length(\widetilde \gamma) - 2\tau\,|\{\text{factor vertices on }\gamma\}|.$$
Thus we have that 
$$d_{\mathcal EX_{bal}^{(1)}}(x,y)\geq \length(\gamma)\geq d_{\mathcal{E}X_{bal}^{(1)}}(x,y)- 2\tau\,|\{\text{factor vertices on }\gamma\}|.$$ 
We claim that $\gamma$ contains at most $\ell/2+1$ factor vertices. Indeed, let us by contradiction assume that there are more than $\ell/2+1$ factor vertices. By definition, the segment between two consecutive factor vertices on $\partial c$ is of length $4k$, and by assumption there are at least $\lfloor \ell/2\rfloor +1$ such segments in $\gamma$. Note that $\lfloor \ell/2\rfloor\geqslant \ell/2-1/2$. Thus 
$$d_{\mathcal EX_{bal}^{(1)}}(x,y) \geq 4k(\lfloor \ell/2\rfloor +1)\geq 4k (\ell/2-1/2)=2k\ell+2k.$$ As $\ell\tau<2k$ by the choice of $k$, this is a  contradiction to Equation \eqref{E:antipodel_upstairs}. This yields the claim. 

Using the claim, we now estimate  $d_{X_{bal}^{(1)}}(p(x), p(y))$ as follows:
\begin{align*}
    d_{X_{bal}^{(1)}}(p(x), p(y)) & \geq d_{\mathcal{E}X_{bal}^{(1)}}(x,y)- 2\tau\,|\{\text{factor vertices on }\gamma\}|\\
     & \geq 2k\ell - 2\tau \left(\frac{\ell}2+1\right) \\
     & \geq 4k\ell \left(\frac{1}{2}-\frac{\tau}{2k}\left(\frac{1}{2}+\frac{1}{\ell}\right)\right).
\end{align*}
One verifies from the definition of $k$ that 
$$\frac{\tau}{2k}\left(\frac{1}{2}+\frac{1}{\ell}\right)<\varepsilon_{d,\ell}.$$
This yields that 
   \begin{align*}
    d_{X_{bal}^{(1)}}(p(x), p(y))  > 4k\ell \left(\frac{1}{2}-\varepsilon_{d,\ell}\right),
\end{align*}
which completes the proof. 
\end{proof}

\begin{notation}\label{notation: Xbal polygon length hypergraph hyperstructure}
For the remainder of the paper, 
$L = 4k\ell$ will be the length of the boundary of a polygon in $X_{bal}$. 

If $e$ is an edge of $X_{bal}$, let $Z_e^X$ denote the hypergraph of $X_{bal}$ dual to $e$. 
If $e$ is an edge of $\EXbal$, we let $W_e^{\EX}$ denote the $\EXbal$--hyperstructure dual to $e$. 
\end{notation}

\begin{defn}
Let $v$ be a vertex of $X_{bal}$. 
    A \emph{fiber complex} $\EX_v$ is the preimage of $v$ under the projection $p:\EXbal\to X_{bal}$.  
\end{defn}

The metric on $X_{bal}^{(1)}$ is the standard edge metric; thus an edge path of length $r$ in $X$ corresponds to an edge path of length at most $2kr$ in $X_{bal}$. 

Let $Y_{bal}\to X_{bal}$ be a diagram in  $X_{bal}$. Then there is a diagram $Y\to X$ such that the following diagram commutes:
       \[ \begin{tikzcd}
\arrow{r}{} X    & X_{bal} \\%
 \arrow{r}{} Y \arrow[swap]{u}{} & Y_{bal} \arrow[swap]{u}{}
\end{tikzcd}
\]
In this case we call $Y$ the preimage of $Y_{bal}$ in $X$. 
 We say that a diagram in $X_{bal}$ is \emph{reduced} if its pre-image in $X$ is a reduced diagram. 

\begin{remark}\label{prop: IPI in subdivision} 
 The inequality \eqref{Eq: cancelbound}, which holds in $X$ by \Cref{thm: cancel implies planar}, continues to hold for reduced $(K,M)$--bounded diagrams in $X_{bal}$ with $2\ell$ replaced by $L$. Thus we can apply \cref{lem: greendlinger} in $X_{bal}$.  Similarly, the conclusion of  \cref{thm: global rel IPI} will continue to hold for $X_{bal}$. \end{remark}

\subsection{Properties of Projected Hypergraphs in \texorpdfstring{$X_{bal}$}{balanced X}}
\label{subsec:embedded-hypergraphs}

In this subsection we show that hypergraphs 
of $X_{bal}$ are embedded with overwhelming probability whenever $d<1/5$. Furthermore, the map from the abstract hypergraph with the hypergraph metric to $X_{bal}$ is a quasi-isometry. 

In order to do so, we analyze diagrams induced by self-intersections of hypergraphs. We use the terminology developed in \cite[Section 3]{ollivier_wise}, and we refer the reader there for further details. 

For every loop $\alpha$ in $X_{bal}^{(1)}$ there is a reduced disk diagram whose boundary is $\alpha$ by \Cref{lem: XR simply connected}. We say this disk diagram is \emph{bounded by $\alpha$}. Similarly, loops in hypergraphs also bound disk diagrams in the following manner (see \Cref{fig:collared}).

\begin{defn}
    Let $\gamma$ be a hypergraph segment of length $N$ passing through $2$-cells $C_1, \dots, C_N$, such that $C_1 = C_N$. There are two cases: the image in $X_{bal}$ of $\gamma$ is either a loop or has a self-intersection at the center of $C_1 = C_N$ in $X_{bal}$. For each edge $e_i$ in $\gamma$ with $1 \leq i \leq N-1$, we pick a path $\alpha_i \subset \partial C_i$ connecting the endpoints of $e_i$. If the image of $\gamma$ is a loop, we  define $\alpha = \alpha_1\dots\alpha_{N-1}$ in $X_\mc{bal}^{(1)}$. On the other hand, if $\gamma$ induces a self-intersection in $e_1, e_N$, replace $\alpha_1$ with a path from the last endpoint of $e_1$ to the first endpoint of $e_N$, and let $\alpha = \alpha_1\dots\alpha_{N-1}$. In either case, a reduced disk diagram $D$ bounded by $\alpha$ is referred to as disk diagram \emph{bounded by $\gamma$}. We construct a \textit{quasi-collared diagram} associated to $\gamma$, denoted $D \bigcup_{\alpha_i} \{C_i\}_{i=1}^{N-1}$, as follows: glue $2$-cells $C_i$ to $D$ along $\alpha_i$. The 2-cell $C_1$ is a \emph{corner} of $D$. The subdiagram $\bigcup_{i=1}^{N-1} C_i$ is the \emph{collar}.
\end{defn}

\begin{figure}
    \centering
    \includegraphics[width=0.8\linewidth]{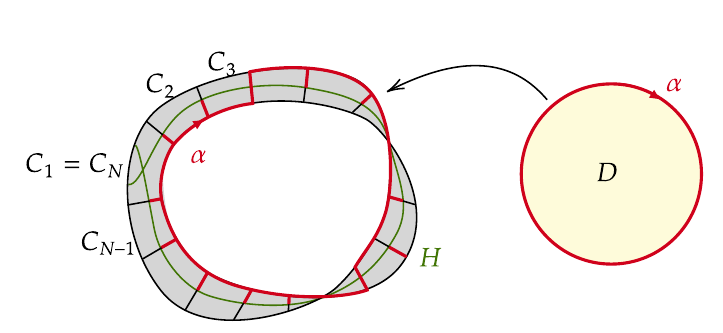}
    \caption{Construction of a collared diagram for a self-intersection in hypergraph $H$. Note that the collared diagram need not be planar.}
    \label{fig:collared}
\end{figure}

Given a loop obtained as the concatenation of hypergraph segments $\gamma_1, \dots, \gamma_k$, we may similarly define a \emph{$k$-quasi-collared diagram}, with \emph{corners} occurring in the 2-cells which contain points in $\gamma_i \cap \gamma_{i+1}$.

\begin{lemma} \label{lem: quasi-collared diagrams are bounded}
Let $N>0$, $d<1/5$. 
There exist $K, M>0$ depending only on $N, d$ such that for every hypergraph segment $\gamma$ in $X_{bal}$ of length $\leq N$,
with an associated quasi-collared diagram $Y$, $Y$ is $(K, M)$-bounded with overwhelming probability.
\end{lemma}

    \begin{proof}
        If $\gamma$ has an associated quasi-collared diagram $Y = D \cup_{\alpha_i} \{C_i\}_{i=1}^{k-1}$ for some $k \leq N$, then $\gamma$ contains a loop or self-intersection. We may assume that $D$ is reduced. Recall that hypergraph segments of $X_{bal}$ are $\epsilon$-antipodal where $\epsilon = \epsilon_{d, \ell} < 1/5 - d$. Since $|\partial D| \leq N L\left(\frac{1}{2}+\epsilon\right)$, there is a uniform bound on the area of $D$ by \cref{prop: IPI in subdivision}. 
        Thus, with overwhelming probability, there is a uniform bound $K$ on the area of $Y$. 

        Let $C = \overline{Y-D}$. By construction the number of connectors in $C$ is bounded above by $4N$. 
        By \cref{ex:disks are KM bounded}, $D$ is $(K, \frac{1}{2}K(K-1)^2+K^2)$-bounded. 
        The number of connectors in $Y$ is bounded above by the sum of the number of connectors in $D$ and $C$ and the number of connectors in $Y$ that lie on $D\cap C$. Therefore the number of connectors in $Y$ is at most $(1 + 4N)(\frac{1}{2}K(K-1)^2+K^2)+4N$.
    \end{proof}

\begin{lemma}\label{lem: local hypergraphs embed}
    Let $N>0$, $d<1/5$.
    With overwhelming probability any hypergraph segment of length $\leq N$  is embedded in $X_{bal}$.
\end{lemma}

    \begin{proof}
        Suppose there is some segment $\gamma$ of length $k \leq N$ which is either a loop or has a self-crossing. We may assume that $\gamma$ is minimal; i.e., there is no self-intersection except at the endpoints of $\gamma$. Let $Y = D \cup_{\alpha_i} \{C_i\}_{i=1}^{k-1}$ be a quasi-collared diagram of $\gamma$ with $D$ reduced.  Up to swapping a choice of $\alpha_i$ for its opposite path in the boundary of $C_i$, we may assume that $D\cup_{\alpha_i}\{C_i\}_{i=1}^{k-1}$ is reduced. By \cref{lem: quasi-collared diagrams are bounded}, $Y$ is $(K, M)$-bounded for some $K, M$, so \Cref{prop: IPI in subdivision} allows us to apply \cref{lem: greendlinger}.
        
       Note that the $2$-cells in $D$ have entirely internal edges or they are contained in $\{C_i\}_{i=1}^{k-1}$. For $i\not = 1$, the cells $C_i$ contribute at most $ L(1/2+\epsilon_{d, \ell})<  L(1/2 + (1/5-d))< L(1 - 5d/2)$ to the boundary. Therefore, the only $2$-cell that can contribute at least $ L(1-5d/2)$ edges to the boundary is $C_1$. 

        We obtain a contradiction because \cref{lem: greendlinger} requires $D\cup \{C_i\}_{i=1}^{k-1}$ to have at least two $2$-cells that contribute at least $ L(1-5d/2)$ edges to the boundary.
    \end{proof}

\begin{theorem}
\label{thm: qi embedded d fifth}
Let $d<1/5$ and let $Z$ be an abstract hypergraph of $X_{bal}$. 
Then the natural immersion $Z \to X_{bal}$ restricts to a $(20 L,0)$-quasi-isometric embedding from $(Z, d_Z)$ to $(X_{bal}, d_{X_{bal}^{(1)}})$.
\end{theorem}

\begin{proof} Let us rescale the metric of $X_{bal}^{(1)}$  by $1/ L$ so that the circumference of each $2$-cell is $1$. Note that with overwhelming probability disk diagrams in this metric satisfy $|\partial D| \geq (1-2d)Area(D);$ in particular we see that $X_{bal}$ is $\delta$-hyperbolic where $\delta$ does not depend on $L$. 

We claim that for any $N>0$ the immersion map on every hypergraph segment of length $\leq N$ is a $(10, 0)$-quasi-isometry with overwhelming probability when $L\to \infty$. We first explain how to derive the assertion of the Lemma from this claim. 
By \cite[Remark 7.2.B]{GromovHyperbolicGroups}, for every $\lambda \geq 1$, there is some $\kappa>0$ such that every $\kappa-$local $(\lambda,0)$-quasigeodesic in $X_{bal}^{(1)}$ is a global $(2\lambda,0)$-quasi-geodesic. Note that $\kappa$ depends only on $\delta$ and $\lambda$, and, hence, not on $L$. Applying this statement  with $N=\kappa$, the assertion of the lemma  follows. 

We now prove the claim. Let $\gamma$ be a hypergraph segment with $k\leq N$ edges. In particular, $\gamma$ passes through at most $N$ 2-cells $C_i$. By \cref{lem: local hypergraphs embed}, we assume that  $\gamma$ is  embedded, so that the $2$-cells $C_i$ are distinct. Let $\alpha$ be a geodesic in $X_{bal}\oskel$ joining the endpoints of $\gamma$. Let $\gamma_i$ be the segment of $\gamma$ contained in $C_i$. Let $\alpha_i$ be an edge path in the boundary of $C_i$ connecting the endpoints of $\gamma_i$. There is a disk diagram $D$ bounded by $\alpha\cdot \alpha_1 \dots \alpha_n$. 

Let $D$ be a disk diagram of minimal area among all such diagrams and any choice of the segments $\alpha_i$. Let $Y = D \cup_{\alpha_i} \{C_i\}_{i=1}^{k}$. Then, by choice of $D$, no boundary $2$-cell of $D$ maps to one of the $C_i$. In particular, $Y$ is reduced and we have that 
$$k=Area(Y)-Area(D).$$
Moreover, the edges $\alpha_i$ are internal in $Y$ or they lie on $\alpha$. Indeed, no boundary $2$-cell of $D$ maps to one of the $C_i$ and $C_i$ is glued to $Y$ along $\alpha_i$. More generally, edges of $D$ are internal in $Y$ unless they lie on $\alpha$.

We now estimate $\can (Y)$. 
Since the definition of $\can(Y)$ does not depend on the metric, rescaling the metric on $X_{bal}$ does not change the computation of $\can(Y)$. Let $E_0$ be the set of internal edges of $Y$. Then 
$$\can(Y)\geq \sum_{e\in E_0} \deg (e)-1 \geq  |E_0|.$$ 
On the other hand, 
\[
2|E_0| \geqslant \sum_{i=1}^k L|\alpha_i|_{X_{bal}}  + Area(D) L - |\alpha|_{X_{bal}} L \]
 because edges that are in $\alpha_i\cap \alpha$ cannot be on the boundary of a 2-cell of $D$. Using the fact that hypergraphs are $\epsilon_{d,\ell}$-antipodal (\cref{P: epsilon antipodal projections}(4)), we obtain that 
    \begin{align*}
        2\can(Y) &\geq  \sum_{i=1}^{k}|\alpha_i|_{X_{bal}} L + Area(D) L - |\alpha|_{X_{bal}} L\\
            &\geq  L(\frac{1}{2}-\epsilon_{d, \ell})k + Area(D) L - |\alpha|_{X_{bal}} L\\
            &\geq \frac{ L}{2}k + LArea(D)-Lk\epsilon_{d, \ell} - L |\alpha|_{X_{bal}}.
    \end{align*}
Note that $N$ gives a uniform bound on $|\gamma|_{Z}$, and hence a uniform bound on $|\alpha|_{X_{bal}}$. By \cref{prop: IPI in subdivision} we have a uniform bound on the size of $D$, so we have a uniform bound on the size of $Y$. By \Cref{prop: IPI in subdivision} we get $\can(Y) \leq d L Area(Y)$, so we get:
\[2dArea(Y) \geq \frac{ 1}{2}k + Area(D)-k\epsilon_{d, \ell} -  |\alpha|_{X_{bal}}.\] 
By rearranging the terms, we have 
    \begin{align*}
            |\alpha|_{X_{bal}} & \geq k\left(\frac{ 1}{2}-\epsilon_{d, \ell}\right) + Area(D) -2d Area(Y). 
    \end{align*}
Recall that $\varepsilon_{d,\ell}\leq \frac{1}{5}-d$ and that $k=Area(Y)-Area(D)$. As $d<1/5$, this yields that 
   \begin{align}
            |\alpha|_{X_{bal}} & \geq (\frac{3}{10} +d)(Area(Y)-Area(D)) +  Area(D) -2d Area(Y) \nonumber\\
             & \geq (\frac{3}{10} -d) Area(Y) + (1-\frac{3}{10} -d)Area(D) > \frac{1}{10} Area (Y). \label{Eq: 5.20 estimate}
    \end{align}

Finally, recall that each edge of $\gamma$ lies in a distinct $C_i$. Therefore $|\gamma|_Z\leq Area(Y)$.  Thus \eqref{Eq: 5.20 estimate} yields that 
    \[
       \frac 1{10} |\gamma|_Z < |\alpha|_{X_{bal}}.
    \]
On the other hand, by \cref{P: epsilon antipodal projections}(4) consecutive points in $\gamma$ are at a distance in $X_{bal}^{(1)}$ of at most $\frac{1}{2}+\epsilon_{d, \ell}$, so 
    \[
        |\alpha|_{X_{bal}}\leq \left(\frac{1}{2}+\epsilon_{d, \ell}\right)|\gamma|_Z\leq |\gamma|_Z.
    \]
In particular, the map from $Z$ to $X_{bal}$ (under the rescaled metric) is an $N$-local $(10,0)$-quasi-isometry. Hence, this map is a $(20, 0)$-quasi-isometry. Composing this with the rescaling map, the immersion of $Z$ into $X_{bal}$ under the standard metric is a $(20L , 0)$-quasi-isometric embedding. 
\end{proof}

The following is a corollary of \Cref{thm: qi embedded d fifth}.
\begin{theorem}\label{thm: hypergraphs embed d fifth}
Let $d<1/5$. 
With overwhelming probability 
hypergraphs are embedded trees. 
\end{theorem}

\begin{proof}
We once again scale the metric on $X_{bal}$ by $1/L$. By \Cref{thm: qi embedded d fifth} for any hypergraph $Z$, with overwhelming probability the immersion map $Z \to X_{bal}$ restricts to a $(20,0)$-quasi-isometric embedding from $(Z, d_Z)$ to $(X^{(1)}_{bal}, \frac{1}{L}d_{X\oskel})$.  
Thus there is a distance $N$, that does not depend on $L$, such that the images of two points of distance $>N$ in $\Gamma$ are distinct in $X$. By \Cref{lem: local hypergraphs embed} we may assume that geodesic paths of length $\leq N$ in $\Gamma$ embed in $X$. This yields the claim. 
\end{proof}

\begin{cor}\label{cor: hypergraphs separate}
    Let $d<1/5$.   
    With overwhelming probability, all edge hypergraphs separate $X_{bal}$ into two components.
\end{cor}

    \begin{proof}
        Let $N(Z)$ be a small neighborhood of a hypergraph $Z$. Note that $N(Z) - Z$ has two components. Furthermore, note that $X_{bal}$ is simply connected so $H_1(X_{bal}) = 0$. By applying a Mayer--Vietoris argument to $N(Z)$ and $X_{bal} - Z$ we see that the number of components of $X_{bal} - Z$ is the same as the number of components of $N(Z)-Z$.
    \end{proof}

More generally, we obtain a similar result for hyperstructures.

\begin{proposition}\label{P: blowup walls contractible separate}
Let $d<1/5$.
    With overwhelming probability,  the hyperstructure $W_e^{\EG}$ associated to an edge $e$ is contractible and separates $\EGbal$ into two components. 
\end{proposition}

\begin{proof}
    The proof is the same as \cite[Lemmas 3.34 and 3.35]{martin_steenbock} except that we rely on \Cref{thm: hypergraphs embed d fifth} to show that projected hypergraphs are embedded trees in $X_{bal}$ rather than \cite[Lemma 3.32]{martin_steenbock}. 
\end{proof}

\subsection{Controlling intersections of hypergraph and vertex stabilizers}

The main result of this subsection is \Cref{P: X hypergraph stabilizers}, which ensures that hypergraph stabilizers have controlled intersections with the peripheral subgroups in our preferred relatively hyperbolic structure for $G$. We use \Cref{P: X hypergraph stabilizers} to satisfy one of the hypotheses of a relative cubulation criterion from \cite{EinsteinNg}. 

 \begin{proposition}\label{P: X hypergraph stabilizers} Let $d<1/5$. 
Let $Z$ be a hypergraph in $X_{bal}$, let $H = \Stab_G(Z)$, and let $v$ be a vertex of $X_{bal}$ with stabilizer $G_v$. With overwhelming probability, if $v$ does not lie in $Z$, then $G_v\cap \Stab_G(Z)$ is finite.
\end{proposition}

We use the fact that $X_{bal}$ is a fine graph (recall \cref{D: fine})  to show that the set of edges and vertices minimizing the distance from $v$ to $Z$ is finite. We then show that some finite index subgroup of $G_v \cap \Stab_G(Z)$ must act trivially on this finite set. The fineness of $X_{bal}\oskel$ implies that we can pass to a further finite index subgroup of $G_v\cap \Stab_G(Z)$ that stabilizes an edge of $X_{bal}$. 

We first prove the following lemma, which follows from the fact that projected hypergraphs are quasi-isometrically embedded.

\begin{lemma}\label{L: skirting the hypergraph} Let $d<1/5$. 
Let $Z$ be a hypergraph or projected hypergraph in $X_{bal}$, let $Y$ be the carrier of $Z$, let $Z^-$ be a component of $X_{bal}\oskel\setminus Z$.  
There is an affine function $\lambda:\R\to\R$ so that for any $x_1,x_2\in Z\cap X_{bal}\oskel$ there is an arc of length at most $\lambda(d(x_1,x_2))$ from $x_1$ to $x_2$ that does not intersect $Z^-$. 
\end{lemma}

\begin{proof}
\Cref{thm: qi embedded d fifth} implies that a geodesic path in $X_{bal}\oskel$ from $x_1$ to $x_2$ has length linearly related to the hypergraph distance $d_Z(x_1, x_2)$. In any polygon $C$ so that $Z$ passes through the interior of $C$,  by \Cref{P: epsilon antipodal projections}\eqref{item:e antipodal}, there are two paths between $Z\cap \partial C$ in $\partial C$ that have length at most $(\frac12+\epsilon_{d,\ell})|\partial C|$. One of these two paths, which we call $\sigma_C^-$, must avoid $Z^-$. We define $\sigma$ to be the concatenation of the paths $\sigma_C^-$, for all polygons $C$ such that the geodesic segment from $x_1$ to $x_2$ in $Z$ passes through its interior. The path $\sigma$ is a  path from $x_1$ to $x_2$, and its length is at most $(\frac12+\epsilon)d_Z(x_1, x_2)L$. Moreover, by construction, $\sigma$ avoids $Z^-$. 
If $\sigma$ is not an arc, it can be made an arc by eliminating any loops, which strictly shortens the path. 
\end{proof}

For the following two propositions, fix a projected hypergraph $Z$ and a vertex $v$ in $X_{bal}$. Let $D_v$ be the collection of edges dual to $Z$ or vertices on $Z$ that realize the minimum distance to $v$. That is, if $e\in D_v$ and $f$ is an edge dual to $Z$ or a vertex in $Z$ not in $D_v$, then $d_{X_{bal}\oskel}(v,e)<d(v,f)$. 

\begin{proposition}\label{P: D finite} Let $d<1/5$. 
Then the collection $D_v$ is finite. 
\end{proposition}

\begin{proof}
Suppose toward a contradiction that there is an infinite collection $e_0,e_1,e_2,\ldots$ so that each $e_i\in D_v$. 
Let $\gamma_0,\gamma_1,\gamma_2,\ldots$ be paths of minimal length from $v$ to $e_i\cap Z$ respectively. 
The paths $\gamma_i$ are geodesic, so if $\gamma_i,\gamma_0$ have common vertices $v,w$, we may reroute $\gamma_i$ so that $\gamma_i$ agrees with $\gamma_0$ between $v,w$. 
Thus we assume that if $\gamma_i$ and $\gamma_0$ have a vertex $w$ in common other than $v$, then the subpath of $\gamma_i$ from $v$ to $w$ agrees with the subpath of $\gamma_0$.

We show that there exists an infinite $I\subseteq \{0,1,2,\ldots\}$  and a vertex $v_0$ on $\gamma_0$ so that $0\in I$ and for all $i\in I$,  $\gamma_i\cap \gamma_0$ is exactly the subsegment of $\gamma_0$ from $v$ to $v_0$.
Our proof is by induction on the length of $\gamma_0$. If the length of $\gamma_0$ is $1$, then the $\gamma_i\cap \gamma_0 = \{v\}$ are pairwise disjoint because the $e_i$ are distinct. 

Now suppose that the length of $\gamma_0$ is $m$ and $f_0,f_1,f_2,\ldots$ are edges such that $f_i$ is the first edge of $\gamma_i$ issuing from $v$. We chose $\gamma_i$ so that if $\gamma_i\cap \gamma_0 \ne \{v\}$, then $\gamma_i$ and $\gamma_0$ must contain a common initial subsegment issuing from $v$. 
Therefore, if $f_i\ne f_0$ for infinitely many $i$, then there are infinitely many $i$ so that $\gamma_i\cap \gamma_0 = \{v\}$ so we can take $v_0=v$. 
Otherwise, there is an infinite $J\subseteq \{0,1,2,\ldots\}$ so that $0\in J$ and $i\in J$ implies $f_i = f_0$. 
Then we can obtain $v_0$ by applying the inductive hypothesis to the collection $\{\gamma_i\setminus f_i:\,i\in J\}$. This completes the proof of the claim. 

Thus there exists a vertex $v_0$, some infinite $I\subseteq \{0, 1, \dots\}$ with $0\in I$, and a collection of geodesic arcs $\{\rho_i:\,i\in I\}$ so that for all $i \in I$, $\rho_i$ issues from vertex $v_0$, $\rho_i$ is a subpath of $\gamma_i$, $|\rho_i|\le |\gamma_i|$, and $\rho_i\cap \rho_0 = \{v_0\}$. 

By \Cref{L: skirting the hypergraph} there are arcs $\sigma_1,\sigma_2,\ldots$ in $X^{(1)}$ so that $\sigma_i$ connects $e_i\cap Z$ to $e_0\cap Z$. Furthermore, since each $\gamma_i$ is geodesic, we have $|\sigma_i|\le 2\lambda |\gamma_i| = \lambda(d_{X_{bal}\oskel}(v,Z))$.

Consider the paths $\mu_i = \rho_0\cdot\sigma_i\cdot\rho_i$. These have length $|\mu_i| \leq 4\lambda(d_{X_{bal}\oskel}(v,Z))$. Furthermore 
each $\mu_i$ is distinct because each path $\rho_i$ has distinct endpoints. Since $\rho_i\setminus \{e_i\cap Z\}$ lies in $Z^-$, $\sigma_i$ does not intersect $Z^-$, and $\rho_i\cap \rho_0 = \{v_0\}$, each $\mu_i$ is an embedded loop containing the initial edge of $\rho_0$ issuing from $v_0$. 
However, there can only be finitely many $\mu_i$ because $X_{bal}\oskel$ is a fine graph by \Cref{prop: XR is fine}.
This contradicts the fact that $I$ is infinite. 
\end{proof} 

\begin{proposition}\label{P: extract edge stab} Let $d<1/5$. 
Let $H_v  = G_v\cap \Stab_G(Z)$. If $v$ does not lie in $Z$, then $H_v$ contains a finite index subgroup that stabilizes an edge of $X_{bal}$. 
\end{proposition}

\begin{proof}
Observe that $H_v$ fixes $v$ and $H_v$ fixes $Z$ setwise, so the action of $H_v$ takes points in $D_v$ to points in $Z$. 
Since points in $D_v$ realize the minimum distance from $v$ to $Z$, and since $H_v$ acts on $X_{bal}\oskel$ by isometries, the action of $H_v$ permutes the elements of $D_v$. 
Since $D_v$ is finite by \Cref{P: D finite}, there is a finite index $H_0\le H_v$ that fixes $D_v$ pointwise. 

If $e_0\cap Z$ is a midpoint of an edge, $H_0$ must fix $e_0$ and we are done. 
If $e_0$ is a vertex, there are finitely many geodesic paths from $v$ to $e_0$ because $X_{bal}$ is a fine hyperbolic graph by \Cref{prop: IPI in subdivision}. 
Since $H_0$ acts by isometries and fixes $e_0$, $H_0$ permutes the geodesic paths from $v$ to $e_0$. 
Thus some finite index subgroup $H_1\le H_0$ fixes some geodesic from $v$ to $e_0$ pointwise and therefore fixes an edge. 
\end{proof}

\begin{proof}[Proof of \Cref{P: X hypergraph stabilizers}]
By \Cref{lem: finite edge stabilizers} 
the edge stabilizers for the action of $G$ on $X_{\mc{R}}$ are trivial and after subdividing, the edge stabilizers of  $X_{bal}\oskel$ are trivial. If $v$ is not in $Z$, then there is a finite index subgroup of $H_v = G_v\cap \Stab_G(Z)$ that fixes an edge of $X_{bal}\oskel$ by \Cref{P: extract edge stab}. Hence $H_v$ is finite because a finite index subgroup is trivial.  
\end{proof}

\subsection{Cutting geodesics with hypergraphs}
The main result of this subsection is the following proposition. 

\begin{proposition}\label{P: many good walls}
Let $d < \frac{1}{6}$ and let $\gamma$ be a geodesic in $X_{bal}\oskel$ containing at least two vertices.
With overwhelming probability, there exists an edge hypergraph $Z$ so that $|Z \cap \gamma| = 1$. Furthermore, if $\gamma$ is infinite then with overwhelming probability there exists an $N>0$ so that every subpath of $\gamma$ of length $N$ intersects an edge hypergraph $Z$ which intersects $\gamma$ exactly once.
\end{proposition}

Both cases will follow from results of Ollivier and Wise in \cite{ollivier_wise}. They consider a Gromov random group and a collection of antipodal hypergraphs in its Cayley complex and establish several results about the behavior of hypergraphs and their carriers. Analogous results for $X_{bal}$ are listed in the following theorem.

\begin{theorem}\label{thm: OW hypergraph properties}
    Let $0<d<1/6.$ 
    The following hold with overwhelming probability in $X_{bal}$.
    \begin{enumerate}
        \item  Every reduced disk diagram $D$ with $Area(D)\geq 3$ contains at least three 2-cells that contribute strictly more than half their edges to $\partial D$.\label{I: strong greendlinger} (See \cite[Theorem 5.1]{ollivier_wise}.)
        \item  If $Z, Z'$ are (projected or edge) hypergraph rays intersecting in a 2-cell $C$, they either intersect in a 2-cell adjacent to $C$ as in \Cref{fig: OW52}, or they do not intersect anywhere else. \label{I: adjacent intersection} (See \cite[Corollary 5.2]{ollivier_wise}.)
        \item The carrier of a (projected or edge) hypergraph $Z$ is a convex subcomplex of $X_{bal}$.\label{I: hypergraph convex} (See \cite[Theorem 8.1]{ollivier_wise}.)
        \item  For all $p, q \in X_{bal}^{(0)}$ and for all $\epsilon>0$ we have 
        \[
            \#(p, q) \geq \frac{1}{2}\left(\frac{1}{6}-d-\varepsilon\right)(d(p, q) - 6L),
        \]
        where $\#(p, q)$ denotes the number of edge hypergraphs separating points $p, q$.\label{I: linear progress} (See \cite[Theorem 9.1]{ollivier_wise}.)
    \end{enumerate}
\end{theorem}

\begin{figure}
    \centering
    \includegraphics[width=0.3\linewidth]{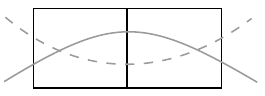}
    \caption{The unique minimal diagram collared by two hypergraphs.}
    \label{fig: OW52}
\end{figure}

The proofs of these claims for edge hypergraphs follow exactly the arguments in \cite{ollivier_wise}. Indeed, the results in \cite{ollivier_wise} are a direct consequence of the antipodal construction of hypergraphs and the fact that the Cayley complex $\tilde{X}$ satisfies a linear isoperimetric inequality $|\partial D|\geq (1-2d-\varepsilon)Area(D)L$ \cite[Theorem 1.6]{ollivier_wise}. In particular, Ollivier--Wise actually prove that these results hold in any simply connected 2-complex $\tilde{X}$ made of $L$-gons which satisfies this linear isoperimetric inequality, equipped with antipodal hypergraphs. 

It remains to prove that \eqref{I: adjacent intersection} and \eqref{I: hypergraph convex} also apply to projected hypergraphs. To this end, note that a key tool in the proofs of \cite{ollivier_wise} is to consider a diagram collared by hypergraphs and use the fact that hypergraphs are antipodal to argue that only corners can contribute more than $L/2$ edges to the boundary of the diagram. Though projected hypergraphs are not in general antipodal, diagrams collared by projected hypergraphs are sufficiently well behaved to allow the proofs of Ollivier--Wise to apply in our setting. The following lemma formalizes this.

\begin{lemma}\label{lem: markus trick}
    Let $Z_1, \dots, Z_k$ be hypergraphs in $X_{bal}$. Let $Y$ be a quasi-collared diagram, collared by $Z_1, \dots, Z_k$. If $C$ is a 2-cell in the collar of $Y$ that is not a corner, then $|C\cap \partial Y|\leq L/2$.
\end{lemma}
    \begin{proof}
        The vertices of $X_{bal}$ can be partitioned into the vertices that existed in $X_\mc{R}$, called \emph{original} vertices, and the vertices introduced by subdivision, called \emph{new} vertices. Note that original vertices around a 2-cell are at distance at least $2k$ from each other, where $X_{bal} = X_{\mc{R}}[k]$. 

        Let $C$ be a 2-cell of the collar of $Y$ and let $\alpha' = C\cap (\overline{Y-C})$. Note that $\alpha'$ is a path in the boundary of $C$ that is interior in $Y$. Since subdivision does not introduce new 2-cells, the endpoints of $\alpha'$ are both original vertices. If $C$ is not a corner, one of the $Z_i$ intersects $C$ with endpoints $x, x' \in \overline{C - \partial Y}$. Let $\alpha$ be the path in $\overline{\partial C - \partial Y}$ from $x$ to $x'$. Note that $\alpha'$ is an extension of $\alpha$. By \Cref{P: epsilon antipodal projections}, $|\alpha| \geq (1/2-\epsilon_{d, \ell})L$. Since $L = 2k\ell$ and $\epsilon_{d, \ell}<1/\ell$, we have $\epsilon_{d, \ell} <2k/L.$
        Thus $|\alpha|/2k > L/4k-1$.  Hence the number of original  vertices on $\alpha$, and also on $\alpha'$, is at least $L/4k =\ell/2$. But since the endpoints of $\alpha'$ are original vertices, this shows that $|\alpha'|\geq L/2$, as desired. 
    \end{proof}

In the language of \cite{ollivier_wise}, \Cref{lem: markus trick} states that any pseudoshell of the diagram $Y$ collared by $Z_1, \dots, Z_k$ has to be a corner. Thus  
 the proofs of \eqref{I: adjacent intersection}, \eqref{I: hypergraph convex} of \cref{thm: OW hypergraph properties} for projected hypergraphs are identical to the proofs given in \cite{ollivier_wise}, substituting \Cref{lem: markus trick} for the antipodal construction of the hypergraphs when necessary.

\begin{proof}[Proof of \Cref{P: many good walls}]
    Let $\gamma$ be an infinite geodesic in $X_{bal}^{(1)}$. Let $N > 6L$. Then by \cref{thm: OW hypergraph properties}\eqref{I: linear progress} along any subpath of $\gamma$ of length $N$ there is at least one edge hypergraph $Z$ which intersects the subpath of 
    $\gamma$ exactly once. 
    If $Z$ intersects $\gamma$ again at a point at distance $> N$, then by convexity of the carrier of $Z$ (\cref{thm: OW hypergraph properties}(\ref{I: hypergraph convex})), the subpath of $\gamma$ between the two intersection points lies in the carrier, and we get a disk diagram $Y$ that cannot satisfy \cref{thm: OW hypergraph properties}(\ref{I: strong greendlinger}) because of \cref{lem: markus trick}.

    Suppose instead that $\gamma$ is a finite geodesic segment. Refer to \cref{fig: good walls finite case}. Note that by \cref{thm: hypergraphs embed d fifth} no edge hypergraph can meet $\gamma$ twice in the same 2-cell, so $\gamma$ is not contained in a single 2-cell.
    Let $Z$ denote the edge hypergraph closest to one of the endpoints of $\gamma$, and let $x\neq y$ denote the first two intersection points of $Z$ with $\gamma$. By \cref{thm: OW hypergraph properties}\eqref{I: hypergraph convex}, which applies to the edge hypergraph $Z$, the subpath of $\gamma$ from $x$ to $y$ lies in the carrier of $Z$. Let $C_1, C_2$ denote the first two 2-cells of the carrier of $Z$.

    First suppose that $|\gamma\cap C_1| = 1$. Then $|C_1 \cap C_2| \geq L/2 -1$. Indeed,  let $e_1= \gamma\cap C_1$, and let $e_2$ be the opposite edge in $\partial C_1$. Note that $e_2\in C_1\cap C_2$. Also, by convexity of the carrier of $Z$ (\cref{thm: OW hypergraph properties}(\ref{I: hypergraph convex})) the intersection of $\gamma$ with $C_1\cap C_2$ is non-empty, hence, $e_1\cap C_2$ is non-empty. As $e_1$ and $e_2$ are opposite in $\partial C_1$, this yields the claim.     
    On the other hand, the isoperimetric inequality from \cref{thm: global rel IPI} (see also \cref{prop: IPI in subdivision}) implies that $\partial(C_1\cup C_2)$ has at least $\frac43 L$ edges, so $|C_1 \cap C_2| \leq L/3$. As $L \to \infty$ this gives a contradiction. Thus we may assume that $|\gamma\cap C_1|>1$. Let $r = Z \cap (C_1 \cap C_2)$, and let $p, s$ be the endpoints of $C_1 \cap C_2$ so that $s \in \gamma$. Let $Z'$ be the edge hypergraph incident to the edge of $\overline{xs}$ closest to $s$. Then $Z \neq Z'$. Note that since $\gamma$ is a geodesic and $r \notin \gamma$, we have $\gamma \cap (C_1 \cap C_2) = s$. As in the previous case,  we have $|C_1 \cap C_2| = d(p, s)=d(s,r)+d(r,p) \leq L/3.$ Since $d(x, s) + d(s, r) = L /2,$ we get that $d(x, s) \geq d(r, p) + L/6.$ On the other hand, if the other endpoint of $Z'\cap C_1$ is in $\overline{ps} = C_1 \cap C_2$, then $d(x, s) < d(r, p)$, which is a contradiction. Thus $Z'\cap C_1 \cap C_2 = \emptyset$.  This implies that $Z'\cap C_2 = \emptyset$. Indeed, if this is not the case then $Z \cap Z'$ includes a point in $C_2$, so by \cref{thm: OW hypergraph properties}\eqref{I: adjacent intersection} $Z'$ would cross an edge of $C_1 \cap C_2$. 

    Suppose that $Z'$ intersects $\gamma$ at least two times. We claim that $Z$ and $Z'$ intersect in a $2$-cell $C\not =C_1$. Indeed, let $a$ denote the point $Z'\cap C_1 \cap \gamma$ and let $b$ be another point in $Z'\cap \gamma$. 
    Note that, by the choice of $x$, $\overline{ab} \subset \gamma$ does not contain $x$. We thus have two cases: either $\overline{ab}\subset \overline{ay}$ or $\overline{ay}\subset \overline{ab}$. In the first case, by convexity of carriers (\cref{thm: OW hypergraph properties}(\ref{I: hypergraph convex})), every edge of $\gamma$ with non-empty intersection with $\overline{ab}$ lies in the carrier of $Z$ and in the carrier of $Z'$. Since $Z'$ is dual to an edge $e$ of $\overline{xy}$, $Z$ and $Z'$ cross in any $2$-cell with $e$ in its boundary. This completes the proof in the first case. The argument in the second case is similar. Thus $Z$ and $Z'$ intersect at $C\not =C_1$. By \cref{thm: OW hypergraph properties}(\ref{I: adjacent intersection}), $Z$ and $Z'$ have to intersect at either $C_2$ or at $C_3$ (\cref{fig: good walls finite case}). $Z'$ does not enter $C_2$, so the only possibility is $C_3$. 

    Consider the diagram induced by the union of $C_1, C_2, C_3$. By \cref{thm: OW hypergraph properties}(\ref{I: strong greendlinger}), all three 2-cells must contribute at least $\frac{L}{2}$ to the boundary. But $C_1$ contributes at most $L - (d(x, s) + d(s, p)) < L/2$ to the boundary. This is a contradiction, so $Z' \cap \gamma = \{a\}$. 
\end{proof}

\begin{figure}[h]
    \centering
    \includegraphics[width=.8\textwidth]{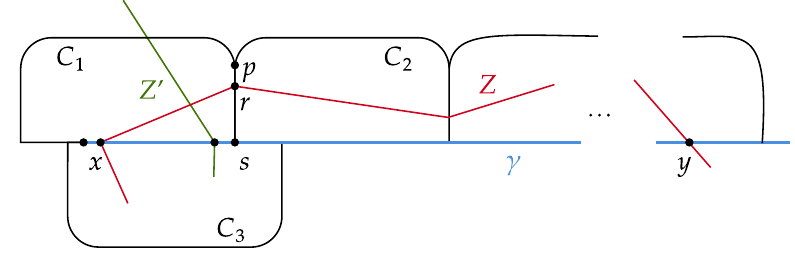}
    \caption{If $\gamma$ is a finite geodesic, there is an antipodal hypergraph that crosses $\gamma$ exactly once.}
    \label{fig: good walls finite case}
\end{figure}

\section{Applying a relative cubulation criterion}\label{sec: applying EN criterion}

In this section, we prove \Cref{thm: rel cubulation at d sixth} using the following relative cubulation condition from \cite{EinsteinNg}, which uses \cite{EMN-boundary-criteria}, and ideas from \cite{BergeronWise,EG:RelGeom} to obtain a cubulation. 

Throughout, we let $G\sim\mathcal{FPD}(\mc{G}; d, m, \ell)$. We assume that each factor has a relatively hyperbolic structure $(G_i, \mathcal{P}_i)$, and that each of these structures admits a relatively geometric cubulation. We denote by $\mathcal{P}$ the union $\bigcup_i \mc P_i$. Note that, with overwhelming probability, $(G,\mc P)$ is relatively hyperbolic.

The statement of \cref{T: blowup cubulation} below has been adapted to the context of $\EGbal$ and $X_{bal}$ of this paper for ease of reading. We will define the terms and conditions in the theorem after stating it, and then show how results from previous sections verify these conditions. 

\begin{theorem}[{\cite[Corollary~8.22]{EinsteinNg}}]\label{T: blowup cubulation}
 If $X_{bal}$ has suitable walls and $\EGbal$ satisfies the projected wall tree, projected wall fullness and two-sided wall projection properties, then $(G,\mc P)$ acts relatively geometrically on a CAT(0) cube complex.
\end{theorem}

We now define the conditions in the hypothesis of \cref{T: blowup cubulation}, again adapted to our context to simplify notations for the reader. 

\begin{defn}[\cite{EinsteinNg} Hypotheses 6.2]\label{D: suitable walls}
The polygonal complex $X_{bal}$ satisfies the \emph{suitable walls} condition if: 
\begin{enumerate}
\item for any edge $e$ of $X_{bal}$, the hypergraph $Z_e \in \mathcal{Z}$ is an embedded tree whose inclusion map into $X_{bal}$ is a quasi-isometric embedding,     \label{I: hypergraphs trees}
\item any hypergraph $Z \in \mathcal{Z}$ separates $X$ into two distinct complementary components, \label{I: hypergraph separates}
\item for any $Z \in \mathcal{Z}$ and any vertex $v \in X$, $\Stab(Z)\cap \Stab(v)$ is finite, and
\label{I: protofull}
\item if $\gamma$ is a combinatorial geodesic, then there exists an edge $e$ of $\gamma$ so that $Z_e$ crosses $\gamma$ exactly once. 
If $\gamma$ is infinite, there exists an $N\in\mathbb{N}$ so that every subsegment of $\gamma$ with length $N$ contains an edge $e$ so that $Z_e$ intersects $\gamma$ exactly once. \label{I: many good walls}
\end{enumerate}
\end{defn}

\begin{proposition}
\label{prop: suitable_walls}
    Let $G\sim\mathcal{FPD}(\mc{G}; d, m, \ell)$ with $d<\frac16$.
    The complex $X_{bal}$ satisfies the suitable walls condition as in \Cref{D: suitable walls}. 
\end{proposition}

\begin{proof}
    We verify \Cref{D: suitable walls} item by item. \Cref{thm: hypergraphs embed d fifth} and \Cref{thm: qi embedded d fifth} imply Item~\eqref{I: hypergraphs trees}. \Cref{cor: hypergraphs separate} implies Item~\eqref{I: hypergraph separates}. 
    \Cref{P: X hypergraph stabilizers} implies Item~\eqref{I: protofull}. 
    \Cref{P: many good walls} implies Item~\eqref{I: many good walls}. 
\end{proof}

Recall that there is a projection map $p : \EGbal \to X_{bal}$ from \cref{def: projection EX to X}.

\begin{defn}[\cite{EinsteinNg} Definition 8.10]\label{D: blowup projection properties}
We say
\begin{enumerate}
\item  $(\EGbal,p)$ has the \emph{projected wall tree property} if $p$ projects every $\EGbal$--hyperstructure to an embedded tree whose inclusion into $X_{bal}$ is a quasi-isometric embedding,   
\item $(\EGbal,p)$ has the \emph{projected wall fullness property} if whenever $W_e^{\EGbal}$ is an $\EGbal$--hyperstructure and $v$ is a vertex of $X_{bal}$,  $\Stab_G(W_e^{\EGbal})$ has infinite intersection with $\Stab_G(v)$ if and only if $p(W_e^{\EGbal})$ intersects $v$, and 
\item $(\EGbal,p)$ has the \emph{two-sided wall projection property} if every $W_e^{\EGbal}$ separates $\EGbal$ into two complementary components such that the corresponding closed half-spaces $U^+$ and $U^-$, which satisfy $U^+ \cap U^- = W_e^{\EGbal}$, then $p(U^+)\cap p(U^-) = p(W_e^{\EGbal})$.
\end{enumerate}
\end{defn}

We are now ready to state the main result of this section:

\begin{theorem}\label{T: endgame}
Let $G\sim\mathcal{FPD}(\mc{G}; d, m, \ell)$ with $d<\frac16$ so that each factor has a relatively hyperbolic structure $(G_1,\mc{P}_1),\ldots,(G_n,\mc{P}_n)$ and each of these structures admits a relatively geometric cubulation. Let $\mc{P}$ be the union  $\bigcup_i\mc{P}_i$. Then with overwhelming probability $(G,\mc{P})$ is relatively hyperbolic and acts relatively geometrically on a CAT(0) cube complex. 
\end{theorem}

\begin{proposition}\label{P: projection properties}
    Let $G\sim\mathcal{FPD}(\mc{G}; d, m, \ell)$ with $d<\frac16$. Then 
 $(\EG_{bal},p)$ satisfies the projected wall tree, projected wall fullness, and two sided wall projection properties. 
\end{proposition}

\begin{proof}
      Recall from the end of \cref{sec: hyperstructures pc complex} that a hypergraph in $X_{bal}$ is either an edge hypergraph in $X_{bal}$ or a projection  of an $\EG_{bal}$--hyperstructure. 
    In particular, \Cref{thm: qi embedded d fifth} and \Cref{thm: hypergraphs embed d fifth} imply that $(\EG_{bal},p)$ satisfies the projected wall tree property. 
    \Cref{P: X hypergraph stabilizers} implies the projected wall fullness property.
    For the two sided wall projection property, we see from \Cref{P: blowup walls contractible separate} that $W=W_e^{\EGbal}$ is two-sided. Let $U^+$ and $U^-$ be closed half-spaces so that $U^+\cap U^-=W_e^{\EGbal}$. Let $x \in p(U^+) \cap p(U^-)$.  
The projection $p: \EG_{bal} \to X_{bal}$ restricts to a homeomorphism on the preimage of each open 1--cell or 2--cell of $X_{bal}$, moreover, the preimage of each 0--cell is a (connected) fiber complex.  

Hence, if $x$ is contained in the interior of a 1--cell or 2--cell then there is a unique point $\tilde{x} \in p^{-1}(x)$ and it must be that $\tilde{x} \in U \cap U^-$.  
Otherwise, $x$ is a 0--cell of $X_{bal}$ with preimage $E_x$ a fiber complex.  
Since $x \in p(U^+)$ it must be that $E_x \cap U^+ \neq \varnothing$, and similarly, $E_x \cap U^- \neq \varnothing$.
Since $U^+,U^-$ are half spaces of $W$ that both intersect $E_x$, $U^+\cap E_x, \, U^-\cap E_x$ are half spaces of a hyperplane $W\cap E_x$ in $E_x$. Thus there exists a point $\tilde{x} \in E_x \cap (U^+ \cap U^-)$ such that $p(\tilde{x}) = x \in p(U^+ \cap U^-)$ as desired.  
The reverse inclusion is obvious.
\end{proof}

\begin{proof}[Proof of \Cref{T: endgame}]
    The result follows by combining the cubulation criterion \cref{T: blowup cubulation} with \Cref{P: projection properties}. 
\end{proof}

\begin{remark}\label{R: cubulated rel to factors}
For an arbitrary random quotient of a free product, the factor $G_i$ acts relatively geometrically on a point with respect to the structure $(G_i,\{G_i\})$. Then take $\mc{P}$ to be the collection of free factors,  $\EG_{bal} = X_{bal}= X_{\mc R}$, with projection $p = id$, and apply \Cref{T: blowup cubulation} to obtain the first part of \cref{thm: rel cubulation at d sixth}.
 \end{remark}

\section{Geometrically cubulating}
\label{S: geometric cubulation}

The goal of this section is the following.

\begin{theorem}\label{cor: actually cubulated}
    If $G\sim\mathcal{FPD}(\mc{G}; d, m, \ell)$ with $d<\frac16$ and each of the factors $G_1,\ldots, G_n$ admits a proper and cocompact cubulation, then with overwhelming probability, $G$ acts properly and cocompactly on a CAT(0) cube complex.
\end{theorem}

Throughout this section, let $G\sim\mathcal{FPD}(\mc{G}; d, m, \ell)$ with $\mc{G} = \{G_1, \dots, G_n\}$ and $d<\frac{1}{6}$. Assume that each $G_i$ is cubulable. If $G_i$ is finite, let $Y_i$ be a single point. Otherwise, let $Y_i$ be an essential  CAT(0)  cube complex on which the group $G_i$ acts properly and cocompactly; that is, every hyperplane of $Y_i$ splits $Y_i$ into two deep components. We refer the reader to \cite[Section 3.1]{caprace_sageev} for the definitions. Each $Y_i$ will be the essential core of the complex $Q_i$ considered in \cref{sec: hyperstructures pc complex}. It is guaranteed to exist by \cite[Proposition 3.5]{caprace_sageev}. Throughout this section, we fix $\EGbal$ to be the mixed even polygonal-cubical complex with fiber complexes $Y_1, \dots, Y_n$.

The conditions in \cref{D: suitable walls} and \cref{D: blowup projection properties} guarantee that (up to passing to an index at most 2 subgroup to avoid inversions) the stabilizers of all $\EG_{bal}$--hyperstructures and edge hypergraphs in $X_{bal}$ are codimension-1 subgroups of $G$ (see for example \cite[Section 6 and 8]{EinsteinNg} and references therein).  

For the remainder of this article, fix $C(G)$ to be the CAT(0) cube complex dual to this $G$-finite collection of codimension-1 subgroups obtained via the Sageev construction (see \cite[Section~2.6]{hruska_wise_2014} for the precise formalism in this setting). 
In \cref{S: proper action} we prove the induced action of $G$ on $C(G)$ is proper and then in \cref{S: cocompact} we prove the action is cocompact. Note that we cannot directly apply \Cref{T: endgame} because there is no relatively geometric action on the factor complexes. Indeed, the factor groups are not necessarily hyperbolic (relative to the trivial subgroup). 

\subsection{Properness}\label{S: proper action}

The main tool used to prove properness is the following result.

\begin{theorem}[{\cite[Theorem 5.1]{BergeronWise}}] \label{T: bw proper}
Let $(G,\mc{P})$ be a relatively hyperbolic pair. Suppose that:
\begin{enumerate}
\item \label{I: reconstruction of fiber}
for each parabolic point $q \in\partial_{\mc P}(G)$, there exist finitely many quasi-isometrically embedded finitely generated codimension--$1$ subgroups of $G$ whose intersections with $\Stab_G(q)$ yields a proper action of $\Stab_G(q)$ on the corresponding dual cube complex, and 
\item 
\label{I: bw proper bdd cut} 
for each pair of distinct points $u,v\in\partial_{\mc P}(G)$ there is a quasi-isometrically embedded finitely generated codimension--$1$ subgroup $H$ such that $u,v$ lie in $H$--distinct components of $\partial_{\mc P} G\setminus\partial H$. 
\end{enumerate}
Then there exists a subcollection of finitely many quasi-isometrically embedded f.g. codimension--$1$ subgroups of $G$ such that the action of $G$ on the dual cube complex is proper. 
\end{theorem}

\begin{remark}\label{R: no subcollection required}
     Even though \cref{T: bw proper} as stated gives a (finite) subcollection of codimension--1 subgroups with respect to which the dual cube complex is proper, it is implicit (see \cite[Lemma 5.4]{BergeronWise}) that, when the collection of codimension--1 subgroups we start with is finite, we do not need to pass to a subcollection. This is the setting that will be relevant to us. We will show that the $G$--finite collection of stabilizers of the $\EGbal$--hyperstructures and the edge hypergraphs of $X$ satisfy the hypotheses of \cref{T: bw proper}. Therefore, the action on the dual cube complex will be proper. See \cref{T: factors geometric implies proper}.
\end{remark}

We will first show that $\EGbal$--hyperstructures are quasi-isometrically embedded. Recall that $p:\EGbal\to X_{bal}$ is the natural projection. 

\begin{proposition}
\label{P:QIembedded_walls}
Let $W$ be a hyperstructure in $\EG_{bal}$. Then $W$ is quasi-isometrically embedded in $\EG_{bal}$.
\end{proposition}

Let $W$ be a hyperstructure in $\EGbal$, let $N(W)$ be its carrier, and let $a, b$ be arbitrary points in $W \cap \EGbal\oskel$. Let $\gamma$ be a geodesic in $\EGbal\oskel$ from $a$ to $b$ and let $\alpha$ be a shortest path in $N(W)\cap \EGbal\oskel$ from $a$ to $b$. Our goal is to show that $|\alpha|$ is uniformly bounded above by a linear function in $|\gamma|$. We do this by cases. There exist minimal decompositions of $\alpha$ and $\gamma$ as 
        \[
            \alpha = \alpha_1 * \dots * \alpha_m.
            \qquad 
            \gamma = \gamma_1 * \dots * \gamma_m
        \]
    so that for each $1 \leq i \leq m$ one of the following holds:
        \begin{enumerate}
            \item [(Case 1)]
                $\alpha_i, \gamma_i$ have the same image in $\EGbal$, 
            \item [(Case 2)]
                $p(\alpha_i), p(\gamma_i)$ have disjoint interiors, or
            \item [(Case 3)]
                $\alpha_i, \gamma_i$ have distinct images in $\EGbal$ but are contained in the same fiber complex. 
        \end{enumerate}

To obtain this minimal decomposition first identify the maximal subpaths of $\alpha$ and $\gamma$ with the same image in $\mc{E}X_{bal}$. Let $\alpha', \gamma'$ be the complements. These split into subpaths $\{\alpha'_i\}, \{\gamma'_i\}$, so that the endpoints of $\alpha'_i$ and $\gamma'_i$ coincide in $\mc{E}X_{bal}$. Fix $i$ and consider $p(\alpha'_i)$ and $p(\gamma'_i)$ in $X_{bal}$. There are maximal subpaths which map to disjoint images in $X_{bal}$. By construction, all other subpaths must lie in the same fiber complex. This completes the decomposition.

In general, it is not true that a geodesic $\gamma$ in $\EGbal$ projects to a geodesic in $X_{bal}$. However, the image of such a geodesic has controlled intersection with polygons. Recall that in this section, $\epsilon$ and the subdivision parameter $k$ have been chosen as defined in \cref{R: bounded length in fiber}. 

\begin{lemma}\label{lem: proj geo almost antipodal}
    Let $c$ be a polygon in $X_{bal}$. If $\gamma$ is a geodesic in $\EGbal$ then $|p(\gamma) \cap c| < L(\epsilon +\frac{1}{2}).$ Similarly, suppose that $W$ is a hyperstructure in $\EGbal$, and $\alpha$ is a shortest path in $N(W)\cap \EGbal\oskel$. Then $|p(\alpha) \cap c| < L(\epsilon +\frac{1}{2}).$
\end{lemma}
    \begin{proof}
        Let $c$ be a polygon in $X_{bal}$. There is a unique polygon $\tilde{c}$ in $\EGbal$ so that $p(\tilde{c}) = c$. If $\gamma$ is geodesic, then $|\gamma \cap \partial \tilde c|\le \frac{1}{2}|\partial \tilde c|.$ By the construction in \Cref{sec: subdividing}, the map $p: \EGbal \to X_{bal}$ collapses cubical edges, so $|p(\gamma) \cap c| < L(\epsilon +\frac{1}{2}).$ Similarly, if $\alpha$ is a shortest path in $N(W)$ and $\tilde c \subset N(W)$ then $|\alpha \cap \partial \tilde c| \le \frac{1}{2}|\partial \tilde c|.$ Thus $|p(\alpha) \cap  c| < L(\epsilon +\frac{1}{2}).$ 
        
        Suppose instead that $\tilde c \nsubset N(W)$ and $|p(\alpha) \cap \partial c| \geq L(\epsilon +\frac{1}{2}).$ Let $e$ be the first edge in $c \cap p(\alpha)$. Let $p(W')$ be the projected hypergraph through $e$, and note that $W' \neq W$. Since $p(W')$ is $\epsilon$-antipodal, $p(W')$ must cross another edge $e' \in p(\alpha) \subset p(N(W))$, hence $p(W) \cap p(W')$ contains at least two distinct points. By \Cref{thm: OW hypergraph properties}\eqref{I: adjacent intersection} these points lie in adjacent 2-cells $d, d' \in p(N(W))$. Then $D = d\cup d' \cup c$ is a disk diagram collared by $p(W)$ and $p(W')$, with area $3$ and at most 2 corners. Since $d \neq d'$, this is impossible by \Cref{thm: OW hypergraph properties}\eqref{I: adjacent intersection}.
    \end{proof}

We now analyze the behavior of a Case 2 pair of arcs, using their projections to $X_{bal}.$ Let $\tau$ be the uniform constant from \cref{R: bounded length in fiber} that bounds the length of cubical segments in $\EGbal$.

\begin{lemma}\label{lem: bounded endpoint distance}
Let $d<\frac16$ and let $1\leq i \leq m$. Suppose $s_\alpha, s_\gamma$ are the starting points of $\alpha_i, \gamma_i$, respectively, and let $t_\alpha, t_\gamma$ be their ending points.  
Then
$d_{\EGbal\oskel}(s_\alpha, s_\gamma) \le \tau$ and $d_{\EGbal\oskel}(t_\alpha, t_\gamma) \le \tau$.
\end{lemma}
    \begin{proof}
        If $i$ is in Case (1), the endpoints of $\alpha_i, \gamma_i$ are equal to each other. If $i$ is in Case (3) and $1<i<m$, then $i-1$ and $i+1$ are in Case (1) or in Case (2). Otherwise,  $i=1$ or $i=m$, and either $m=1$, or $i+1$, respectively, $i-1$, is in Case (1) or Case (2). So it suffices to prove this for the endpoints in Case (2).
    
        Suppose $\alpha_i, \gamma_i$ are in Case (2). Then $p(s_\alpha) = p(s_\gamma)$ and $p(t_\alpha) = p(t_\gamma)$. Let $D$ be a  reduced disk diagram bounded by $p(\alpha_i), p(\gamma_i)$. By \Cref{lem: greendlinger} and the choice of the subdivision parameter $k$ in \Cref{sec: subdividing}, either $|D| = 1$ or $D$ contains at least two $\epsilon$-supershells (see \Cref{def: supershell}).  
 
        We will show that the first edge of $p(\alpha_i)$ and the first edge of $p(\gamma_i)$ are part of the same polygon of $D$. Let $c_{\alpha}$ be the polygon of $D$ adjacent to the first edge of $p(\alpha_i)$, and let $c'_{\alpha}$ be the polygon of $D$ adjacent to the last edge of $p(\alpha_i)$. We define $c_\gamma$ and $c'_\gamma$ analogously. By \Cref{lem: proj geo almost antipodal}, the intersections of any $\epsilon$-supershell with $p(\alpha_i)$, and $p(\gamma_i)$  respectively, must both contain at least one edge. Moreover, these edges are on $\partial D$. Let $c$ be the first polygon in $D$ adjacent to $p(\alpha_i)$, going along $p(\alpha_i)$ from $p(s_{\alpha})$ to $p(t_{\alpha})$, that is an $\epsilon$-supershell. Note that $c$ may be equal to $c_\alpha$ or $c'_\alpha$. 
        
        We claim that $D\setminus c$ splits into at least two components unless $c_{\alpha}=c_{\gamma}=c$ or $c=c'_{\alpha}=c'_\gamma$. Indeed, recall that the intersections of $c$ with $p(\alpha_i)$ and $p(\gamma_i)$ both contain at least one edge, and that these edges need to be in $\partial D$. Thus, removing $c$ we see that the boundary path is cut into (at least) two components, hence, the same holds for $D$. 
        
        Suppose $c_{\alpha}\not=c_{\gamma}$.  If $c'_{\alpha} \neq c'_\gamma$ or if $c \neq c'_{\alpha} = c'_{\gamma}$, let $D_1$ be the component of $D \setminus c$ containing $c_\gamma$, and let $D' = D_1 \cup c$. Then $D'$ contains $c_\alpha$ and $ c_\gamma$, but $c$ is the only $\epsilon$-supershell of $D'$. This contradicts \Cref{lem: greendlinger}. Hence $c = c'_{\alpha} = c'_{\gamma}$. But then $c$ is the only $\varepsilon$-supershell of $D$, which also contradicts \Cref{lem: greendlinger}.  

        So $c_\alpha = c_\gamma$ and $c'_\alpha = c'_\gamma$. Hence there is a path from $s_\alpha$ to $s_\gamma$ that lies in the boundary of the polygon $\tilde c$ satisfying $p(\tilde c) = c_\alpha$. By \Cref{R: bounded length in fiber}, there is a constant $\tau>0$ so that $d(s_\alpha, s_\gamma)<\tau$. Similarly, $d(t_\alpha, t_\gamma)< \tau.$
    \end{proof}

    We are now ready to prove \Cref{P:QIembedded_walls}.

    \begin{proof}[Proof of \Cref{P:QIembedded_walls}]
        Let $a, b \in N(W)$, let $\gamma$ be a geodesic from $a$ to $b$, and let $\alpha$ be a shortest path in $N(W)$ from $a$ to $b$. Since $\gamma$ is a geodesic, $|\gamma|\leq |\alpha|$. Thus it suffices to show $|\alpha|$ is at most a linear function of $|\gamma|$.  We demonstrate this by showing it is true for each pair $\alpha_i, \gamma_i$.
        
        If $i$ is in Case (1), then $|\alpha_i| = |\gamma_i|$. Suppose that $i$ is in Case (2). By \Cref{thm: qi embedded d fifth} $p(W)$ is $(\lambda,\epsilon)$--quasi-isometrically embedded in $X_{bal}$. The endpoints of $p(\alpha_i)$ are within $\frac L 2$ of $p(W)$, thus $p(\alpha_i)$ is a $(\lambda,\epsilon+L)$--quasigeodesic in $X_{bal}\oskel$. Hence $|p(\gamma_i)|\ge \frac1\lambda |p(\alpha_i)| -\epsilon -L$.  Therefore, 
        \begin{align*}
            |\alpha_i| & \leq 2\tau |p(\alpha_i)|+ 2\tau \\
            & \leq 2 \tau \lambda  |p(\gamma_i)| + 2\tau \lambda (L+\epsilon) \leq 2 \lambda \tau |\gamma_i| + 2\tau \lambda (L+\epsilon).
            \end{align*}
            Recall that $\tau$ is the constant given by \Cref{R: bounded length in fiber}.
        If $i$ is in Case (3), then $\alpha_i, \gamma_i$ lie in a common fiber CAT(0) cube complex. Since $\alpha_i$ is a geodesic in a hyperplane carrier, which is a convex subcomplex of a CAT(0) cube complex, we have that 
        $\alpha_i$ is geodesic in the fiber complex. 
        By the triangle inequality and \Cref{lem: bounded endpoint distance}, 
            \[
                |\alpha_i|\leq |\gamma_i|+2\tau \leq |\gamma_i|(1 + 2\tau). 
            \qedhere \]
    \end{proof}

\begin{lemma}\label{L:hypergraph lifts to qi wall}
    Let $H$ be an edge hypergraph of $X_{bal}$. There is a unique lift $W$ of $H$ to $\EGbal$, and $W$ is quasi-isometrically embedded in $\EGbal$.
\end{lemma}
\begin{proof}
Since $H$ does not contain any vertices, $p\inv(h)$ is a single point for any $h \in H$. Thus we can lift the embedding $H\hookrightarrow X_{bal}$ to $H\hookrightarrow W = p\inv(H)\subseteq \EGbal$ and $\Stab_G(W) = \Stab_G(H)$. We metrise $W$ by $d_{W}(a,b) =d_H(p(a),p(b))$. 
Let $a,b\in W$ and let $\gamma$ be an $\EGbal\oskel$--geodesic between $a,b$. Observe that 
\begin{align*}d_{\EGbal\oskel}(a,b) & = |\gamma|\ge |p(\gamma)| \ge d_{X_{bal}\oskel}(p(a),p(b)) \\ & \ge \frac1{20 L} d_W(p(a),p(b)) = \frac1{20L} d_{\hat W}(a,b)   \end{align*}
by \Cref{thm: qi embedded d fifth}. Let $\hat L\geq L + 2\tau\ell$, so that $\hat L$ is an upper bound on the number of sides of a polygon in $\EGbal$. Then there is a path in $\EGbal\oskel$ between $a$ and $b$ of length at most $\hat L d_{\hat W}(a,b)$. Thus $W$ is quasi-isometrically embedded in $\EGbal$. 
\end{proof}

Since any hyperstructure in $\EGbal$ projects to a hyperstructure in the compact $\rightQ{\EGbal}{G}$, the stabilizer of any hyperstructure $W$ acts cocompactly on $W$. We thus have the following. 

\begin{proposition}\label{P: wall stabilizers qc in G}
    Let $W$ be an $\EGbal$-hyperstructure or the lift of an edge hypergraph in $X_{bal}$. Then $\Stab_G(W)$ acts cocompactly on W, and is quasi-isometrically embedded in $G$. \qedhere
\end{proposition}

\begin{theorem}\label{T: factors geometric implies proper}
    If $G\sim\mathcal{FPD}(\mc{G}; d, m, \ell)$ with $d<\frac16$ and each of the factors $G_1,\ldots, G_n$ admits a proper and cocompact cubulation, then $G$ acts properly on the CAT(0) cube complex $C(G)$.
\end{theorem}

\begin{proof} 
Since $X_{bal}$ has suitable walls (see \cref{prop: suitable_walls}),  Proposition 6.10 of \cite{EinsteinNg} shows that the endpoints in $X_{bal}\cup \partial X_{bal}$ of any geodesic are separated by at least one edge hypergraph. \cite[Theorem 6.12]{EinsteinNg} then shows that  \Cref{T: bw proper}(\ref{I: bw proper bdd cut}) is satisfied. 

For $p\in X_{bal}$, if $\Stab_G(p)$ is infinite then $\Stab_G(p)$ is conjugate to some $G_i$. Moreover, any $\EGbal$--hyperstructure stabilizer that intersects $\Stab_G(p)$ intersects $\Stab_G(p)$ in exactly one hyperplane stabilizer. Indeed, by \Cref{thm: hypergraphs embed d fifth} any $\EGbal$--hyperstructure projects to a tree in $X_{bal}$, so they cannot intersect a fiber complex more than once. Cubulating over the hyperplanes of the fiber cube complex recovers the geometric action on the fiber complex corresponding to $p$, proving \Cref{T: bw proper}(\ref{I: reconstruction of fiber}) is satisfied. \cref{T: bw proper}, along with \cref{R: no subcollection required}, gives the required result.
\end{proof}

\subsection{Cocompactness}\label{S: cocompact}
We now show that the action of $G$ on the dual cube complex $C(G)$ is cocompact. Let $P = G_i \in \{G_1, \dots, G_n\}$ and let $v_P$ be the vertex in $X_{bal}$ fixed by $P$. Let $Y = Y_i$ be the corresponding fiber complex. Note that every edge-hypergraph $H$ in $X_{bal}$ has a unique lift $\tilde{H}$ in $\EGbal$, and furthermore since $H \cap X_{bal}^{(0)} = \emptyset$, $\tilde{H}$ is a wall in $\EGbal$.
Let $\mathcal V$ denote the set of the halfspaces of $\EGbal$ given by the $\EGbal$--hyperstructures and lifts of edge-hypergraphs in $X_{bal}$. 
That is, given a hyperstructure or lifted edge hypergraph $W$ and halfspaces $U_W, U_W'$ satisfying $\EGbal \setminus W = U_W \cup U_W'$, we have $U_W, U'_W \in \mathcal{V}$. In what follows, we will use $W$ to refer to the wall $(U_W, U'_W).$ Note that the cube complex dual to $\mathcal{V}$ is exactly the cube complex $C(G)$; indeed, $\Stab_G(H) = \Stab_G(\tilde{H})$ because edge stabilizers in $X_{bal}$ are trivial.

Let $1/2>r>0$. Recall that \(N_r(Y)\) denotes the $r$-neighbourhood of $Y$. Define
$$\mathcal U (Y)= \mathcal U_{r*} (Y) =\{U\in \mathcal V\mid diam (U\cap N_r(Y))=\infty\}.$$
Since each edge on $\EGbal$ has length $1$, by choosing $r<1/2$ we guarantee that the vertices of $N_r(Y)$ are exactly the vertices of $Y$. Thus, 
$$\mathcal U(Y)=\{U\in \mathcal V\mid diam (U\cap Y)=\infty\}.$$

\begin{remark}
    Let $W=(U,U')$ be a wall in $\mathcal V$. Suppose that both $U$ and $U'$ are in $\mathcal U (Y)$. Then there is a hyperplane $\mathfrak h$ of $Y$ such that $Y\setminus \mathfrak h = (U\cap Y) \cup (U'\cap Y)$. In other words,  $U\cap Y$ and $U'\cap Y$ are the halfspaces in $Y$ defined by $\mathfrak h$, and the hyperstructure that defines the wall $W$ extends $\mathfrak h$ to $\EGbal$.
\end{remark}

The set $\mathcal U (Y)$ is a \emph{hemiwallspace} in the sense of \cite[Definition 3.18]{hruska_wise_2014}. We denote the cube complex dual to this hemiwallspace by $C(\mathcal U)=C_{r*}(Y)$. This  cube complex embeds as a convex subcomplex in $C(\EGbal)=C(\EGbal,\mathcal V)$ \cite[Lemma 3.24]{hruska_wise_2014}, and is defined as follows \cite[Construction 3.21]{hruska_wise_2014}: the $0$-cubes are subcollections $c\subset \mathcal U$ with non-empty pairwise intersection such that for each $V=(U,U')\in \mathcal V$ exactly one of $U$ and $U'$ is in $c$. The $0$-cubes are connected by a $1$-cube if they differ on two complementary halfspaces $U$, $U'$ such that $(U,U')\in \mathcal V$. Finally, for every $k>0$, one adds a $k$-cube if the $(k-1)$-skeleton of a cube appears. 

\begin{remark}
    As the hyperstructure extending disjoint hyperplanes of $Y$ may intersect in $\EGbal$, the cube complex $Y$ is not isomorphic to $C_{r*}(Y)$. However, $Y$ is a subcomplex of $C_{r*}(Y)$. 
\end{remark}
The cube complex $C(\mathcal U)$ is isomorphic to the cube complex dual to the wallspace $\mathcal U'$ on $\EGbal$ defined by removing all those halfspaces from $\mathcal U$ such that only one halfspace has non-empty intersection with $Y$ \cite[Remark 3.22]{hruska_wise_2014}.  Note that in our context these are halfspaces whose corresponding $\EGbal$--hyperstructures are disjoint from $Y$. This cube complex is denoted by $C(\mathcal U')$. Note that $\mathcal U\setminus \mathcal U'$ is contained in every $0$-cube of $C(\mathcal U)$.  

\begin{remark}\label{R:cocompactness-P-equivariant}
   The cube complex $C(\mathcal U')$ embeds in $C(G)$ with image $C(\mathcal U)$, where the embedding is defined as follows: a $0$-cell $c'$ of $C(\mathcal U')$ is sent to the $0$-cell $c'\cup (\mc{U}\setminus \mc{U'})$ of $C(\mathcal U)$. This extends to the higher dimensional cubes of $C(\mathcal U ')$ in a natural way. The action of $P$ on $\EGbal$ induces an action of $P$ on $C(\mathcal U')$, or $C(\mathcal U)$, respectively. Also observe that $P$ stabilizes the set of walls in $\mathcal U'$ and the set of walls in $\mathcal U\setminus \mathcal U'$. By definition, the embedding of $C(\mathcal U')$ is $P$-equivariant. 
\end{remark}

\begin{proposition}\label{P:cocompactness}
    The group $P$ acts  cocompactly on $C(\mathcal U)$.
\end{proposition}
By \Cref{R:cocompactness-P-equivariant} it is sufficient to prove that $P$ acts cocompactly on $C(\mathcal U')$.
 We first prove the following. Recall that $Y$ is an essential CAT(0) cube complex, on which $P$ acts properly and cocompactly. 
 
 \begin{lemma}\label{L:cocompactness}
     Let $W_1$ and $W_2$ be two walls of $\mathcal U'$ that intersect in $\EGbal$. Then either they intersect in $Y$, or the intersection of the carriers of $p(W_1)$ and $p(W_2)$ contains an edge of $X_{bal}$ at $v_P$, the vertex of $X_{bal}$ fixed by $P$. 
 \end{lemma}
 
 \begin{proof}
     Suppose that $W_1$ and $W_2$ do not intersect in $Y$. 
     If $p(W_1)$ and $p(W_2)$ intersect in the interior of a $2$-cell containing $v_P$ we are done. 
     
     Otherwise, $p(W_1)$ and $p(W_2)$ collar a diagram of $X_{bal}$. Let $D$ be a minimal such diagram containing $v_P$. Recall that a corner of $D$ is a $2$-cell $c$ of $D$ such that $p(W_1)$ and $p(W_2)$ intersect in the interior of $c$. Note that a $2$-cell is a corner if and only if both $p(W_1)$ and $p(W_2)$ have non-empty intersection with its interior. Indeed, suppose $p(W_1)$ and $p(W_2)$ both have non-empty intersection with the interior of a $2$-cell $c$ of $X_{bal}$ and let $\tilde c$ be the pre-image of $c$ in $\EGbal$. Both $W_1$ and $W_2$ connect opposite edges of $\tilde c$. Thus, $W_1$ and $W_2$ intersect in the interior of $\tilde c$, hence, $p(W_1)$ and $p(W_2)$ intersect in the interior of~$c$. 
     
     As $p(W_1)$ and $p(W_2)$ do not intersect in the interior of a $2$-cell containing $v_P$, $D$ has at most one corner. 
     By \Cref{lem: markus trick} every $\varepsilon$-supershell of $D$ has to be a corner. Thus, by \Cref{lem: greendlinger} the diagram $D$ has at most two $2$-cells. 
     Both must contain $v_P$, as the interior of one contains a segment of $p(W_1)$ and the interior of the other contains a segment of $p(W_2)$. 
     Also $p(W_1)$ and $p(W_2)$ have to intersect in a point of the intersection of these two $2$-cells that is distinct from $v_P$. Hence, the intersection of the carriers of $p(W_1)$ and $p(W_2)$ contains an edge at $v_P$. 
 \end{proof}
 
 \begin{proof}[Proof of \Cref{P:cocompactness}]
 Let $W_1,W_2$ be walls in $\mathcal U'$ that intersect in $\EGbal$. Let $\mathfrak h_i$ be the hyperplane of $Y$ such that $W_i\cap Y=\mathfrak h_i$. 
 
 Recall the constant $\tau$ from \cref{R: bounded length in fiber}.  
 Let $R := 2\tau$. Observe that $R \geq 2\max\{||h||_Y\mid h\in B_P(m)\}$,
 where $||h||_Y$ denotes the translation length of $h$ on $Y$ and $B_P(m)$ is the ball of radius $m$ in the Cayley graph of $P$. 

We denote by $N_R(\mathfrak h)$ the cubical $R$-neighbourhood of $\mathfrak h$ and claim that $N_R(\mathfrak h_1)\cap N_R(\mathfrak h_2)$ is non-empty. Indeed, if $W_1$ and $W_2$ intersect in $Y$, then the intersection of $\mathfrak h_1$ and $\mathfrak h_2$ is non-empty. Otherwise, the carriers of the projected hypergraphs of $W_1$ and $W_2$ in $X_{bal}$  intersect in an edge of $X_{bal}$ at $v_P$. By construction of $\EGbal$, this implies that $d(\mathfrak h_1,\mathfrak h_2)\leqslant R$, hence, the claim.   

 Now let $W_1,W_2,\ldots, W_k$ be a set of walls of $\mathcal U'$ that pairwise intersect in $\EGbal$. Then $N_R(\mathfrak h_1), N_R(\mathfrak h_2),\ldots, N_R(\mathfrak h_k)$ pairwise intersect. Note that the spaces $N_R(\mathfrak h_i)$ are convex in $Y$. By Helly's theorem (see Lemma 13.13 of \cite{haglund_wise_special}, for instance) there is a point $y\in Y$ in the intersection of the sets $N_R(\mathfrak h_i)$. Thus, there is a radius $R'>R$ such that each of the hyperplanes $\mathfrak h_i$ intersects the finite ball $B_{R'}(y)$. Since $P$ acts cocompactly on $Y$, up to the group action of $P$ on $\mathcal U'$ there are at most finitely many families of walls of $\mathcal U'$ that pairwise intersect in $\EGbal$. 
 
 We conclude that $P$ acts cocompactly on $C(\mathcal U')$. See for instance \cite[Lemma 7.2]{hruska_wise_2014}. This finishes the proof.
 \end{proof}

\begin{proof}[Proof of \Cref{cor: actually cubulated}]
In view of \Cref{P: wall stabilizers qc in G}, the assumptions of \cite[Theorem 7.12]{hruska_wise_2014} hold in our situation. Thus  there is a compact $K\subset C(G)$ such that 
$$C(G)= GK\cup \bigcup_{i}GC_{r*}(Y_i)=GK\cup \bigcup_{i}GC(\mathcal U_i).$$
By \Cref{P:cocompactness} there is a compact $K_i\subset C_{r*}(Y_i)$ such that $G_iK_i=C_{r*}(Y_i)$. Let $K'=K \cup \bigcup_i K_i$. Note that $K'$ is the union of finitely many  compact subspaces, hence is itself compact.
Note that 
$$C(G)= GK\cup \bigcup_{i}GC_{r*}(Y_i)=GK\cup \bigcup_{i}GPK_i= GK'.$$

We conclude that the action of $G$ on $C(G)$ is cocompact. By \Cref{T: factors geometric implies proper}, the action is also proper.
\end{proof}

\printbibliography

@book{serre_arbres_1977,
 author = {Serre, Jean-Pierre},
 editor = {Bass, Hyman},
 title = {Arbres, amalgames, {{\(\mathrm{SL}_2\)}}. {R{\'e}dig{\'e}} avec la collaboration de {Hyman} {Bass}},
 fseries = {Ast{\'e}risque},
 series = {Ast{\'e}risque},
 issn = {0303-1179},
 volume = {46},
 year = {1977},
 publisher = {Soci{\'e}t{\'e} Math{\'e}matique de France (SMF), Paris},
 language = {French},
 zbMATH = {3575812},
 Zbl = {0369.20013}
}

@book{bridson_haefliger_1999,
 author = {Bridson, Martin R. and Haefliger, Andr{\'e}},
 title = {Metric spaces of non-positive curvature},
 fseries = {Grundlehren der Mathematischen Wissenschaften},
 series = {Grundlehren Math. Wiss.},
 issn = {0072-7830},
 volume = {319},
 isbn = {3-540-64324-9},
 year = {1999},
 publisher = {Berlin: Springer},
 language = {English},
 zbMATH = {1385418},
 Zbl = {0988.53001}
}

@article{Tsai-phase-transition,
  title={Phase transition for the existence of van Kampen 2--complexes in random groups},
  author={Tsai, Tsung-Hsuan},
  journal={Algebraic \& Geometric Topology},
  volume={24},
  number={7},
  pages={3897--3917},
  year={2024},
  publisher={Mathematical Sciences Publishers}
}

@article{EMN-boundary-criteria,
     author = {Einstein, Eduard and {M S}, Suraj Krishna and Ng, Thomas},
 title = {On the boundary criterion for relative cubulation: multiended parabolics},
 fjournal = {Bulletin of the London Mathematical Society},
 journal = {Bull. Lond. Math. Soc.},
 issn = {0024-6093},
 volume = {57},
 number = {7},
 pages = {2177--2189},
 year = {2025},
 language = {English},
 doi = {10.1112/blms.70087},
 zbMATH = {8073355},
 Zbl = {1575.20039}
}

@Article{hruska_wise_2014,
 Author = {Hruska, G. C. and Wise, Daniel T.},
 Title = {Finiteness properties of cubulated groups.},
 FJournal = {Compositio Mathematica},
 Journal = {Compos. Math.},
 ISSN = {0010-437X},
 Volume = {150},
 Number = {3},
 Pages = {453--506},
 Year = {2014},
 Language = {English},
 DOI = {10.1112/S0010437X13007112},
 zbMATH = {6293660},
 Zbl = {1335.20043}
}

@article {EG:RelGeom,
	AUTHOR = {Einstein, Eduard and Groves, Daniel},
	TITLE = {Relative cubulations and groups with a 2-sphere boundary},
	JOURNAL = {Compos. Math.},
	FJOURNAL = {Compositio Mathematica},
	VOLUME = {156},
	YEAR = {2020},
	NUMBER = {4},
	PAGES = {862--867},
	ISSN = {0010-437X},
	MRCLASS = {20F65},
	MRNUMBER = {4079630},
	MRREVIEWER = {Bruno P. Zimmermann},
	DOI = {10.1112/s0010437x20007095},
	URL = {https://doi-org.proxy.cc.uic.edu/10.1112/s0010437x20007095},
}

@Article{tsai_freiheit_2023,
 Author = {Tsai, Tsung-Hsuan},
 Title = {Freiheitssatz and phase transition for the density model of random groups},
 FJournal = {Mathematische Zeitschrift},
 Journal = {Math. Z.},
 ISSN = {0025-5874},
 Volume = {303},
 Number = {3},
 Pages = {25},
 Note = {Id/No 65},
 Year = {2023},
 Language = {English},
 DOI = {10.1007/s00209-022-03186-2},
 zbMATH = {7661317},
 Zbl = {1521.20067}
}

@book{ollivier_invitation_2005,
 author = {Ollivier, Yann},
 title = {A {January} 2005 invitation to random groups.},
 fseries = {Ensaios Matem{\'a}ticos},
 series = {Ensaios Mat.},
 issn = {0103-8141},
 volume = {10},
 isbn = {85-85818-30-1},
 year = {2005},
 publisher = {Rio de Janeiro: Sociedade Brasileira de Matem{\'a}tica},
 language = {English},
 zbMATH = {5012619},
 Zbl = {1163.20311}
}

@article {BowditchRH,
	AUTHOR = {Bowditch, Brian H.},
	TITLE = {Relatively hyperbolic groups},
	JOURNAL = {Internat. J. Algebra Comput.},
	FJOURNAL = {International Journal of Algebra and Computation},
	VOLUME = {22},
	YEAR = {2012},
	NUMBER = {3},
	PAGES = {1250016, 66},
	ISSN = {0218-1967},
	MRCLASS = {20F67 (20F65)},
	MRNUMBER = {2922380},
	MRREVIEWER = {R\'{e}mi Bernard Coulon},
	DOI = {10.1142/S0218196712500166},
	URL = {https://doi-org.proxy.cc.uic.edu/10.1142/S0218196712500166},
}

@article{papasoglu,
  title={An algorithm detecting hyperbolicity},
  author={Papasoglu, Panagiotis},
  journal={Geometric and computational perspectives on infinite groups (Minneapolis, MN and New Brunswick, NJ, 1994)},
  volume={25},
  pages={193--200},
  year={1996}
}

@incollection {gromov_asymptotic,
    AUTHOR = {Gromov, M.},
     TITLE = {Asymptotic invariants of infinite groups},
 BOOKTITLE = {Geometric group theory, {V}ol. 2 ({S}ussex, 1991)},
    SERIES = {London Math. Soc. Lecture Note Ser.},
    VOLUME = {182},
     PAGES = {1--295},
 PUBLISHER = {Cambridge Univ. Press, Cambridge},
      YEAR = {1993},
   MRCLASS = {20F32 (57M07)},
  MRNUMBER = {1253544},
}

@incollection {GromovHyperbolicGroups,
	AUTHOR = {Gromov, M.},
	TITLE = {Hyperbolic groups},
	BOOKTITLE = {Essays in group theory},
	SERIES = {Math. Sci. Res. Inst. Publ.},
	VOLUME = {8},
	PAGES = {75--263},
	PUBLISHER = {Springer, New York},
	YEAR = {1987},
	MRCLASS = {20F32 (20F06 20F10 22E40 53C20 57R75 58F17)},
	MRNUMBER = {919829},
	MRREVIEWER = {Christopher W. Stark},
	DOI = {10.1007/978-1-4613-9586-7_3},
	URL = {https://doi-org.proxy.library.cornell.edu/10.1007/978-1-4613-9586-7_3},
}

@article{montee_314,
AUTHOR = {Montee, MurphyKate},
     TITLE = {Random groups at density {$d<3/14$} act non-trivially on a
              {${\rm CAT}(0)$} cube complex},
   JOURNAL = {Trans. Amer. Math. Soc.},
  FJOURNAL = {Transactions of the American Mathematical Society},
    VOLUME = {376},
      YEAR = {2023},
    NUMBER = {3},
     PAGES = {1653--1682},
    }

@article{MP,
	author = "Mackay, John M. and Przytycki, Piotr",
	doi = "10.1307/mmj/1434731930",
	fjournal = "The Michigan Mathematical Journal",
	journal = "Michigan Math. J.",
	month = "06",
	number = "2",
	pages = "397--419",
	publisher = "University of Michigan, Department of Mathematics",
	title = "Balanced walls for random groups",
	url = "https://doi.org/10.1307/mmj/1434731930",
	volume = "64",
	year = "2015"
}

@misc{odr_nonplanar,
 author = {Tomasz Odrzyg{\'o}{\'z}d{\'z}},
 title = {Nonplanar isoperimetric inequality for random groups},
 year = {2021},
 howpublished = {Preprint, {arXiv}:2104.13903 [math.{GR}] (2021)},
 url = {https://arxiv.org/abs/2104.13903},
 arXiv = {arXiv:2104.13903}
}

@article{odr_bent_walls,
    author = {Odrzyg{\'o}{\'z}d{\'z}, Tomasz},
 title = {Bent walls for random groups in the square and hexagonal model},
 fjournal = {Transactions of the American Mathematical Society},
 journal = {Trans. Am. Math. Soc.},
 issn = {0002-9947},
 volume = {378},
 number = {6},
 pages = {3779--3822},
 year = {2025},
 language = {English},
 doi = {10.1090/tran/9141},
 zbMATH = {8048823},
 Zbl = {1567.20107}
}

@article{ollivier_wise,
    AUTHOR = {Ollivier, Yann and Wise, Daniel T.},
     TITLE = {Cubulating random groups at density less than {$1/6$}},
   JOURNAL = {Trans. Amer. Math. Soc.},
  FJOURNAL = {Transactions of the American Mathematical Society},
    VOLUME = {363},
      YEAR = {2011},
    NUMBER = {9},
     PAGES = {4701--4733},
      ISSN = {0002-9947},
   MRCLASS = {20F65 (20P05)},
  MRNUMBER = {2806688},
MRREVIEWER = {Fran\c{c}ois Dahmani},
       DOI = {10.1090/S0002-9947-2011-05197-4},
       URL = {https://doi.org/10.1090/S0002-9947-2011-05197-4},
}

@article{oll_somesmall,
author = {Ollivier, Yann},
year = {2004},
month = {10},
pages = {},
title = {Some small cancellation properties of random groups},
volume = {17},
journal = {International Journal of Algebra and Computation},
doi = {10.1142/S021819670700338X}
}

@article {ollivier_gafa,
    AUTHOR = {Ollivier, Y.},
     TITLE = {Sharp phase transition theorems for hyperbolicity of random
              groups},
   JOURNAL = {Geom. Funct. Anal.},
  FJOURNAL = {Geometric and Functional Analysis},
    VOLUME = {14},
      YEAR = {2004},
    NUMBER = {3},
     PAGES = {595--679},
      ISSN = {1016-443X},
   MRCLASS = {20F67 (20P05 37C35 60C05)},
  MRNUMBER = {2100673},
MRREVIEWER = {Goulnara N. Arzhantseva},
       DOI = {10.1007/s00039-004-0470-y},
       URL = {https://doi.org/10.1007/s00039-004-0470-y},
}

@Article{kasia_cubulating_2022,
 Author = {Jankiewicz, Kasia and Wise, Daniel},
 Title = {Cubulating small cancellation free products},
 FJournal = {Indiana University Mathematics Journal},
 Journal = {Indiana Univ. Math. J.},
 ISSN = {0022-2518},
 Volume = {71},
 Number = {4},
 Pages = {1397--1409},
 Year = {2022},
 Language = {English},
 DOI = {10.1512/iumj.2022.71.9628},
 zbMATH = {7619424}
}

@article {Martin:NPC_boundary,
	AUTHOR = {Martin, Alexandre},
	TITLE = {Non-positively curved complexes of groups and boundaries},
	JOURNAL = {Geom. Topol.},
	FJOURNAL = {Geometry \& Topology},
	VOLUME = {18},
	YEAR = {2014},
	NUMBER = {1},
	PAGES = {31--102},
	ISSN = {1465-3060},
	MRCLASS = {20F65 (20F67 20F69)},
	MRNUMBER = {3158772},
	MRREVIEWER = {Xiangdong Xie},
	DOI = {10.2140/gt.2014.18.31},
	URL = {https://doi.org/10.2140/gt.2014.18.31},
}

@Article{einstein_relatively_2022,
 Author = {Einstein, Eduard and Groves, Daniel},
 Title = {Relatively geometric actions on {{\(\operatorname{CAT} (0)\)}} cube complexes},
 FJournal = {Journal of the London Mathematical Society. Second Series},
 Journal = {J. Lond. Math. Soc., II. Ser.},
 ISSN = {0024-6107},
 Volume = {105},
 Number = {1},
 Pages = {691--708},
 Year = {2022},
 Language = {English},
 DOI = {10.1112/jlms.12556},
 zbMATH = {7730390},
 Zbl = {1521.20096}
}

@misc{EinsteinNg,
 author = {Eduard Einstein and Thomas Ng},
 title = {Relative {Cubulation} of {Small} {Cancellation} {Free} {Products}},
 year = {2021},
 howpublished = {Preprint, {arXiv}:2111.03008 [math.{GR}] (2021)},
 url = {https://arxiv.org/abs/2111.03008},
 arXiv = {arXiv:2111.03008}
}

@article {GMSpecializing,
    AUTHOR = {Groves, Daniel and Manning, Jason Fox},
     TITLE = {Specializing cubulated relatively hyperbolic groups},
   JOURNAL = {J. Topol.},
  FJOURNAL = {Journal of Topology},
    VOLUME = {15},
      YEAR = {2022},
    NUMBER = {2},
     PAGES = {398--442},
      ISSN = {1753-8416,1753-8424},
   MRCLASS = {20F67 (57M60)},
  MRNUMBER = {4413505},
MRREVIEWER = {Inhyeok\ Choi},
       DOI = {10.1112/topo.12226},
       URL = {https://doi.org/10.1112/topo.12226},
}

@article {martin_steenbock,
    AUTHOR = {Martin, Alexandre and Steenbock, Markus},
     TITLE = {A combination theorem for cubulation in small cancellation
              theory over free products},
   JOURNAL = {Ann. Inst. Fourier (Grenoble)},
  FJOURNAL = {Universit\'{e} de Grenoble. Annales de l'Institut Fourier},
    VOLUME = {67},
      YEAR = {2017},
    NUMBER = {4},
     PAGES = {1613--1670},
      ISSN = {0373-0956},
   MRCLASS = {20F65 (20F06 20F67)},
  MRNUMBER = {3711135},
MRREVIEWER = {Kasia Jankiewicz},
       URL = {http://aif.cedram.org/item?id=AIF_2017__67_4_1613_0},
}

@article{futer_wise,
     Author = {Futer, David and Wise, Daniel T.},
 Title = {Cubulating random quotients of hyperbolic cubulated groups},
 FJournal = {Transactions of the American Mathematical Society. Series B},
 Journal = {Trans. Am. Math. Soc., Ser. B},
 ISSN = {2330-0000},
 Volume = {11},
 Pages = {622--666},
 Year = {2024},
 Language = {English},
 DOI = {10.1090/btran/180},
 zbMATH = {7814405},
 Zbl = {1535.20226}
}

@inproceedings {scott_wall,
    AUTHOR = {Scott, Peter and Wall, Terry},
     TITLE = {Topological methods in group theory},
 BOOKTITLE = {Homological group theory ({P}roc. {S}ympos., {D}urham, 1977)},
    SERIES = {London Math. Soc. Lecture Note Ser.},
    VOLUME = {36},
     PAGES = {137--203},
 PUBLISHER = {Cambridge Univ. Press, Cambridge-New York},
      YEAR = {1979},
   MRCLASS = {57M05 (20E06)},
  MRNUMBER = {564422},
MRREVIEWER = {W. Jaco},
}

@article {WiseSmallCancellation,
    AUTHOR = {Wise, D. T.},
     TITLE = {Cubulating small cancellation groups},
   JOURNAL = {Geom. Funct. Anal.},
  FJOURNAL = {Geometric and Functional Analysis},
    VOLUME = {14},
      YEAR = {2004},
    NUMBER = {1},
     PAGES = {150--214},
      ISSN = {1016-443X,1420-8970},
   MRCLASS = {20F65 (20F06)},
  MRNUMBER = {2053602},
MRREVIEWER = {Martin\ Edjvet},
       DOI = {10.1007/s00039-004-0454-y},
       URL = {https://doi.org/10.1007/s00039-004-0454-y},
}

@article {Hruska2010,
	AUTHOR = {Hruska, G. Christopher},
	TITLE = {Relative hyperbolicity and relative quasiconvexity for
	countable groups},
	JOURNAL = {Algebr. Geom. Topol.},
	FJOURNAL = {Algebraic \& Geometric Topology},
	VOLUME = {10},
	YEAR = {2010},
	NUMBER = {3},
	PAGES = {1807--1856},
	ISSN = {1472-2747},
	MRCLASS = {20F65 (20F67)},
	MRNUMBER = {2684983},
	MRREVIEWER = {Eduardo Mart{\'{\i}}nez-Pedroza},
	DOI = {10.2140/agt.2010.10.1807},
	URL = {http://dx.doi.org/10.2140/agt.2010.10.1807},
}

@article {BergeronWise,
	AUTHOR = {Bergeron, Nicolas and Wise, Daniel T.},
	TITLE = {A boundary criterion for cubulation},
	JOURNAL = {Amer. J. Math.},
	FJOURNAL = {American Journal of Mathematics},
	VOLUME = {134},
	YEAR = {2012},
	NUMBER = {3},
	PAGES = {843--859},
	ISSN = {0002-9327},
	MRCLASS = {20F67},
	MRNUMBER = {2931226},
	DOI = {10.1353/ajm.2012.0020},
	URL = {https://doi-org.proxy.library.cornell.edu/10.1353/ajm.2012.0020},
}

@article {Tsai_density_2022,
    AUTHOR = {Tsai, Tsung-Hsuan},
     TITLE = {Density of random subsets and applications to group theory},
   JOURNAL = {J. Comb. Algebra},
  FJOURNAL = {Journal of Combinatorial Algebra},
    VOLUME = {6},
      YEAR = {2022},
    NUMBER = {3-4},
     PAGES = {223--263},
      ISSN = {2415-6302,2415-6310},
   MRCLASS = {20P05 (20F05 20F06 60C05)},
  MRNUMBER = {4514439},
MRREVIEWER = {Alla\ S.\ Detinko},
       DOI = {10.4171/jca/63},
       URL = {https://doi.org/10.4171/jca/63},
}

@Article{Zuk2003,
	author    = {Andrzej \.{Z}uk},
	title     = {Property ({T}) and {K}azhdan constants for discrete groups},
	journal   = {Geometric And Functional Analysis},
	year      = {2003},
	volume    = {13},
	number    = {3},
	pages     = {643--670},
	month     = {6},
	doi       = {10.1007/s00039-003-0425-8},
	publisher = {Springer Science and Business Media {LLC}},
}

@Article{KK,
	author    = {Marcin Kotowski and Micha{\l} Kotowski},
	title     = {Random groups and property (T ): {\.{Z}}uk{'}s theorem revisited},
	journal   = {Journal of the London Mathematical Society},
	year      = {2013},
	volume    = {88},
	number    = {2},
	pages     = {396--416},
	month     = {8},
	doi       = {10.1112/jlms/jdt024},
	publisher = {Wiley},
}

@misc{Ashcroft_1/4,
 author = {Calum J Ashcroft},
 title = {Random groups do not have {Property} ({T}) at densities below 1/4},
 year = {2022},
 howpublished = {Preprint, {arXiv}:2206.14616 [math.{GR}] (2022)},
 url = {https://arxiv.org/abs/2206.14616},
 arXiv = {arXiv:2206.14616}
}

@article {caprace_sageev,
    AUTHOR = {Caprace, Pierre-Emmanuel and Sageev, Michah},
     TITLE = {Rank rigidity for {CAT}(0) cube complexes},
   JOURNAL = {Geom. Funct. Anal.},
  FJOURNAL = {Geometric and Functional Analysis},
    VOLUME = {21},
      YEAR = {2011},
    NUMBER = {4},
     PAGES = {851--891},
      ISSN = {1016-443X,1420-8970},
   MRCLASS = {20F65 (20F67 53C24)},
  MRNUMBER = {2827012},
MRREVIEWER = {Tetsu\ Toyoda},
       DOI = {10.1007/s00039-011-0126-7},
       URL = {https://doi.org/10.1007/s00039-011-0126-7},
}

@article {haglund_wise_special,
    AUTHOR = {Haglund, Fr\'ed\'eric and Wise, Daniel T.},
     TITLE = {Special cube complexes},
   JOURNAL = {Geom. Funct. Anal.},
  FJOURNAL = {Geometric and Functional Analysis},
    VOLUME = {17},
      YEAR = {2008},
    NUMBER = {5},
     PAGES = {1551--1620},
      ISSN = {1016-443X},
   MRCLASS = {20F36 (20F55 20F67)},
  MRNUMBER = {2377497},
MRREVIEWER = {Patrick Bahls},
       DOI = {10.1007/s00039-007-0629-4},
       URL = {http://dx.doi.org/10.1007/s00039-007-0629-4},
}
\end{document}